\documentclass[lms]{amsart}
 \usepackage[]{amsmath, amsthm, amsfonts, amssymb, amscd,mathtools}
\usepackage[mathscr]{eucal}
\usepackage{hyperref}
\usepackage[all]{xy}
\usepackage[usenames,dvipsnames]{color}
\usepackage{todonotes}\usepackage{rotating}
\usepackage[shortlabels]{enumitem}
\usepackage{tikz}

\setlist[enumerate]{label=\it{(\roman*)},
  ref=\it{(\roman*)}}

\usepackage{dsfont}

\usepackage{blindtext}
\usepackage{subfiles}

\usepackage{stackengine}
\stackMath

\makeatletter
\@namedef{subjclassname@2020}{%
  \textup{2020} Mathematics Subject Classification}
\makeatother

\newcommand\fancyconj[2]{\stackon[-1.5pt]{#1}{\rule[2pt]{\widthof{$#1$}}{.4pt}_{{\text{\tiny
          #2}}}}}
\newcommand{\Betti}{\text{\rm B}}
\newcommand{\dR}{\text{\rm dR}}
\newcommand\betticonj[1]{\fancyconj{#1}{\text{\tiny \Betti}}}

\newcommand*\circled[1]{{ \tikz[baseline=(char.base)]{
            \node[shape=circle,draw,inner sep=.7pt] (char)
            {#1};}}}

\newcommand{\A}{\mathds{A}}

\newcommand{\C}{\mathds{C}}
\newcommand{\D}{\mathds{D}}

\newcommand{\G}{\mathds{G}}
\renewcommand{\H}{\mathds{H}}

\renewcommand{\P}{\mathds{P}}
\newcommand{\Q}{\mathds{Q}}
\newcommand{\R}{\mathds{R}}

\newcommand{\Z}{\mathds{Z}}

\newcommand{\caB}{\mathcal{B}}

\newcommand{\caI}{\mathcal{I}}

\newcommand{\caO}{\mathcal{O}}
\newcommand{\caP}{\mathcal{P}}

\newcommand{\caR}{\mathcal{R}}
\newcommand{\caS}{\mathcal{S}}

\newcommand{\caV}{\mathcal{V}}
\newcommand{\caW}{\mathcal{W}}
\newcommand{\caX}{\mathcal{X}}

\newcommand{\caZ}{\mathcal{Z}}

\newcommand{\fD}{\mathfrak{D}}

\newcommand{\bfone}{{\mathds{1}}}
\newcommand{\fg}{{\mathfrak{g}}}

  \newcommand {\p} {{\prime}}

   \DeclareMathOperator {\Gr} {Gr}

 \DeclareMathOperator {\Div}{div}

 \DeclareMathOperator{\codim}{codim}
 \DeclareMathOperator{\im}{Im}
 \DeclareMathOperator {\cl}{cl}
 \DeclareMathOperator {\CH}{CH} 
  
\DeclareMathOperator{\Spec}{Spec}
\DeclareMathOperator{\Aut}{Aut}
\DeclareMathOperator{\Ext}{Ext}
\DeclareMathOperator{\Hom}{Hom}
\DeclareMathOperator{\Id}{Id}

\DeclareMathOperator{\fDTW}{\fD_{TW}}

\DeclareMathOperator{\DB}{\fD}
\DeclareMathOperator{\vol}{Vol}
\DeclareMathOperator{\GL}{GL}
\DeclareMathOperator{\gl}{\mathfrak{gl}}

\DeclareMathOperator{\lsl}{\mathfrak{sl}}
\DeclareMathOperator{\Ad}{Ad}
\DeclareMathOperator{\ad}{ad}
\DeclareMathOperator{\Ht}{ht}
\DeclareMathOperator{\RE}{Re}
\DeclareMathOperator{\Li}{Li}

\DeclareMathOperator{\simple}{s}

\newcommand{\WF}{\text{\rm WF}}

\newcommand{\Arch}{\text{\rm Arch}}
\newcommand{\geom}{\text{\rm geom}}
\newcommand{\Hodge}{\text{\rm Hodge}}
\newcommand{\TW}{\text{\rm TW}}

\newcommand{\sm}{\text{\rm sm}}

\newcommand{\I}{\text{Im}}
\newcommand{\ol}{\overline}

\newcommand{\MHS}{\mathbf{MHS}}

\newcommand{\old}{\text{\rm old}}
\newcommand{\new}{\text{\rm new}}
\newcommand{\fin}{\text{\rm fin}}

\newcommand{\sing}{\text{\rm sing}}
\newcommand{\Zar}{\text{\rm Zar}}

\newcommand{\dsquare}{\settowidth{\dimen0}{$\square$}\square\hspace*{-\dimen0}\,\square}

\numberwithin{equation}{section}

\theoremstyle{plain}

\newtheorem{prop}{Proposition}[section]

\newtheorem{cor}[prop]{Corollary}
\newtheorem{lem}[prop]{Lemma}
\newtheorem{thm}[prop]{Theorem}
\newtheorem{theoalph}{Theorem}

\theoremstyle{definition}
\newtheorem{df}[prop]{Definition}

\newtheorem{notation}[prop]{Notation}
\newtheorem{assumption}[prop]{Assumption}
\newtheorem{defalph}{Definition}

\theoremstyle{remark}
\newtheorem{rmk}[prop]{Remark}
\newtheorem{ex}[prop]{Example}

\begin{document}

\renewcommand\stackalignment{l}

\title{Height Pairing on higher cycles and mixed Hodge structures}

\author{J.~I. Burgos~Gil}

\address{Instituto de Ciencias Matem\'aticas (CSIC-UAM-UCM-UCM3).
  Calle Nicol\'as Ca\-bre\-ra~15, Campus UAM, Cantoblanco, 28049 Madrid,
  Spain} 
\email{burgos@icmat.es}

\author{S.~Goswami}

\address{Office 230, Fakult\"at f\"ur Mathematik\\
Universit\"at Regensburg\\
93040, Regensburg, Germany}
\email{souvik.goswami@mathematik.uni-regensburg.de}

\author{G.~Pearlstein}

\address{Mathematics Department\\
Universit\`a Di Pisa\\
Largo Bruno Pontecorvo, 5, 56127 Pisa PI, Italy}
\email{greg.pearlstein@unipi.it}

\date{\today}

\date{\today}
\newif\ifprivate
\privatetrue
\thanks{Burgos Gil was partially supported by the
  MINECO research project PID2019-108936GB-C21 and
  CEX2019-000904-S (ICMAT  Severo Ochoa). Goswami is
  currently supported by the CRC 1085 Higher Invariants (Universit\"at
  Regensburg, funded by DFB). Pearlstein and Goswami were partially supported
  by grant DMS 
  1361120. The three authors were supported by EPRSC grant
  EP/R014604/1  during their stay at the Newton Institute.}

\begin{abstract}
For a smooth, projective complex variety, we introduce several
mixed Hodge structures associated to higher algebraic cycles. Most
notably, we introduce a mixed Hodge structure for a pair of higher
cycles which are in the refined normalized complex and intersect
properly. In a special case, this mixed Hodge structure is an oriented
biextension, and its height agrees with the higher archimedean height
pairing introduced in a previous paper by the first two authors. We
also compute a non-trivial example of this height given by
Bloch-Wigner dilogarithm function. Finally we study the variation of
mixed Hodge structures of Hodge-Tate type, and show that the height 
extends continuously to degenerate situations.
\end{abstract}

\maketitle

\tableofcontents{}

\section*{Introduction}
\label{sec:introduction}
\subsection*{Main objectives}\label{sec:main-onjectives}
Let $X$ be a smooth projective variety of dimension $d$ defined over a number
field $F$. The height pairing between cycles is an arithmetic analogue of
the intersection product and can be seen as a
linking number. It plays a central role  in arithmetic geometry.

Arakelov theory and concretely arithmetic intersection
theory \cite{GilletSoule:ait}
provides a general framework to define and study the height pairing,
exploiting the analogy with the intersection product.
Let $Z $ and $W$ be disjoint, homologically trivial algebraic cycles on $X$ of
codimension $p$ and $q=d+1-p$ respectively.

Assume that there is a regular
model $\caX$  of $X$ over $\caO_{F}$, the ring of integers of
$F$, and that the cycles $Z$ and $W$ can be extended to cycles $\caZ$
and $\caW$ on $\caX$, whose intersection with any vertical cycle is
zero.  Then we can choose liftings $\widehat Z=(\caZ,g_{Z})$ and
$\widehat W=(\caW,g_{W})$ of $Z$ and $W$
in the arithmetic Chow groups $\widehat \CH^{p}(\caX)$ and $\widehat
\CH^{q}(\caX)$ respectively, satisfying the additional condition
\begin{displaymath}
  dd^{c}g_{Z}+\delta _{Z}=dd^{c}g_{W}+\delta _{W}=0
\end{displaymath}
In this setting the height pairing is given by
\begin{displaymath}
  \langle Z, W\rangle_{\Ht}=\widehat{\deg}(\widehat Z\cdot \widehat W)
\end{displaymath}
and is independent of the choice of liftings. 
This height pairing can be written as a sum of
components
\begin{displaymath}
  \langle Z, W\rangle_{\Ht}=\langle Z,W\rangle_{\fin}+\langle Z,W\rangle_{\Arch}\in \R,
\end{displaymath}
where $\langle Z,W\rangle_{\fin}$ is the \emph{finite contribution}
that is defined using intersection theory on the model $\caX$, while $\langle
Z,W\rangle_{\Arch}$ is the archimedean height pairing and is
computed using the Green currents in the complex manifold associated to
$X$.
\begin{displaymath}
  \langle Z,W\rangle_{\Arch}=\int_{X}g_{Z}\wedge \delta
  _{W}=\int_{X}g_{W}\wedge \delta _{Z}.
\end{displaymath}

Note that, even if $\langle Z, W\rangle_{\Ht}$ depends only on
the rational equivalence class of $Z$ and $W$, the finite and
archimedean components depend of the actual cycles $Z$ and $W$. 

In the paper \cite{Hain:Height} R. Hain has given a Hodge theoretical
interpretation of the archimedean height
pairing. Namely, to the pair of cycles $Z$ and $W$ one can associate
a \emph{biextension} $B_{Z,W}$ of mixed Hodge structure. The isomorphism classes of
biextension mixed Hodge structures are classified by a single real
invariant and the archimedean height pairing agrees with this
invariant. In fact, not only the archimedean component can be
interpreted as the class of an extension, but also other local
components of the height pairing can be obtained as extension classes
of motivic origin.  See for instance
\cite{Scholl:MR1265545} and \cite{Scholl:MR1110402}

S. Bloch has introduced the higher Chow groups $\CH^p(X,n)$ in
\cite{Bloch:achK} as a concrete example of motivic cohomology theory.
Subsequently,  in \cite{BurgosFeliu:hacg} E. Feliu and the first author have
introduced the higher arithmetic Chow groups. These groups have been
further studied by the first and second author in
\cite{Burgoswami:hait}. Moreover, they have introduced a height
pairing between higher cycles 
whose real regulators are zero.  Although there are
many differences between the case of algebraic cycles and the case of
higher cycles, the height pairing between higher cycles still
decomposes as a sum
\begin{equation}\label{eq:65}
  \langle Z, W\rangle_{\Ht}=\langle Z,W\rangle_{\geom}+\langle
  Z,W\rangle_{\Arch},
\end{equation}
of an \emph{archimedean contribution}, that will
be called the archimedean higher height pairing, and a
\emph{geometric contribution} that, although is very different in
nature to the finite contribution in the case of ordinary cycles, is
also related to an intersection product.

The archimedean higher height pairing depends only on the complex
manifold associated to $X$ and can be defined for higher cycles on a
smooth projective complex variety. 
The aim of the present paper is to generalize Hain's result and give a
Hodge theoretical 
interpretation of the archimedean higher height
pairing between certain higher cycles. 
More precisely, as we review below \eqref{eq:41}, $\CH^*(X,*)$ can be
computed as the homology of a complex $(Z^*(X,*)_{00},\delta)$.  The main
result of this paper can be compiled in the following theorem:

\begin{theoalph}
Let $X$ be a smooth complex projective variety of dimension $d$ and
$Z\in Z^p(X,1)_{00}$ and $W\in Z^q(X,1)_{00}$ be elements which
satisfy the following conditions: 
\begin{enumerate}
\item \label{item:11} $p+q=d+2$,
\item \label{item:12} $\delta Z=\delta W=0$,
\item \label{item:13} $Z$ and $W$ intersect properly,
\item \label{item:14} the intersection of $Z$ and $W$ also satisfies Assumption \ref{def:7}. 
\end{enumerate}
Then, in analogy with Hain's construction, there is a canonical mixed
Hodge structure $B_{Z,W}$ attached to $Z$ and $W$ from which one can extract
a Hodge theoretical height pairing $\langle
Z,W\rangle_{\Hodge}$. Moreover (see Theorem \ref{compareheight}), if $Z$ and $W$ both have real
regulator zero then
\begin{displaymath}
  \langle Z,W\rangle_{\Hodge} = \langle Z,W\rangle_{\Arch}.
\end{displaymath}
\end{theoalph}

Regarding condition \ref{item:11}, much of our analysis carries
through the case where $Z \in Z^p(X,n)_{00}$ and $W \in Z^q(X,m)_{00}$
provided that $2(p+q-d-1) = m+n$.  However, condition \ref{item:13}  allows for
non-trivial intersections of $Z$ and $W$ which contribute to the mixed
Hodge structure $B_{Z,W}$.  In the case $m=n=1$, this intersection is just
a finite set of points and is easy to handle provided we assume some
extra technical conditions that are satisfied generically (see
Assumption \ref{def:7}).

At first glance, the contribution from the intersection of $Z$ and $W$
might appear to be just a technical issue arising during the
construction of $B_{Z,W}$.  However, on reflection, it is exactly this
new contribution which allows $B_{Z,W}$ to have interesting deformations
which satisfy Griffiths horizontality.

The asymptotic behavior of the archimedean component of the height
pairing has been extensively studied by the third author in
\cite{pearlstein:sl2} using the Hodge theoretical
interpretation. Moreover, in collaboration with P. Brosnan, in
\cite{BP:jumps} he has
given an explanation of the height jump phenomenon. The asymptotic
behavior of the height and the height jump phenomenon has also been
studied by the first author in collaboration with R. de Jong and
D. Holmes in \cite{BHdJ}.

A second objective of this paper is to use the Hodge theoretical
interpretation of the archimedean higher height pairing to start the
study of its asymptotic behavior.
In section \ref{sec:an-example-dimension-2} we study an example in
dimension $2$ in which $n=m=1$ and the cohomology of $X$ is of
Hodge--Tate type and we observe that the height can be extended
continuously to the degenerate situations. This is in sharp contrast
with the usual height pairing that has logarithmic singularities when
approaching degenerate situations. We show that this is a general
phenomenon of higher heights for Hodge--Tate variations of mixed Hodge
structures (Theorem \ref{asymptotic-thm}).  

\begin{theoalph}
Let $S$ be a Zariski open subset of
a complex manifold $\bar S$ such that $D=\bar S-S$ is a normal crossing
divisor.  Let $\mathcal V\to S$ be an oriented graded polarized
Hodge--Tate variation with length $\ell(\mathcal V)\ge
4$. Assume $\mathcal V$ is admissible with respect to $\bar S$ and has unipotent
local monodromy about $D$. Let $p\in D$. Then
 the limit mixed 
Hodge structure $\mathcal V_p$ of $\mathcal V$ at $p\in D$ is an oriented
Hodge--Tate structure with the same weight filtration as $\mathcal V$.
Moreover,
\begin{displaymath}
  \lim_{s\to p} \Ht(\mathcal V_s)= \Ht(\mathcal V_p).
\end{displaymath}
\end{theoalph}

In this result oriented means that the top and the bottom graded
pieces are constant variations of rank one, and $\Ht(\mathcal V_s)$
denotes the height of the oriented mixed Hodge structure $\mathcal
V_s$ (Definition \ref{signed-height}). The important hypotheses are
first, that the length
$\ell(\mathcal V)\ge
4$ that is, the difference between the minimal and maximal weight is
at least 4 (hence we are dealing with a higher height) and second
that the whole variation is of Hodge--Tate type. In Example
\ref{exam:6iii} we show that this last hypothesis is necessary. 

\subsection*{Background for usual cycles}\label{sec0-2}
Before giving a more precise statement of the main results of the
paper we briefly recall the case of ordinary cycles.\\

Assuming several conjectures,  Beilinson
\cite{Beilinson:hp} has defined a height pairing between the Chow group of
cycles homologous to zero.
\begin{displaymath}
  \langle\ ,\ \rangle_{HT}\colon
  \CH^p(X)^0\otimes \CH^{d-p+1}(X)^0\rightarrow \R, 
\end{displaymath}
 where $\CH^p(X)^0$ indicates the subgroup of  $\CH^p(X)$ consisting of
 cycles homologous to zero. This is the same thing as the kernel of
 the cycle class map to real Deligne cohomology
 \begin{displaymath}
   \CH^p(X)^0=\ker\left(\cl_p\colon \CH^p(X)\rightarrow H^{2p}_{\fD}(X, \R(p))\right).
 \end{displaymath}
Up to certain assumptions on $X$, which are true for certain class of
examples like curves and abelian varieties, Beilinson's height pairing can be constructed using Gillet and
Soul\'e's arithmetic intersection theory (see \cite{Kunnemann:hpvhp}
for more details). More concretely, writing $S=\Spec(\caO_{F})$, we have to make the following
assumptions on $X$. 
\begin{enumerate}
\item [{\bf A1}]\label{item:34} There exists a regular scheme $\caX$,
  flat and projective 
  over $S$, such that $X=\caX\times \Spec(F)$.
\item [{\bf A2}] \label{item:35} Every cycle $x\in \CH^{p}(X)^{0}_{\Q}$ can be lifted to a
  cycle $\ol x\in \CH^{p}(\caX)_{\Q}$ such that $\ol x\cdot Y=0$ for every
  cycle $Y\in Z^{d+1-p}(\caX)_{\fin}$. Here  $Z^{d+1-p}(\caX)_{\fin}$
  is the group of cycles whose support is contained in a finite number
  of fibers of the structural map $\caX\to S$.  
\end{enumerate}
Then, under the assumptions {\bf A1} and {\bf A2} we can construct
Beilinson's height pairing after tensoring with $\Q$ using 
arithmetic intersection on $\caX$. We give a very succinct description of the pairing below.\\

Arithmetic Chow groups \cite{GilletSoule:ait} comes equipped with an
intersection product
\begin{displaymath}
  \widehat{\CH}^p(\caX)_{\Q}\otimes
  \widehat{\CH}^{d-p+1}(\caX)_{\Q}\rightarrow
  \widehat{\CH}^{d+1}(\caX)_{\Q}, 
\end{displaymath}
push-forward maps
\begin{displaymath}
  \widehat{\CH}^{d+1}(\caX)\to \widehat{\CH}^{1}(S)
  \to \widehat{\CH}^{1}(\Spec(\Z))
\end{displaymath}
and an isomorphism
\begin{displaymath}
  \widehat{\CH}^{1}(\Spec(\Z))\simeq \R.
\end{displaymath}
Combining the push-forward and the above isomorphism we obtain an
arithmetic degree map
\begin{displaymath}
  \widehat{\deg}\colon \widehat{\CH}^{d+1}(\mathcal{X})\rightarrow \R.
\end{displaymath}
Composing the intersection product with the arithmetic degree we
obtain a pairing
\begin{equation}
  \label{eq:96}
 (\ ,\ )_{\caX}\colon \widehat{\CH}^p(\caX)_{\Q}\otimes
 \widehat{\CH}^{d-p+1}(\caX)_{\Q}\rightarrow
 \widehat{\CH}^{d+1}(\caX)_{\Q}\xrightarrow{\widehat{\deg}}\R.   
\end{equation}
 
 Now let $\widehat{\CH}^p(\mathcal{X})^0$ be the subgroup
 $\widehat{\CH}^p(\mathcal{X})$,  generated by arithmetic cycles
 $(Z,g_{Z})$ such that $dd^{c}g_Z+\delta _{Z}=0$ and $Z\cdot Y=0$ for every
 $Y\in Z^{d+1-p}(\caX)_{\fin}$. This implies in particular that the
 restriction of $Z$ to the generic fiber $X$ is homologous to zero.

The assumption {\bf A2} implies that the map
$\widehat{\CH}^{p}(\caX)\to \CH^{p}(X)$ induces a surjective map
\begin{displaymath}
  \widehat{\CH}^{p}(\caX)_{\Q}^{0}\twoheadrightarrow \CH^{p}(X)_{\Q}^{0}. 
\end{displaymath}
Finally, for elements $x_1\in \CH^p(X)^0_{\Q}$ and $x_2\in
\CH^{d-p+1}(X)^0_{\Q}$, Beilinson's height pairing is defined as
follows: Lift $x_1$ to $\tilde {x}_1\in
\widehat{\CH}^p(\mathcal{X})^0_{\Q}$ and $x_2$ to $\tilde{x}_2\in
\widehat{\CH}^{d-p+1}(\mathcal{X})^0_{\Q}$ and define
\begin{displaymath}
  \langle x_1,x_2\rangle_{HT}\coloneqq (\tilde{x}_1,\tilde{x}_2)_{\mathcal{X}}.
\end{displaymath}
One can easily show that the right hand side does not depend on the
lifting (see \S5 of \cite{Kunnemann:hpvhp}).

This height pairing is an important tool, and has a number of
conjectural properties which are linked to the Beilinson's conjectures
(see \S 5 of \cite{Beilinson:hp} for further details).

Beilinson's height pairing can be decomposed in a sum of local
contributions. One for each place of $\Q$. The sum of the finite
contributions can be grouped together in an intersection theoretical
contribution, while the archimedean contribution has a Hodge
theoretical interpretation. Let $x_{1}$ and $x_{2}$ as before and
choose representatives $Z\in Z^{p}(X)$ and $W\in Z^{d+1-p}(X)$
of $x_{1}$ and $x_{2}$ respectively that intersect properly. By the
codimensions of $Z$ and $W$ proper intersection means in this
case that they do not meet. Lift $Z$ and $W$ to cycles $\caZ$ and
$\caW$ satisfying the condition in assumption {\bf A2}, and choose
Green currents $g_{Z}$ and $g_{W}$ whose associated forms are zero. Then
\begin{displaymath}
  \langle x_{1},x_{2}\rangle_{HT} = \langle
  Z,W\rangle_{\fin}+
  \langle Z,W\rangle_{\Arch},
\end{displaymath}
where
\begin{align*}
  \langle  Z,W\rangle_{\fin}
  &= \deg (\caZ\cdot \caW),\\
  \langle Z,W\rangle_{\Arch}
  &= \widehat{\deg}(g_{Z}\ast g_{W})=\int_{X}\delta_Z\wedge g_{W}\in \R.
\end{align*}
It is important to remark that, while the height pairing $\langle
x_{1},x_{2}\rangle_{HT}$ depends only on the classes $x_{1}$ and
$x_{2}$, the decomposition in finite and archimedean components depends
on the choice of cycles $Z$ and $W$ representing these classes.\\

We now discuss R. Hain's Hodge theoretic interpretation of $\langle Z,W\rangle_{\Arch}$ (see \cite{Hain:Height} for details). Let $H$ be a torsion free integral pure
Hodge structure of weight $-1$. A \emph{biextension} $B$ associated
to $H$ is a mixed Hodge structure of non-zero weights $-2,-1,0$, with
the graded pieces satisfying 
\begin{displaymath}
\begin{gathered}
\Gr^W_0B= \Z(0),\\
\Gr^W_{-1}B= H,\\
\Gr^W_{-2}B= \Z(1).
\end{gathered}
\end{displaymath}
Let $\caB(H)$ denote the set of isomorphism classes of biextensions as
before and $\caB(H)_{\R}$ the isomorphism classes of real mixed Hodge
structures of the same shape. 
The following results are proved in \cite{Hain:Height} (Corollary
3.1.6, Corollary 3.2.2 and Corollary 3.2.9)
\begin{enumerate}
\item
  $\Ext^1_{\MHS}(\Z(0), H)$ and $\Ext^1_{\MHS}(H, \Z(1))$ are dual tori.
\item The projection
\begin{displaymath}
\caB(H)\rightarrow \Ext^1_{\MHS}(\Z(0), H)\times \Ext^1_{\MHS}(H, \Z(1))
\end{displaymath}
given by $B\mapsto (B/W_{-2}, W_{-1})$ has the structure of a
principal $\C^\ast$ bundle.
\item   $\Ext^1_{\R-\MHS}(\R(0), H_{\R})=\Ext^1_{\R-\MHS}(H_{\R},
  \R(1))=0$.
\item There is a canonical bijection $\caB_{\R}(H) \xrightarrow{\cong}
  \R$.
\end{enumerate}

In particular if $Z\in Z^p_{\hom}(X)$ and $W\in Z^q_{\hom}(X)$ are two
cycles homologous to zero, intersecting properly with $p+q=d+1$, then
the Abel-Jacobi images of $Z$ and $W$ define elements
\begin{displaymath}
\begin{gathered}
e_Z\in \text{Ext}^1_{\MHS}(\Z(0), H),\\
e^{\vee}_{W}\in \text{Ext}^1_{\MHS}(H, \Z(1)),
\end{gathered}
\end{displaymath}
where $H=H^{2p-1}(X,\Z(p))/\text{torsion}$. The extension class $e_Z$
is defined by a short exact sequence
\begin{displaymath}
0\rightarrow H\rightarrow E_Z\rightarrow \Z(0)\rightarrow 0,
\end{displaymath}
$E_Z$ being a sub-Hodge structure of $H^{2p-1}(X\setminus |Z|,
\Z(p))/\text{torsion}$, whereas $e^{\vee}_W$ is given by a short
exact sequence 
\begin{displaymath}
0\rightarrow \Z(1)\rightarrow E^{\vee}_W\rightarrow H\rightarrow 0,
\end{displaymath}
with $E^{\vee}_{W}$ being a quotient of
$H^{2p-1}(X,|W|,\Z(p))/\text{torsion}$. 
Combining both constructions we get a biextension (Proposition 3.3.2 of \cite{Hain:Height})
\begin{displaymath} 
B_{Z,W}\mapsto  (e_Z, e^{\vee}_W),
\end{displaymath}
which is a subquotient of the mixed Hodge structure
\begin{displaymath}
H^{2p-1}(X\setminus |Z|,|W|, \Z(p))/\text{torsion}.
\end{displaymath}
If $\nu \colon
\caB(H)\to \R$ is the composition of the change of coefficients
$\caB(H)\to \caB(H)_{\R}$ with the bijection above, we have
(Proposition 3.3.12 of \cite{Hain:Height})
\begin{displaymath}
\nu(B_{Z,W})=-\langle Z,W\rangle_{\Arch}.
\end{displaymath}
Since proper intersection
means $|Z|\cap|W|=\phi$, there is a duality
\begin{displaymath}
H^{2p-1}(X\setminus |Z|,|W|, \Q(p))\cong
\left(H^{2q-1}(X\setminus |W|,|Z|,
  \Q(q-1))\right)^{\vee},
\end{displaymath}
which implies that the above pairing is symmetric.

In \cite[Theorem 5.19]{Pearlstein:dl2o} the Hodge theoretical interpretation of the
archimedean height pairing is used to obtain results about its
asymptotic behavior.  Let $Z_{s},W_{s}\subset X_{s}$ be a flat family
of cycles homologous to zero over a smooth curve $S$. Let $z$ be a
local holomorphic coordinate on a small disk $\Delta \subset S$ such
that, for $0\not =z\in \Delta $, the variety  $X_{z}$ is smooth and
the cycles $Z_{z}$ and $W_{z}$ intersect properly and such that the
variation of mixed Hodge structures $B_{Z_{z},W_{z}}$ has unipotent
monodromy. Then there is a rational number $\mu $ that can be read
from the monodromy, and such that
\begin{displaymath}
  \langle Z_{z},W_{z}\rangle_{\Arch}=\mu \log|z| + \eta(z),
\end{displaymath}
where $\eta(z)$ is real analytic and remains bounded when $z$ goes to zero.

\subsection*{Higher intersection pairing}
We recall the construction of the higher height pairing of 
\cite{Burgoswami:hait}. Before that we will also have to recall some
terminology.  

Let now $F$ be any field and $X$ a smooth projective variety over $F$. 
There are two equivalent descriptions of Bloch's higher Chow groups,
the simplicial and the cubical versions. The simplicial version is the
one originally introduced by Bloch, but the cubical version is the one
more well suited for the product structure. In this paper we will use
the cubical description. 

In the cubical version, in order to compute the right homology, one is
forced to normalize the complex in order to get rid of degenerate
elements. There are two versions of
the normalization. In fact, there are two quasi-isomorphic complexes 
\begin{equation}\label{eq:41}
Z^p(X,\ast)_{00}\subset Z^p(X,\ast)_0,
\end{equation}
whose homologies compute the cubical version of higher Chow
groups. We will use the complex
$Z^p(X,\ast)_{00}$ because its cycles are easier to link with
relative cohomology.

Let $\square = \P^{1}\setminus \{1\}$ denote a copy of the affine line
where the role of $\infty$ is played by the point $1$ and let
$\square^{n}$ denote the $n$-th cartesian product. Recall that there
are coface maps $\delta _{j}^{i}\colon \square ^{n-1}\to  \square
^{n}$, $i=1,\dots,n$, $j=0,1$, given by
\begin{align*}
  \delta_0^i(t_1,\dots,t_{n-1}) &= (t_1,\dots,t_{i-1},0,t_{i},\dots,t_{n-1}), \\
\delta_1^i(t_1,\dots,t_{n-1}) &= (t_1,\dots,t_{i-1},\infty,t_{i},\dots,t_{n-1}).
\end{align*}
For any scheme $X$, we denote also by $\delta
^{i}_{j}$ the induced maps $X\times \square^{n-1}\to X\times
\square^{n}$. Any intersection of images of the maps $\delta ^{i}_{j}$ is called a
face. 

Let $Z^{p}(X,n)$ denote the group of algebraic cycles on $X\times
\square ^{n}$ that intersect properly all the faces. Then
\begin{displaymath}
  Z^{p}(X,n)_{00}=\bigcap_{i=1}^{n}\ker(\delta _{1}^{i})^{\ast}\cap
  \bigcap_{i=2}^{n}\ker(\delta _{0}^{i})^{\ast} \end{displaymath}
with differential $\delta \colon Z^{p}(X,n)_{00}\to Z^{p}(X,n-1)_{00}$
given by $\delta =-(\delta _{0}^{1})^{\ast}$. An element of
$Z^{p}(X,n)_{00}$ will be called a pre-cycle, while an element of $Z\in
Z^{p}(X,n)_{00}$ with $\delta Z=0$ is called a cycle. The higher Chow groups of $X$ are
the homology of the complex $(Z^{p}(X,\ast)_{00},\delta )$:
\begin{displaymath}
  \CH^{p}(X,n)=H_{n}(Z^{p}(X,\ast)_{00},\delta ),\quad n\ge 0,\  p\ge
  0. 
\end{displaymath}

There is a graded commutative product in $\CH^{\ast}(X,\ast)$ given
by the intersection product. 

Two pre-cycles $Z\in Z^{p}(X,n)_{00}$ and $W\in Z^{q}(X,m)_{00}$ are
said to intersect properly if $\pi ^{-1}_{1}Z$ and $\pi ^{-1}_{2}W$
intersect properly among them and with all the faces of $X\times
\square^{n+m}$. Here $\pi _{1}\colon X\times \square^{n+m}\to
X\times \square^{n}$ and $\pi _{2}\colon X\times \square^{n+m}\to
X\times \square^{m}$ are the two projections. If $Z$ and $W$ intersect
properly, then the intersection product $Z\cdot W$ is a well defined
pre-cycle of $Z^{p+q}(X,n+m)$.

Let $\alpha \in \CH^{p}(X,n)$ and
$\beta \in \CH^{q}(X,m)$. Then there exist representatives $Z\in
Z^{p}(X,n)_{00}$ and $W\in Z^{q}(X,m)_{00}$ of $\alpha $ and $\beta $
respectively that intersect properly. The product $\alpha \cdot
\beta $ is represented by $Z\cdot W$.

Let now $F$ be a number field and $\Sigma $ the set of complex
immersions of $F$. To the smooth projective variety $X$ over $F$ we associate a
complex variety 
\begin{displaymath}
 X_{\Sigma }=\coprod_{\sigma \in \Sigma } X\times_{\sigma }\C. 
\end{displaymath}
This complex manifold has an antilinear involution $F_{\infty}$ and we
denote $X_{\R}=(X_{\Sigma },F_{\infty})$ the corresponding real
variety.

There are regulator maps $\rho \colon \CH^{p}(X,n)\to
H^{2p-n}_{\DB}(X_{\R},\R(p))$, where $H_{\DB}$ denotes
Deligne cohomology.

In the papers \cite{BurgosFeliu:hacg} and \cite{Burgoswami:hait} the
higher arithmetic Chow groups $\widehat{\CH}^{\ast} (X,\ast,\DB_{\TW})$ of $X$
are introduced and studied.  These groups depend on the choice of a
particular complex $\DB_{\TW}$ that computes Deligne cohomology (see
section \ref{sec-deligne-cohom}). These groups satisfy many properties
similar to the ones of classical arithmetic Chow groups. We summarize
the properties needed in the definition of the height pairing.
\begin{enumerate}
\item The elements of $\widehat{\CH}^{p} (X,n,\DB_{\TW})$ are
  represented by pairs $(Z,g_{Z})$ with $Z\in Z^{p}(X,n)_{00}$ with
  $g_{Z}$ a Green form for $Z$ in the appropriate sense (Definition
  \ref{GF}).
\item To each Green form $g_{Z}$, there is an associated canonical differential form
  $\omega (g_{Z})\in \DB_{\TW}^{2p-n}(X,p)$ that represents the class
  of the regulator $\rho (Z)\in H^{2p-n}_{\DB}(X_{\R},\R(p))$.
\item There is a $\ast$-product of Green forms. 
\item The groups $\widehat{\CH}^{\ast}(X,\ast,\DB_{\TW})$ form a graded
  commutative algebra, where the product is induced by the
  intersection product of cycles meeting properly and the star product
  of Green forms. 
\item If $f\colon X\to Y$ is a smooth morphism of relative dimension
  $e$, there are morphisms
  \begin{displaymath}
    f_{\ast}\colon \widehat{\CH}^{p}(X,n,\DB_{\TW})\to
    \widehat{\CH}^{p-e}(Y,n,\DB_{\TW}), \qquad n,p \ge 0. 
  \end{displaymath}
\item \label{item:15} Writing $X_{F}=\Spec(F)$, there is a short exact sequence
  \begin{equation}\label{eq:64}
    0\to \frac{H_{\DB}^{1}(X_{F,\R},\R(p))}{\text{Image}(\rho )} \to
    \widehat{\CH}^{p}(X_F,2p-2,\DB_{\TW}) \to 
    \CH^{p}(X_F,2p-2) \to 0.
  \end{equation}
\end{enumerate}
In the above $X_{F,\R}$ is the real variety associated
to $X_{F}$. So $H_{\DB}^{1}(X_{F,\R},\R(p))$ is a real vector space of
dimension $r_{1}+r_{2}$ if $p$ is odd and $r_{2}$ if $p$ is
even, where $r_{1}$ is the number of real immersions of $F$ and
$2r_{2}$ is the number of non-real complex immersions. Moreover $\rho
$ agrees with Borel's regulator up to a normalization factor
\cite{Burgos:bb}. Hence $\text{Image} (\rho )$ is a lattice in
$H_{\DB}^{1}(X_{F,\R},\R(p))$. Also $\CH^{p}(X_F,2p-2)$ is
torsion. Thus $\widehat{\CH}^{p}(X_F,2p-2,\DB_{\TW})$ is an
extension of a torsion group by a real torus. 

Let now $\alpha \in \CH^{p}(X,n)$ and $\beta \in \CH^{q}(X,m)$ be two
classes satisfying
\begin{equation}
  \label{eq:43}
 2(p+q-d-1)=n+m 
\end{equation}
and $\rho (\alpha )=\rho (\beta
)=0$. We can find representatives $Z\in Z^{p}(X,n)_{00}$ and $W\in
Z^{q}(X,m)_{00}$ intersecting properly. Much like the usual cycle scenario, we can choose Green
forms $g_{Z}$ and $g_{W}$ for $Z$ and $W$ satisfying the attached differential forms $\omega
(g_{Z})=\omega (g_{W})=0$. Using the properties of the arithmetic Chow
groups we obtain an element
\begin{displaymath}
  \pi _{\ast}((Z,g_{Z})\cdot (W,g_{W}))\in  \widehat{\CH}^{p+q-d}(X_F,n+m,\DB_{\TW}).
\end{displaymath}
The condition \eqref{eq:43} assures us that the target group of the
fits in a short exact sequence like
\eqref{eq:64}. After tensoring with $\Q$ to get rid of the torsion
group on the right of the exact sequence, 
the height pairing of $\alpha $ and
$\beta $  is defined as
\begin{displaymath}
  \langle \alpha ,\beta \rangle_{\Ht}\coloneqq  \pi
  _{\ast}((Z,g_{Z})\cdot (W,g_{W})) \in
  \frac{H_{\DB}^{1}(X_{F,\R},\R(p+q-d))}{\text{Image}(\rho)}\otimes \Q. 
\end{displaymath}
The pairs
\begin{displaymath}
  \langle Z,W\rangle_{\geom}\coloneqq (\pi _{\ast}(Z\cdot
  W),0),\quad\text{and}\quad \langle Z,W\rangle_{\Arch}\coloneqq
  (0,\pi _{\ast}(g_{Z}\ast g_{W}))
\end{displaymath}
give well defined elements of
$\widehat{\CH}^{p+q-d}(X_F,n+m,\DB_{\TW})_{\Q}$ obtaining the
decomposition \eqref{eq:65}. 

We note several differences between the higher height pairing and the
usual height pairing.
\begin{enumerate}
\item The higher arithmetic Chow groups are defined for the variety
  $X$ over $F$ and not for a model $\caX$ of $X$. This is due to the
  fact that the good properties of the higher arithmetic Chow groups
  are only established for varieties over a field. At first glance
  this may seem a big loss of information. However, if $F$ is
  a number field and $\caO_{F}$ is its ring of integers, then for odd $i>1$,
  \begin{displaymath}
    K_{i}(F)\otimes \Q\cong K_{i}(\caO_{F})\otimes \Q. 
  \end{displaymath}
  Therefore, for the purpose of defining higher heights, there is no
  great benefit
  in considering an integral model of the variety. More so if we
  consider that in the classical case, in order to have a well defined
  intersection product on the model one has to tensor with $\Q$.
\item Since we are working over a field, we can define the product
  without tensoring with $\Q$. Nevertheless 
  we can tensor with $\Q$ to eliminate the torsion group
  $\CH^{p+q-d}(X_F,n+m)$.
\item Even if there is no model over the ring of integers
  involved, there is still a geometric contribution of the height
  pairing coming from the intersection of the cycles. By contrast, the
  definition of the archimedean higher height pairing is formally
  identical to the classical case.
\item In order for the height pairing to be independent on the choice
  of the Green forms, we need the condition that the real regulator is
  zero. By contrast the Hodge theoretical invariant associated to the
  pair of cycles can be defined even when the regulator of the cycles
  is non zero.
\item The higher height pairing is not a real number but an element of the
  quotient of $H^{1}_{\DB}(X_{F,\R},\R(p+q-d))$ by the image of the
  regulator. The main difference with the classical case is that, for
  $p+q-d>1$ the image of the regulator is a full rank
  lattice. Therefore we can not obtain a well defined real number. 
\end{enumerate}

Although the height pairing is an arithmetic invariant of the rational
equivalence class of the higher cycles and is well defined up to the
image of the regulator, the archimedean higher height pairing can be
defined purely in the complex case.
\begin{defalph}\label{def:4} Let $X$ be a smooth projective variety over
  $\C$ of dimension $d$ and
$Z\in Z^p(X,n)_{00}$ and $W\in Z^q(X,m)_{00}$ be elements which
satisfy the following conditions: 
\begin{enumerate}
\item  $2(p+q-d-1)=n+m$,
\item  $\delta Z=\delta W=0$,
\item  $Z$ and $W$ intersect properly.
\item  $\rho (Z)=\rho (W)=0$.
\end{enumerate}
Let $g_{Z}$ and $g_{W}$ be Green forms for $Z$ and $W$ satisfying
$\omega (g_{Z})=\omega (g_{W})=0$. 
Then the archimedean height pairing of $Z$ and $W$ is defined as
\begin{displaymath}
  \langle Z,W\rangle_{\Arch}\coloneqq\int_{X}g_{Z}\ast g _{W}\in H^{1}_{\DB}\big(\Spec(\C),\R(p+q-d)\big).
\end{displaymath}

\end{defalph}

\subsection*{Mixed Hodge structures associated to higher cycles}
From now on we consider a smooth projective variety $X$ over $\C$, and
we discuss several mixed Hodge structures associated to a pair of
higher cycles. 

Given cycles $Z\in Z^p(X,n)_{00}$, $W\in Z^q(X,m)_{00}$, with
$p,q,n,m$ satisfying
\begin{equation}
  \label{eq:68}
  2(p+q-d-1)=n+m
\end{equation}
and intersecting properly. Let
$\pi_1\colon X\times (\P^1)^n\times (\P^1)^m\rightarrow X\times
(\P^1)^n$ and $\pi_2\colon X\times (\P^1)^n\times (\P^1)^m\rightarrow
X\times (\P^1)^m$ be the two projections and
\begin{displaymath}
S=\pi^{-1}_1(|Z|)\cap \pi^{-1}_2(|W|)
\end{displaymath}
be the intersection of the pull backs of supports of $Z$ and
$W$. Notice that unlike the usual algebraic cycle scenario, proper intersection
of $Z$ and $W$ no longer means that $S$ is empty. 

We first construct a mixed Hodge structure $E_Z$ for $Z$, fitting in
the short exact sequence
\begin{displaymath}
0\longrightarrow H^{2p-n-1}(X;p)\longrightarrow E_Z\longrightarrow
\Q(0)\longrightarrow 0,
\end{displaymath}
and hence defining an element in 
\begin{displaymath}
  \Ext^1_{\MHS}\left(\Q(0),
  H^{2p-n-1}(X;p)\right)=H^{2p-n}_{\DB}(X,\Q(p)).
\end{displaymath}
This element agrees
with the regulator of $Z$. Next for $W$ we consider the dual
$E^{\vee}_W$ extension, fitting in the short exact sequence
\begin{displaymath}
0\longrightarrow \Q(0)\longrightarrow E^{\vee}_W
  \longrightarrow H^{2d-2q+m+1}(X;d-q)\longrightarrow 0.
\end{displaymath}
We stress the fact that, in giving a geometric interpretation of $E_{W}^{\vee}$ we
face the technical problem that the duality in Lemma \ref{lemm:1}
requires the hypothesis of local product situation. We will address this problem latter in the main body when we give more details on the construction
of the mixed Hodge structures.  

Note that condition \eqref{eq:43} implies that
\begin{displaymath}
  H^{2p-n-1}(X;p)=H^{2d-2q+m+1}(X;d-q+(m+n)/2+1).
\end{displaymath}
Hence, after the appropriate twist, the cohomology groups appearing in
both extensions agree and one may hope to glue together $E_{Z}$ and
$E_{W}^{\vee}$ is a biextension. Here the presence of the non
trivial intersection $S$ makes life more interesting. 
In fact,  under several assumptions aimed to keep the contribution
from $S$ under control, we associate to the pair $(Z,W)$  
a mixed Hodge structure $B_{Z,W}$, which fits in Figure
\ref{fig:oriented_diagram}. In the special case of $n=m=1$, and
under Assumption \ref{def:7}, $B_{Z,W}$ is a generalized biextension
(Definition \ref{def:8} and Corollary \ref{cor:4}), with three
non-zero weight graded pieces.

In Section \ref{sec:height-mixed-hodge} we define the height $\Ht(H)\in\R$ of an
oriented mixed Hodge structure $H$ using the Deligne splitting (Definition
\ref{signed-height}), in particular we can define $\Ht(B_{Z,W})$. For $n=m=1$, to compare it with the archimedean height that lives in $H^{1}_{\DB}(\Spec(\C),\R(2))=\C/(2\pi i)^{2}\R$, we
make the following definition.
\begin{defalph}
  Let $\rho _{2}\colon \C/(2\pi i)^{2}\R\to \R$ be the isomorphism
  given by 
  \begin{displaymath}
\rho _{2}(v)=\I(v/(2\pi i)^{2}),
\end{displaymath}
where for a complex number $z$, $\I(z)$ denote its imaginary part. Then the Hodge theoretic height pairing of $Z$ and $W$ is
  \begin{displaymath}
    \langle Z, W\rangle _{\Hodge}=\rho _{2}^{-1}(\Ht(B_{Z,W})). 
  \end{displaymath}
\end{defalph}

We give a little bit more details on the construction of the above
mixed Hodge structures.
On $(\P^1)^n$ we have two divisors
 \begin{align*}
    A &= \{(t_1,\cdots , t_n)\mid \exists i, t_i=1\},\\
    B & =\{(t_1,\cdots , t_n)\mid \exists i, t_i\in \{0, \infty\}\}.
  \end{align*}
Then $A\cup B$ is a simple normal crossing divisor. Moreover, since 
$\square\coloneqq \P^1\setminus \{1:1\}$, we have
\begin{displaymath}
  (\P^1)^n\setminus A=\square^n,\quad (\P^1)^n\setminus B=(\C^{\times})^n
\text{ and } B\cap \square^n=\partial \square^n.
\end{displaymath} 
Further for $A_X\coloneqq X\times A$, $B_X\coloneqq X\times B$, we get
isomorphisms of Hodge structures 
\begin{align}
  H^r(X\times (\P^1)^n\setminus A_X, B_X)
  &\cong H^{r-n}(X),\label{eq:74}\\
H^r(X\times (\P^1)^n\setminus B_X, A_X)&\cong H^{r-n}(X;-n).\label{eq:92}
\end{align}
Since $A_X$ and $B_X$ are in local product situation (Definition
\ref{def:sec-1-product}), the above isomorphisms are compatible
with duality
\begin{multline*}
H^r(X\times (\P^1)^n\setminus A_X, B_X, \Q(p))
\cong \\ \left(H^{2d+2n-r}\left(X\times (\P^1)^n\setminus B_X, A_X,
    \Q(d+n-p)\right)\right)^{\vee}.
\end{multline*}

Since $Z\in Z^{p}(X,n)_{00}$ belongs to the refined normalized
complex, then the restriction $Z |_{B_{X}\setminus A_{X}}$ is
zero. Therefore,
the cycle $Z$ defines an unique class (Proposition \ref{propclass}) 
\begin{displaymath}
[Z]\in H^{2p}_{|Z|\setminus A_X}\left(X\times (\P^1)^n\setminus A_X,
    B_X\setminus A_X;p\right)_{\Q}.
\end{displaymath}
By Lemma \ref{lemcy0} its image in $H^{2p}\left(X\times (\P^1)^n\setminus A_X,
  B_X\setminus A_X;p\right)_{\Q}$ is zero.
Pulling back the long exact sequence \ref{eq:62} of mixed Hodge
structures, by the class $[Z]$ and using the isomorphism \eqref{eq:74}
we obtain the extension $E_{Z}$. We remark that 
the mixed Hodge structure $E_Z$ is a sub Hodge structure of
\begin{displaymath}
H^{2p-1}(X\times (\P^1)^n\setminus A_X\cup |Z|, B_X;p).
\end{displaymath}

We now consider the dual construction for $W$.
As mentioned before, in order to dualize this construction we face the
problem that, in general, $A_X\cup |W|$ and $B_{X}$ are not in local
product situation. Therefore, to dualize we need first to blow up
$|W|\cap B_{X}$ until a local product situation is obtained. Let $\caX_{W}$
be such a blow up with $\widehat A_{X}$, $\widehat B_{X}$ and $\widehat W$ being
the strict transforms of $A_{X}$, $B_{X}$ and $|W|$. Let $D$ be the
exceptional divisor. Naively one would expect the mixed Hodge
structure to be a quotient of
\begin{multline*}
  H^{2d+2m-2q+1}(X\times (\P^1)^m\setminus B_{X}, A_X\cup |W|;d+m-q)
  = \\H^{2d+2m-2q+1}(\caX_{W}\setminus \widehat B_{X}\cup D, \widehat
  A_X\cup \widehat W;d+m-q),
\end{multline*}
but in fact $E_{W}^{\vee}$ is a quotient of
\begin{displaymath}
  H^{2d+2m-2q+1}(\caX_{W}\setminus \widehat B_{X}, \widehat
  A_X\cup \widehat W\cup D;d+m-q).
\end{displaymath}
Note that the exceptional divisor is in a different position. See
Section \ref{dualext} for more details.

Finally, for the construction of $B_{Z,W}$ we refer to Section
\ref{sec:orient-mhs-attach} and \ref{sec:case-n=m=1}. We just remark
that in this construction we have to deal, not only with the duality
problem mentioned above but also with the contribution of the
intersection $S$ of $Z$ and $W$. Although the methods of this paper can be extended
to much more general situations, for the moment we have only made the
complete study in the case $n=m=1$ and Assumption \ref{def:7}. One of
the main reasons is that we want $B_{Z,W}$ to be a generalized
biextension (Definition \ref{def:8}), so there is a clean definition
of the height of $B_{Z,W}$ that we can compare with the height pairing
of $Z$ and $W$. This forces us to keep $S$ under control to avoid many
spurious components in  $B_{Z,W}$. For instance, even if $S$ is a
point, if it is contained in the singular locus of $|Z|$ and $|W|$,
the cohomology with support on $S$ can be very complicated and
mask the classes of $Z$ and $W$.

Nevertheless, using the Deligne splitting one can define a more general height attached
to an oriented mixed Hodge structure (Definition \ref{signed-height}). One would expect
that the main result of this paper can be extended to a more general
situation using this generalization of the height of a mixed Hodge
structure.

\subsection*{Examples}\label{sec:examples}

We compute two examples of the higher height pairing. The first one is
in dimension $0$ with $p=q=n=m=1$. In this case we find that the
higher height pairing is always zero.

The second more interesting example in in dimension $2$, with
$X=\P^{2}$, $p=q=2$ and $n=m=1$. A method of constructing higher
cycles in $\P^{2}$ is to consider three section $s_{0}$, $s_{1}$ and
$s_{2}$ of $\caO(1)$. They determine a triangle in $\P^{2}$ and a
higher cycle as explained in Definition \ref{def:12}. For two such
higher cycles $Z$ and $W$ in general position we compute their higher
height pairing. It turns out to be given by a linear combination of
values of the Bloch Wigner dilogarithm function. A remarkable feature
of this example is that, in the space of parameters of such pair of
divisors, the height function can be extended continuously  to the
degenerate situations. A second observation from the example is that,
when both triangles are defined over $\R$ the higher height pairing
vanishes. Both phenomenons turn out to hold in more general
situations. With respect to the second one, we show in Proposition
\ref{prop:7} that the higher height pairing between cycles defined
over $\R$ should be zero as long as $(n+m)/2$ is odd.

With respect to the continuity of the height function, this is the
starting point of the study of the  asymptotic behavior. As mentioned
previously, we show that the higher height of an admissible variation
$\caV$ of oriented Hodge-Tate mixed Hodge structures extends continuously to
the boundary. It is important to note that this is no longer true if
the variation is not of Hodge-Tate type (Example \ref{exam:6iii}) or
if $\ell(\caV)=2$.

\subsection*{Layout of the paper}\label{subsec-layout}
Our paper is organized as follows. Sections \ref{sec:prelim} and
\ref{sec:mhs-arc} are preliminary in nature where we set up notations
and collect all the necessary results and definitions needed for the
rest of the sections. In Section \ref{sec: mixed hodge-higher cycles}
we study the mixed Hodge structure associated to higher cycles, the
key among them is a mixed Hodge structure associated to a pair of
higher cycles satisfying a numerical condition. In Section \ref{sec5} we
compute the invariants associated to these mixed Hodge structures in a
special case scenario. A key result in this section is the equality of
higher archimedean height pairing and the height of the biextension,
in case the higher cycles have trivial real regulators. Section \ref{sec:example} is devoted towards computing these invariants in specific examples arising from non-degenerate triangles in $\P^2$. We see that the height of the biextension in this case is given by a sum of Bloch-Wigner Dilogarithm functions. Finally in Section \ref{sec:assymptotic mhs} we study the asymptotic behavior of variations of oriented mixed Hodge structures of Hodge--Tate type and as well as arbitrary admissible variations.

\subsection*{Acknowledgements}\label{subsec-Ac} During the elaboration
of this paper we have benefited with stimulating conversation with many
mathematicians. Our thanks to S. Bloch, P. Brosnan, D. Roessler and
specially J. Lewis  who has a different approach to higher heights and
has continuously encouraged us to pursue this research and to Matt
Kerr who shared with us some unpublished notes with his ideas about
higher heights. We also want to thank the anonymous referee for carefully reading the manuscript and making many useful comments,
suggestions and questions.  

The whole project was envisioned and mostly carried out at the Texas A\&M University. We would like to thank the institute for their generous funding and warm hospitality. Part of the work on this project has been done during a stay of the three
authors at the 
Newton Institute during the program KAH. The authors would like to
thank the institute for their support and hospitality.  Part of
this work was also done during a visit of the first and third authors
to the Freiburg Institute for Advanced Studies, and we thank the
institute for their support and hospitality. The second author would
also like to thank the SFB Higher Invariants group at the
Universit\"at Regensburg for funding and kindly hosting him during the current
development of the project.

\section{Preliminaries}
\label{sec:prelim}

In this section we gather all the conventions, notations and known results that will be used
throughout the paper. All through the section $X$ will denote a smooth complex variety of dimension
$d$. To avoid cumbersome notation, we will not distinguish notationally
between a complex algebraic variety and its associated complex
space. That is the symbol $X$ will also denote the 
associated analytic manifold with the classical topology. It will always
be clear from the context whether $X$ denotes the algebraic variety or
the complex manifold.  

\subsection{Mixed Hodge structures.}
\label{sec:conv-mixed-hodge}

A $\Q$-mixed Hodge structure is a triple
\begin{displaymath}
  H=((H_{\Q},W),(H_{\C},W,F),\alpha)
\end{displaymath}
where $(H_{\Q},W)$ is a $\Q$-vector space with an increasing
filtration  $W$, while $(H_{\C},W,F)$ is a complex vector space with
an increasing filtration $W$ and a decreasing filtration $F$, and that $\alpha
\colon (H_{\Q},W)\otimes \C \to (H_{\C},W)$ is a filtered
isomorphism. These data are subjected to several axioms. See for
instance \cite[Definition 3.13]{MHS}. The vector space $H_{\Q}$ is
called the \emph{Betti component} and $H_{\C}$ the \emph{de Rham component}, while
$\alpha $ is the comparison isomorphism. The \emph{rank} of a mixed
Hodge structure $H$ is the complex
dimension of $H_{\C}$ that agrees with the dimension of $H_{\Q}$ over
$\Q$. 

One can also consider real mixed Hodge structures, where instead
of a $\Q$-vector space $H_{\Q}$ one has a real vector space
$H_{\R}$. In fact, 
given a mixed $\Q$-Hodge
structure $H$ we will denote $H_{\R}=H_{\Q}\otimes \R$ obtaining an
$\R$-mixed Hodge structure.
Usually one
identifies $H_{\Q}$ and $H_{\R}$ with its image in $H_{\C}$ through
$\alpha $.

When studying variations of mixed Hodge structures it is convenient to
fix the underlying vector space and move the filtrations $F$ and $W$. Thus
if we fix an ($\R$ or $\Q$) vector space $V$, then a pair of
filtrations $(F,W)$ on $V\otimes \C$ and $V$ respectively is called a
mixed Hodge structure if the triple
\begin{displaymath}
  ((V,W),(V\otimes \C,W,F),\Id_{V\otimes \C})
\end{displaymath}
is a mixed Hodge structure.

For $a\in \Z$, the Tate mixed
Hodge structure $\Q(a)$ is the mixed Hodge structure given by the
following data
\begin{gather*}
  \Q(a)_{\Q}=\Q,\quad W_{-2a-1}\Q(a)_{\Q}=0,\quad
  W_{-2a}\Q(a)_{\Q}=\Q\\
  \Q(a)_{\C}=\C,\quad F^{-a}\Q(a)_{\C}=\C,\quad
  F^{-a+1}\Q(a)_{\C}=0\\
  \alpha (1)=(2\pi i)^{a}\in \C.
\end{gather*}
Note that on $\Q(a)_{\C}=\C$ we have two possible complex conjugations. The
usual conjugation of $\C$ and the one induced by the isomorphism
$\alpha $. The first one will be called the \emph{de Rham conjugation}
and
denoted $z\mapsto \fancyconj{z}{\dR}$ and the second will be called the
\emph{Betti conjugation} and denoted $z\mapsto \betticonj{z}$. These
two conjugations are related by
\begin{displaymath}
  \betticonj{z}=(-1)^{a}\fancyconj{z}{\dR}.
\end{displaymath}
In the sequel we will mainly use the Betti conjugation and write
$\overline z=\betticonj{z}$.
Moreover, the mixed Hodge structure $\Q(a)$ comes equipped with the
choice of two
generators.
\begin{align*}
  \bfone(a)_{\Q} &=1\in \Q=\Q(a)_{\Q}\\
  \bfone(a)_{\C} &=1\in \C=\Q(a)_{\C}.
\end{align*}
These generators are called the Betti and the de Rham generators. They
satisfy
\begin{displaymath}
  \overline {\bfone(a)_{\Q}}=\bfone(a)_{\Q},\qquad
  \overline {\bfone(a)_{\C}}=(-1)^{a}\bfone(a)_{\C},\qquad
  \bfone(a)_{\Q}=(2\pi i)^{a}\bfone(a)_{\Q}.
\end{displaymath}
\begin{rmk}\label{rem:4}
  Note that, although the isomorphisms class of $\Q(a)$ does not
  depend on the choice of a square root of $-1$, $i=\sqrt{-1}$, when
  $a$ is odd, the ratio of the chosen generators
  $\bfone(a)_{\Q}/\bfone(a)_{\C}$ does.
\end{rmk}
\begin{rmk}
  Let $H$ be a $\Q$-mixed Hodge structure of rank one. Then $H$ is
  necessarily pure of even weight, say $2a$. It follows that it is
  isomorphic to 
  $\Q(-a)$. The choice of an isomorphism $H\to \Q(a)$ is equivalent to
  the choice of a generator $e$ of $H_{\Q}$. 
\end{rmk}

If $Z\subset X$ is a closed subvariety and $r\in \Z$, then the cohomology groups
\begin{displaymath}
  H^{r}(X;\Q),\quad H^{r}(X,Z;\Q) \text{ and }
  H^{r}_{Z}(X;\Q)=H^{r}(X,X\setminus Z;\Q)
\end{displaymath}
are all the Betti part of $\Q$-mixed Hodge structures that we denote as 
\begin{displaymath}
  H^{r}(X),\quad H^{r}(X,Z) \text{ and } H^{r}_{Z}(X)=H^{r}(X,X\setminus Z) 
\end{displaymath}
respectively. We will use the shorthand
\begin{displaymath}
  H^{r}(X;p)=H^{r}(X)\otimes \Q(p).
\end{displaymath}
Then $H^{r}(X;p)_{\Q}$, $H^{r}(X;p)_{\R}$ and $H^{r}(X;p)_{\C}$ will
denote the rational and  real Betti component and the complex de Rham
component respectively.

Frequently, in the sequel we will use complexes that compute relative
cohomology of a complex projective variety, but they only have
information about the real structure and the Hodge filtration, and
not about the weight filtration. To work with these at ease we
introduce the following notation.  

\begin{df}\label{def:1} A \emph{weak $\R$-Hodge complex} is a complex
  $(A^{\ast},d)$ of $\C$-vector spaces together with an anti-linear 
  involution  $\omega \mapsto \overline \omega $ commuting with $d$
  and a decreasing filtration $F$ (called the Hodge filtration)
  compatible with $d$. If $A^{\ast}$ 
  is a weak $\R$-Hodge complex, we denote by $A^{\ast}_{\R}$ the
  subcomplex of elements fixed by the involution. 

  Given a weak $\R$-Hodge complex $A^{\ast}$, the
  \emph{Tate
    twisted} weak $\R$-Hodge complex is defined as
  $A^{\ast}(a)=A^{\ast}\otimes \Q(a)_{\C}$. Using the identification
  $\Q(a)_{\C}=\C$, the complex $A^{\ast}(a)$ is given by the following
  data  
  \begin{displaymath}
    A^{\ast}(a)=A^{\ast},\qquad
    \fancyconj{z}{\new}=(-1)^{a}\fancyconj{z}{\old},\qquad
    F^{b}A^{\ast}(a)=F^{a+b}A^{\ast}.
  \end{displaymath}
  The superindexes $\new$ and $\old$ are written here for clarity but
  will not be used in the sequel. Due to the identification 
  $A^{\ast}(a)=A^{\ast}\otimes \Q(a)_{\C}=A^{\ast}\otimes \C=A^{\ast}$
  there is a potential ambiguity in 
  the use of the symbol $\overline \omega $, as it depends on whether
  we consider $\omega $ as an element of $A^{\ast}$ or
  of $A^{\ast}(a)$. In some rare cases, for clarity, an element $\omega \in
  A^{\ast}(a)$ will be written as $\omega\otimes \bfone(a)_{\C}$. 
\end{df}

\begin{rmk}\label{rem:2}
  Any \emph{Dolbeault complex} as in \cite[Definition
2.2]{Burgos:CDB} defines a weak $\R$-Hodge complex.  
\end{rmk}

Recall that the shifted complex 
$A^{\ast}[r]$ is defined by $A^{n}[r]=A^{n+r}$ with
differential $(-1)^{r}d$.

\subsection{Conventions on differential forms and
  currents}\label{subsecconv}
\label{sec:conv-diff-forms}

When dealing with differential forms, currents and cohomology
classes, one can use the topologist's convention, where the emphasis is
put on having real or integral valued classes in singular
cohomology. For instance, in this 
convention the first Chern class of  
a line bundle will have integral coefficients. In algebraic geometry,
the fact that rational de Rham classes are not rational in singular
cohomology, the ubiquitous appearance of the period $2\pi 
i$, and the fact that the choice of a particular square root of $-1$ is non
canonical, makes it useful to use a different convention.

This algebro-geometric convention aims to control the obvious powers
of $2\pi i$ and to be independent of the choice of the imaginary unit
$i=\sqrt{-1}$.

Of course
using one convention or the other is a matter of taste and one can go
easily from one to the other by a normalization factor. In this paper
we will follow the algebro-geometric convention. Therefore, it is
useful to incorporate different powers of $2\pi i$
in the standard operations regarding forms and currents as in  \cite[\S
5.4]{BurgosKramerKuehn:cacg}. We summarize here the conventions used
because they differ from commonly used notations. 

We will denote by $E_{X}^{\ast}$ the
differential graded algebra of complex valued differential forms on
$X$, by $E_{X,\R}^{\ast}$ the subalgebra of real valued forms and by
$E_{X,c}^{\ast}$ and $E_{X,\R,c}^{\ast}$ the subalgebras of
differential forms with compact support. The complexes of currents are
defined as the topological duals of the latter ones. Namely
$E^{\prime-n}_{X}$ and $E^{\prime-n}_{X,\R}$  are the topological dual
of $E_{X,c}^{n}$ and $E_{X,\R,c}^{n}$ respectively, with differential
given by
\begin{displaymath}
  dT(\eta)=(-1)^{n+1}T(d\eta).
\end{displaymath}
Recall that $X$ is smooth of dimension $d$.  
We write
\begin{displaymath}
  D^{\ast}_{X}=E_{X}^{\prime \ast}[-2d](-d).
\end{displaymath}
This implies that 
  \begin{displaymath}
    D^{n}_{X,\R}=\{T\in D^{n}_{X}\mid \forall \eta\in
    E^{2d-n}_{X,\R,c},\ T(\eta)\in (2\pi i)^{-d}\R\}.
  \end{displaymath}
  Hence one can see $D^{n}_{X,\R}(p)$ as the topological dual of
  $E^{2d-n}_{X, \R}(d-p)$.

  We now consider the current $\int_{X}$ given  by
  \begin{displaymath}
    \omega \mapsto \int_{X}\omega.
  \end{displaymath}
Then $\int_{X}\in E_{X,\R}^{\prime-2d}=D_{X,\R}^{0}(d)$. This suggest
to define
\begin{displaymath}
  \delta _{X}\coloneqq {\textstyle \int_{X}}\otimes \bfone(-d)_{\Q}=\frac{1}{(2\pi i)^{d}}{\textstyle \int_{X}}\otimes \bfone(-d)_{\C}\in D^{0}_{X,\R}.
\end{displaymath}

\begin{rmk}\label{rem:15}
  The current $\delta _{X}$ has two advantages over the current
$\int_{X}$. The first one is that $\delta _{X}$ is independent of the choice of
 square root of $-1$  while the
 current $\int _{X}$ is not. Indeed, If
  $z_{1},\dots,z_{d}$ are local complex coordinates with
  $z_{j}=x_{j}+i y_{j}$, then the standard orientation is given by
  the volume form
  \begin{displaymath}
    \vol=dx_{1}\land d y_{1}\land \dots \land dx_{d}\land d y_{d}.
  \end{displaymath}
  If we change the choice of the square root of $-1$ from $i$ to
  $-i$ then $\vol $ is sent to $(-1)^{d}\vol$, which is the same
  change of sign suffered by $(2\pi i)^{d}$. 
  Of course this explains the presence of $i^{d}$
  but not the presence of $(2\pi)^{d}$. The second advantage of
  $\delta_{X}$ is that,   
  if $X$ is defined over $\Q$ and $\omega $ is a differential form 
  representing a rational class in $H_{\Zar}^{2d}(X,\Omega 
  _{X_{\Q}}^{\ast})$, then $\delta _{X}(\omega )\in \Q$.   
\end{rmk}

To be consistent with the previous choice we also need to adjust the
definition
of the current associated to a locally integrable form and to an
algebraic cycle. Given a locally integrable differential form
$\omega $ of degree $n$, there is a current
\begin{displaymath}
  \int_{X}\omega \wedge \cdot \in E_{X}^{\prime n-2d}=D_{X}^{n}(d).
\end{displaymath}
we will denote by $[\omega ]\in D^{n}_{X}$
the current defined by
\begin{equation}\label{eq:7}
  [\omega ]=\int_{X}\omega \land \cdot \otimes \bfone(-d)_{\Q}
    =\frac{1}{(2\pi i)^{d}}\int_{X}\omega \land \cdot \otimes
    \bfone(-d)_{\C} \in  D^{n}_{X}.
\end{equation}
In other words $[\omega ]=\delta _{X}\wedge \omega $.
With this convention, the morphism of complexes 
\begin{math}
  [\cdot]\colon E^{\ast}_{X}\to D^{\ast}_{X} 
\end{math}
respects the structure of weak Hodge complexes on both sides.

If $f\colon X\to Y$ is a proper map of smooth complex varieties, of
dimensions $d,d'$ and relative
dimension $e=d-d'$, then the
push-forward of currents
$f_{\ast}\colon E^{\prime \ast}_{X}\to
E^{\prime \ast}_{Y}$
is defined, for $T\in D^{n}_{X}$ and $\eta\in
E^{2d-n}_{Y,c}$ by
\begin{displaymath}
  f_{\ast}T(\eta)=T(f^{\ast}\eta).
\end{displaymath}
It induces a map 
$f_{\ast}\colon D^{\ast}_{X}\to
D^{\ast}_{Y}[-2e](-e)$.

Finally, let $Z\subset X$ be a
codimension $p$ irreducible subvariety of $X$. Let $\iota \colon \widetilde Z\to X$ be
a resolution of singularities of $Z$. Then the current integration
along $Z$ is defined as 
\begin{displaymath}
  \delta _{Z}=\iota_{\ast}(\delta _{\widetilde Z})
  \in D^{2p}_{X,\R}(p)
\end{displaymath}

\begin{rmk}
  Since $Z$ is irreducible, $H^{2p}_{Z}(X,p)=\Q(0)$ and the class of
  $\delta _{Z}$ is at the same time the Betti and the de Rham
  generator of $\Q(0)$. 
\end{rmk}

Given any cycle
$\zeta\in Z^{p}(X)$ we define $\delta _{\zeta}$ by linearity.
Following Remark \ref{rem:15}, the symbols
$[\omega ]$ and $\delta _{Y}$ do not depend on a
particular choice of $\sqrt{-1}$.

\begin{ex}\label{exm:5} To see how this conventions, together with the
  convention in definition \ref{def:1} work in
  practice, we review the classical example of the
  logarithm. Consider $X=\P^{1}$ with absolute
  coordinate $t$, so $\Div(t)=[0]-[\infty]$, and let
  $U=X\setminus \{0,\infty \}$. Write
  \begin{displaymath}
    \log(t\bar t)\in E^{0}_{U}(1),\quad \frac{dt}{t},\ \frac{d\bar
      t}{\bar t}\in E^{1}_{U}(1).
  \end{displaymath}
  Note that, if we want to stress the fact that these elements belong
  to the twisted complex we will denote them like $\log(t\bar
  t)\otimes \bfone(1)_{\C}$. These elements satisfy
  \begin{displaymath}
    \left[\log t\bar t \,\right]\in D^{0}_{\P^{1}}(1),\quad 
      \left[\frac{dt}{t}\right ],\ \left[\frac{d\bar t}{\bar
          t}\,\right ]
      \in D^{1}_{\P^{1}}(1).
  \end{displaymath}
  Moreover,
  \begin{equation}\label{eq:77}
    \begin{aligned}
      \overline{\log (t\bar t)}&=- \log (t\bar t)\\
      \overline{\frac{dt}{t}} & = -\frac{d\bar t}{\bar t}\\ 
      d \left[\frac{dt}{t}\right ]
      &=\delta _{\Div t}=\delta_{0}-\delta _{\infty},\\
      d \left[\frac{d\bar t}{\bar t}\,\right ]
      &=-\delta _{\Div t}=\delta_{\infty}-\delta _{0}.\\
      \partial\bar \partial \left[\log t\bar t \,\right]
      &=-\delta _{\Div t}=\delta_{\infty}-\delta _{0}
    \end{aligned}
  \end{equation}
  Note how, in the above formulae all the $(2\pi i)$ factors are now
  implicit.

  Recall also the potential ambiguity on the sign of the conjugation
  mentioned at the end of 
  Definition \ref{def:1}. The typical example to keep in mind would be
  the form 
  \begin{displaymath}
    \eta=\frac{1}{2}
    \left(\frac{dt}{t}-\frac{d\overline{t}}{\overline{t}}\right)\in  
    E^{1}_{U}(1), 
  \end{displaymath}
  that represents a generator of $H^{1}(U;1)$. Since $\eta$
  is an element of $E^{1}_{U}(1)$ then $\overline
  \eta=\eta$. Hence $\eta\in E^{1}_{U}(1)_{\R}$. By contrast, if
  $\eta_{0}\in E^{1}_{U}$ is the differential form with the same
  values as $\eta$, but this time belonging to $E^{1}_{U}$, then
  $\overline {\eta_{0}}=-\eta_{0}$. Thus
  $\eta_{0}$ is purely imaginary. 
\end{ex}

\subsection{Local product situation and duality}
\label{sec:local-prod-situ}

Assume in this subsection that $X$ is projective in
order to have Poincar\'e duality. 
Let $A\subset X$ be a Zariski closed subset and $a,r\in \Z$. Then
Lefschetz duality tells us 
that there is an isomorphism of Mixed Hodge structures
\begin{displaymath}
  H^{r}(X\setminus A;a)\cong H^{2d-r}(X,A;d-a)^{\vee}.
\end{displaymath}
If $B$ is a second Zariski closed subset one may ask if there is a
refined duality
\begin{equation}\label{eq:10}
  H^{r}(X\setminus A,B;a)\cong H^{2d-r}(X\setminus B,A;d-a)^{\vee}? 
\end{equation}
In general the answer is no as the following example shows.

\begin{ex}\label{exm:1}
  In this example we put $X=\P^{2}$. Let $\ell_{0}$, $\ell_{1}$ and
  $\ell_{2}$ be three different lines passing through the same point
  $p$ and write $A=\ell_{0}\cup \ell_{1}$ and $B=\ell_{2}$. Then
  \begin{alignat*}{2}
    H^{1}(X\setminus A,B)&= \Q(-1),& \quad H^{3}(X\setminus A,B)&=0,\\
    H^{1}(X\setminus B,A)&= \Q(0),& H^{3}(X\setminus B,A)&=0.
  \end{alignat*}
  Thus, the answer to question \eqref{eq:10} is negative.
\end{ex}

Nevertheless, if we add some hypothesis to the sets $A$ and $B$ we can
have a positive answer.

\begin{df}\label{def:sec-1-product}
  Let $A$ and $B$ be closed subvarieties of $X$. We say that $A$ and
  $B$ are in a \emph{local product situation} if, for any point $x\in X$
  there is a
  neighborhood $U$ of $x$, a decomposition $U=U_{A}\times U_{B}$,
  where $U_{A}$ and $U_{B}$ are open disks of smaller dimension, and
  analytic subvarieties $A'\subset U_{A}$ and $B'\subset U_{B}$ such
  that
  \begin{displaymath}
    A\cap U = A'\times U_{B},\qquad B\cap U = U_{A}\times B'.
  \end{displaymath}
\end{df}

\begin{rmk}
  The sets $A$ and $B$ of Example \ref{exm:1} are not in a local product
  situation. By contrast, if $A$ and $B$ are divisors
  without common components such that $A\cap B$ is a normal
  crossing divisor, then $A$ and $B$ are in local product situation. 
\end{rmk}

The following result is proved in \cite[Lemma 6.1.1]{BKV:Fihnf}.

\begin{lem}\label{lemm:1}
  Let $A$ and $B$ be closed
  subvarieties of $X$ in local product situation. Then, for every
  $a,r\in \Z$, there is an
  isomorphism of mixed Hodge structures
  \begin{displaymath}
    H^{r}(X\setminus A,B;a)\xrightarrow{\cong}
    H^{2d-r}(X\setminus B,A;d-a)^{\vee}.
  \end{displaymath}
\end{lem}

In the next section we will explain how to realize this isomorphism
explicitly, after tensoring with $\R$, using differential forms. 

We give now two applications of duality.

\begin{lem}\label{lemm:2} Let $Z\subset X$
  be a closed subvariety and let $\pi \colon \widetilde X\to X$
  be a blow-up with center contained in $Z$ such that $\widetilde X$
  is smooth. Write $\widetilde Z=\pi ^{-1}(Z)$. Then, for all $a,r\in
  \Z$, the maps 
  \begin{align}
    \label{eq:12}
    H^{r}(X\setminus Z;a)&\xrightarrow{\pi ^{\ast}}
    H^{r}(\widetilde X\setminus \widetilde Z;a)\\
    H^{r}(X, Z;a)&\xrightarrow{\pi ^{\ast}}
    H^{r}(\widetilde X,\widetilde Z;a)
    \label{eq:13}
  \end{align}
  are isomorphisms.
\end{lem}
\begin{proof}
  The fact that \eqref{eq:12} is an isomorphism is obvious because
  $X\setminus Z = \widetilde X\setminus \widetilde Z$. By the
  functoriality of duality, the
  morphism \eqref{eq:13} is the composition
  \begin{multline*}
    H^{r}(X, Z;a)\xrightarrow{\cong}
    H^{2d-r}(X\setminus Z;d-a)^{\vee}
    \xrightarrow{(\pi _{\ast})^{\vee}}\\
    H^{2d-r}(\widetilde X\setminus \widetilde Z;d-a))^{\vee}
    \xrightarrow{\cong}
    H^{r}(\widetilde X,\widetilde Z;a).
  \end{multline*}
  Since the map $\pi _{\ast}$ in the middle is also an isomorphism by
  the same reason as before, we conclude that \eqref{eq:13} is an
  isomorphism. 
\end{proof}

The next result tell us the surprising fact that, under some
conditions, we can shift, in the isomorphism of Lemma \ref{lemm:1},
part of the closed subset $A$ to the closed subset $B$.

\begin{lem}\label{lemm:3} Let $A,B$
  be two divisors without common components such that $A\cup B$ is a
  normal crossing divisor. Let $\pi \colon \widetilde X \to X$ be a
  blow up with center contained in $A\cap B$ such that $\widetilde X$
  is smooth and $\pi ^{-1}(A\cup B)$ is a normal crossing
  divisor. Let $\widehat A$ and $\widehat B$ be the strict transforms of
  $A$ and $B$ respectively and $C$ the exceptional divisor of $\pi
  $. Then, for all $a,r\in \Z$ there are isomorphism
  \begin{align}
    \label{eq:14}
    H^{r}(X\setminus A,B;a)
    &\xrightarrow{\pi ^{\ast}}
      H^{r}(\widetilde X\setminus \widehat A\cup C,\widehat B;a),\\
    H^{r}(X\setminus A,B;a)\label{eq:15}
    &\xrightarrow{\cong}
      H^{r}(\widetilde X\setminus \widehat A,\widehat B\cup C;a).    
  \end{align}
\end{lem}
\begin{proof}
  The fact that $\pi ^{\ast}$ is an isomorphism is a consequence of
  the equalities
  \begin{displaymath}
    X\setminus A = \widetilde X\setminus (\widehat A\cup C),\qquad
    B\setminus A = \widehat B \setminus (\widehat A\cup C)
  \end{displaymath}
  The isomorphism \eqref{eq:15} is the composition of the isomorphisms
  \begin{align*}
    H^{r}(X\setminus A,B;a)&\xrightarrow{\cong}
    H^{2d-r}(X\setminus B,A;d-a)^{\vee}\\
    &\xrightarrow{\cong}
    H^{2d-r}(\widetilde X\setminus \widehat B\cup C,\widehat A;d-a)^{\vee}\\
    &\xrightarrow{\cong}
    H^{r}(\widetilde X\setminus \widehat A,\widehat B\cup C;a),
  \end{align*}
  where the the existence of the first and third isomorphism is a
  consequence of Lemma \ref{lemm:1} and the second isomorphism agrees
  with the
  isomorphism \eqref{eq:14} applied with  $A$ and $B$ interchanged. 
\end{proof}
\subsection{Differential forms with zeros and logarithmic poles}
\label{sec:diff-forms-with}
Let $Y\subset X$ be a closed subvariety, $\widetilde Y$ a
resolution 
of singularities of $Y$ and $\iota \colon \widetilde Y\to X$ the induced
map. We denote 
\begin{displaymath}
  \Sigma _{Y}E^{\ast}_{X}=\{\omega \in E^{\ast}_{X}\mid
  \iota ^{\ast} \omega =0\}.
\end{displaymath}
Then $\Sigma _{Y}E^{\ast}_{X}$ is an example of a Dolbeault
complex. In particular is a weak $\R$-Hodge complex. Therefore we can
apply to it the notation of Definition~\ref{def:1}. We begin with a
basic observation. 

\begin{prop}\label{prop0} Let $Y\subset X$ be a smooth subvariety. Then
The complexes $\Sigma _{Y}E^{\ast}_{X}$ and
$s(E^\ast_X\xrightarrow{\iota^\ast} E^\ast_{\widetilde{Y}})$ are
quasi-isomorphic.
\end{prop}

\begin{proof}
For smooth $Y$ the sequence
\begin{displaymath}
0\rightarrow  \Sigma _{Y}E^{\ast}_{X}\rightarrow E^\ast_X\rightarrow
E^\ast_Y\rightarrow 0
\end{displaymath}
is exact, what implies the result.
\end{proof}

Note that we do
not put a weight filtration on $\Sigma _{Y}E^{\ast}_{X}$. Nevertheless
in good conditions  the complex $\Sigma
_{Y}E^{\ast}_{X}$ allows us to compute part of the mixed Hodge
structure of the relative cohomology of the pair $(X,Y)$.

\begin{prop}\label{prop:1} Assume that $X$ is projective. Let $A$ be a
  normal crossing divisor of $X$ and let $W$ be a smooth closed
  subvariety that intersects transversely all intersections among the
  components of $A$. Write $Y=A\cup W$. Assume furthermore that all possible
  intersections among components of $Y$ are smooth and irreducible. 
  Then, there is a mixed Hodge complex $K$ that computes the relative
  cohomology groups $H^{\ast}(X,Y)$, a quasi-isomorphism
  \begin{displaymath}
    \Sigma_{Y}E^{\ast}_{X,\R}\longrightarrow K_{\R}
  \end{displaymath}
  and a compatible filtered quasi-isomorphism 
  \begin{displaymath}
    (\Sigma_{Y}E^{\ast}_{X},F)\longrightarrow (K_{\C},F).
  \end{displaymath}
\end{prop}
\begin{proof}
  Let $Y=Y_{1}\cup\dots \cup Y_{r}$ be the
  decomposition of $Y$ into irreducible components. For $I\subset
  \{1,\dots,r\}$ we write $Y_{I}=\bigcap _{i\in I}Y_{i}$. Then there
  is an exact sequence
  \begin{equation}
    \label{eq:32}
    0\to \Sigma_{Y}E^{\ast}_{X} \to E^{\ast}_{X}\to 
    \bigoplus _{|I|=1}E^{\ast}_{Y_{I}} \to \bigoplus
    _{|I|=2}E^{\ast}_{Y_{I}} \to \cdots
  \end{equation}
Moreover this sequence remains exact after taking the $F^{p}$
subcomplex at each degree. Since the total complex of the sequence
\begin{displaymath}
  \bigoplus _{|I|=1}E^{\ast}_{Y_{I}} \to \bigoplus
  _{|I|=2}E^{\ast}_{Y_{I}} \to \cdots
  \to \bigoplus _{|I|=k}E^{\ast}_{Y_{I}}\to \cdots
\end{displaymath}
is the de Rham part of a mixed Hodge  complex that computes 
$H^{\ast}(Y)$, the result follows. 
\end{proof}

Let $A\subset X$ be a normal crossing divisor and $E^{\ast}_{X}(\log
A)$ the complex of differential forms on $X$ with logarithmic
singularities along $A$ introduced in \cite{Burgos:CDc}. This is also
a Dolbeault complex, so it has a real structure and a Hodge
filtration. Although in this case it also has a weight filtration. We
will use the shorthand
\begin{displaymath}
  E^{\ast}_{X}(\log A;a)\coloneqq E^{\ast}_{X}(\log A)(a)
\end{displaymath}
The following result is proved in
\cite{Burgos:CDc}.

\begin{prop}\label{prop:2} Assume again that $X$ is projective and that
  $A\subset X$ is a normal crossing
  divisor.  Then, $((E^{\ast}_{X}(\log A)_{\R},W),(E^{\ast}_{X}(\log
  A),W,F))$ is a mixed Hodge complex computing the real mixed Hodge
  structure $H^{\ast}(X\setminus A)$.
\end{prop}

Proposition \ref{prop:2} can be applied to general subvarieties of $X$
by using resolution of singularities. In order to get a complex that
does not depend on the choice of a particular resolution one can take
a limit with respect to all possible resolutions. In the sequel we
will have a mixed situation where there is already present a normal
crossing divisor $A$ that we want to preserve as much as possible and an
arbitrary subvariety $Z$ that meets $A$ properly. In this case we use
the following notation
\begin{equation}\label{eq:11}
  E^{\ast}_{X}(\log A\cup Z)
  = \lim_{\substack{\longrightarrow\\\widetilde X}}E^{\ast}_{X}(\log A')
\end{equation}
where the limit runs over all proper modifications $\pi \colon
\widetilde X\to X$ such that $A'=\pi ^{-1}(A\cup Z)$ is a normal
crossing divisor and that the restriction $\pi |_{\widetilde
  X\setminus \pi ^{-1}(Z)}\colon \widetilde
  X\setminus \pi ^{-1}(Z)\to
  X\setminus Z$ is an isomorphism. In other words we are
  only allowed to make blow ups supported on $Z$. The complex
  $E^{\ast}_{X}(\log A\cup Z)$ inherits a real structure, a Hodge
  filtration and a weight filtration. Proposition
  \ref{prop:2} easily implies the next result.

  \begin{cor}\label{cor:1}
    Assume that $X$ is projective, that
  $A\subset X$ is a normal crossing 
  divisor and that $Z\subset X$ is a closed subvariety. Then,
  \begin{displaymath}
((E^{\ast}_{X}(\log A\cup Z)_{\R},W),(E^{\ast}_{X}(\log
  A\cup Z),W,F))
\end{displaymath}
is a mixed Hodge complex computing the real mixed Hodge
  structure $H^{\ast}(X\setminus A\cup Z)$.   
\end{cor}

We can now combine Proposition \ref{prop:1} and Corollary
\ref{cor:1}.

\begin{df}\label{def:2}
Let $A$ be a normal crossing divisor of $X$. Let $Z, W\subset X$ be
closed subvarieties such that no component of $W$ is contained in
$A\cup Z$. Let $\iota\colon  \widetilde W \to X$ be a resolution of
singularities of $W\setminus A\cup Z$. Then we write
\begin{equation*}
  \Sigma _{W}E^{\ast}_{X}(\log A\cup Z)=
  \{\omega \in E^{\ast}_{X}(\log A\cup Z)\mid
  \iota ^{\ast} \omega =0\}
  \subset E^{\ast}_{X}(\log A\cup Z).
\end{equation*}
We again use the shorthand, for $a\in \Z$, 
\begin{displaymath}
  \Sigma _{W}E^{\ast}_{X}(\log A\cup Z;a)\coloneqq \Sigma
  _{W}E^{\ast}_{X}(\log A\cup Z)(a). 
\end{displaymath}
\end{df}
The complex $\Sigma _{W}E^{\ast}_{X}(\log A\cup Z)$ has
a real structure and a Hodge filtration but not a weight
filtration.

\begin{thm}\label{thm:1}
  Assume that $X$ is projective and that $A$, $B$ are
  divisors without common components such that $A\cup B$ is
  a normal crossing divisor. Let $W$ be a smooth subvariety
  intersecting transversely all intersections of components of $A\cup B$
  and such that all intersections between components of $A\cup B\cup
  W$ are smooth and irreducible. Let $Z$ be a closed subvariety.
  Then, there is a mixed Hodge complex $K$ that computes the relative
  cohomology groups $H^{\ast}(X\setminus (B\cup Z), (A\cup W)\setminus
  (B\cup Z))$, a quasi-isomorphism
  \begin{displaymath}
    \Sigma_{A\cup W}E^{\ast}_{X}(\log B\cup Z)_{\R}\longrightarrow K_{\R}
  \end{displaymath}
  and a compatible filtered quasi-isomorphism 
  \begin{displaymath}
    (\Sigma_{A\cup W}E^{\ast}_{X}(\log B\cup Z)_{\C},F)\longrightarrow (K_{\C},F).
  \end{displaymath}
\end{thm}
\begin{proof}
  The proof is essentially the same as the proof of Proposition
  \ref{prop:1} by using Corollary \ref{cor:1} on each intersection
  among components of $A\cup W$.
\end{proof}

\begin{cor}\label{cor:2} With the hypothesis of Theorem \ref{thm:1}.
  For each $a,r\in \Z$ there is a canonical isomorphism
  \begin{displaymath}
    H^{r}(\Sigma_{A\cup W}E^{\ast}_{X}(\log B\cup Z,a))\xrightarrow{\cong}
    H^{r}(X\setminus(B\cup Z),(A\cup W )\setminus (B\cup Z); a)_{\C} 
  \end{displaymath}
  compatible with the Hodge filtration and the real
  structure. Moreover, the spectral sequence associated to the Hodge
  filtration $F$ degenerates at the term $E_{1}$. Therefore the
  differential $d$ in the complex $\Sigma_{Y}E^{\ast}_{X}$ is strict
  with respect to the filtration $F$.
\end{cor}
\begin{proof}
  The first statement is a direct  consequence of Theorem
  \ref{thm:1}. The second statement is also consequence of
  Theorem \ref{thm:1} and standard properties of mixed Hodge
  complexes.  
\end{proof}

Finally we explain how to use differential forms with zeros and poles
to make effective the duality of Lemma \ref{lemm:1} in the normal
crossing case.

\begin{prop}\label{prop:3}
  Assume that $X$ is projective. Let $A$ and $B$ be two divisors of $X$
  without common components such that $A\cup B$ is a normal crossing
  divisor. For $\eta\in \Sigma _{A}E^{r}_{X}(\log B)$ and $\omega \in
  \Sigma _{B}E^{2d-r}_{X}(\log A)$, the top differential form
  $\eta\wedge \omega $ is locally integrable. Moreover, the pairing 
  \begin{displaymath}
    H^{r}(X\setminus B,A)_{\C}\otimes H^{2d-r}(X\setminus A,B)_{\C}
    \longrightarrow \R(-d)_{\C}
  \end{displaymath}
  given, for $\eta$ and
  $\omega $ closed, by
  \begin{displaymath}
    \langle \eta,\omega \rangle =[\eta\wedge \omega ](1)=
    \frac{1}{(2\pi i)^{d}}\int_{X}\eta\wedge \omega
  \end{displaymath}
  is a perfect pairing inducing an isomorphism as in Lemma
  \ref{lemm:3}. 
\end{prop}

\subsection{Currents on a subvariety}
\label{sec:currents-subvariety}

Let $Z$ be a subvariety of $X$. We denote by $\Sigma
_{Z}E^{\ast}_{X,c}\subset \Sigma
_{Z}E^{\ast}_{X}$  the subspace of differential forms with compact
support on $X$ that vanish on $Z$ and we write
\begin{displaymath}
  E^{\prime -n}_{X,Z}=\{T\in E^{\prime -n}_{X}\mid
  T(\omega )=0,\ \forall \omega \in
  \Sigma_{Z}E^{\ast}_{X,c}\}
\end{displaymath}
The space $E^{\prime -n}_{X,Z}$ has been introduced by Bloom and Herrera
in \cite{BloomHerrera:drcas} and, in the case when $Z$
is smooth, it agrees with $E^{\prime -n}_{Z}$. This space is a
Dolbeault complex and we write 
\begin{displaymath}
  D^{\ast}_{X,Z}=E_{X,Z}^{\prime \ast}[-2d](-d),\qquad D^{\ast}_{X/Z}= D^{\ast}_{X}/D^{\ast}_{X,Z}.
\end{displaymath}
Again, the complex $D^{\ast}_{X,Z}$ has a real structure and a Hodge
filtration but not a weight filtration.

Let $A\subset X$ be a normal crossing divisor and $Z$ a closed
subvariety, write $Y=A\cup Z$. If $\omega \in E^{\ast}_{X}(\log A\cup
Z;a)$ and $\eta\in \Sigma _{Y}E^{\ast}_{X}$ then the differential form
$\omega \land \eta$ is locally integrable in any proper modification
$\widetilde X\to X$ where $\omega $ is defined. This induces a map
\begin{displaymath}
  [\cdot]\colon E^{\ast}_{X}(\log A\cup Z;a) \to D^{\ast}_{X/Y}(a)
\end{displaymath}
given by
\begin{displaymath}
  [\omega ](\eta)=\frac{1}{(2\pi i)^{d}}\int_{X}\omega \land \eta. 
\end{displaymath}

 \begin{prop}\label{prop:5}
   Let $A$, $Z$ and $Y$ be as before. Assume that $Z$ is smooth and
   that meets transversely all the strata of $A$. Then the map 
  \begin{displaymath}
    (E^{\ast}_{X}(\log A\cup Z;a),F)\longrightarrow  (D^{\ast}_{X/Y}(a),F)
  \end{displaymath}
  is a filtered quasi-isomorphism compatible with the real structure.  
\end{prop}
\begin{proof}
  The case when $A\cup Z$ is a normal crossing divisor has been proved in
  \cite[Theorem 5.44]{BurgosKramerKuehn:cacg} using the techniques
  from \cite{Fujiki:dmHs} and \cite{HerreraLieberman:rpvcs}. Let
  $\pi \colon \widetilde X\to X$ be the blow-up of $X$ along $Z$. The
  conditions on $Z$ imply that $\widetilde Y\coloneqq \pi ^{-1}(Y)$ is
  a normal crossing divisor.  
  Consider the commutative diagram with exact rows
  \begin{displaymath}
    \xymatrix{
      0\ar[r]& D^{\ast}_{\widetilde X,\widetilde Y}\ar[r]\ar[d]&
      D^{\ast}_{\widetilde X}\ar[r]\ar[d]&
      D^{\ast}_{\widetilde X/\widetilde Y}\ar[r]\ar[d]&0\\
      0\ar[r]& D^{\ast}_{X,Y}\ar[r]&
      D^{\ast}_{X}\ar[r]&
      D^{\ast}_{X/Y}\ar[r]&0.
    }
  \end{displaymath}
  The formula for the cohomology of a blow-up implies that the total
  complex associated to the diagram of complexes
  \begin{displaymath}
    \xymatrix{
      D^{\ast}_{\widetilde X,\widetilde Y}\ar[r]\ar[d]&
      D^{\ast}_{\widetilde X}\ar[d]\\
      D^{\ast}_{X,Y}\ar[r]&
      D^{\ast}_{X}}
  \end{displaymath}
  is acyclic. Even more, every subcomplex defined by the Hodge
  filtration is acyclic. This implies that the arrow
  \begin{displaymath}
    (D^{\ast}_{\widetilde X/\widetilde Y},F)
    \to
    (D^{\ast}_{X/Y},F)
  \end{displaymath}
  is a filtered quasi-isomorphism. Thus the result follows from the
  normal crossing case. 
\end{proof}
\begin{cor}
  With the hypothesis of Proposition \ref{prop:5}, for every $a,r\in \Z$, there is a canonical
  isomorphism
  \begin{displaymath}
    H^{r}(X\setminus Y;a)_{\C} = H^{r}(D^{\ast}_{X/Y}(a))_{\C} 
  \end{displaymath}
  compatible with the Hodge filtration and the real structure. 
\end{cor}

\subsection{Wave front sets}
\label{sec:wave-front-sets}

A current $T$ can be viewed as a differential form with
\emph{distribution} coefficients or as a \emph{generalized
section} of a vector bundle. As such, it has a wave front set that is
denoted by
$\WF(T)$. The theory of wave front sets of distributions is
developed in \cite{Hormander:MR1065993} chap. VIII. For the theory of
wave front 
sets of generalized sections, the reader can consult
\cite{GuilleminStenberg:MR0516965} chap. VI. Since we will work
with currents and 
hence with generalized sections of vector bundles, we will mainly follow 
\cite{Hormander:MR1065993}.

Denote he conormal bundle of $X$ minus the
zero section as
$T_{0}^{\vee}X=T^{\vee}X\setminus \{0\}$.
The wave front set
of a current $T$ is a closed conical subset of $T_{0}^{\vee}X$.
This set describes the
points and directions of the singularities of $T$ and it allows us
to define certain products and inverse images of currents. For a
concise description of the basic properties of the wave front set we
refer to \cite[\S 4]{BurgosLitcanu:SingularBC}.

Let $\caS\subset T^{\vee}_{0}X$ be a closed conical subset. We denote by
$D^{\ast}_{X;\caS}$ the space of currents on $X$ with wave front set
contained in $\caS$. Then \cite[Theorem 4.5]{BurgosLitcanu:SingularBC}
implies that

\begin{prop}\label{prop:6} Assume that $X$ is projective. Then 
  the morphisms  
  \begin{displaymath}
(E^{\ast}_{X},F)\to (D^{\ast}_{X;\caS },F)\to
  (D^{\ast}_{X},F)
\end{displaymath}
 are filtered quasi-isomorphism. 
\end{prop}

We will need an analogue of Theorem \ref{thm:1} for currents with
controlled wave front sets. Although the theory of wave front sets 
depends only of the underlying structure of differentiable manifolds
we will state the needed notations and results in the complex case.

\begin{df}\label{def:11} Let $f\colon Y\to X$ be a morphism of complex
  manifolds, and let $\caS \subset T_{0}^{\vee}X$ and $\caR\subset
  T_{0}^{\vee}$ closed conical subsets. Then we denote
  \begin{align*}
    N^{\vee}_{0}f &= \{(x,\xi ) \in T_{0}^{\vee}X\mid x=f(y),\
                    df(y)^{t}\xi  =0\},\\
    f^{\ast}\caS  &=\{(y,\eta ) \in T_{0}^{\vee}Y\mid \exists (x,\xi)\in
                \caS , \ x=f(y), \ \eta = 
                df(y)^{t}\xi \},\\
    f_{\ast}\caR &= N_{0}^{\vee}f \{(x,\xi) \in T_{0}^{\vee}X\mid \exists (y,\eta )\in
                \caR, \ x=f(y), \ \eta = 
                df(y)^{t}\xi  \}. 
  \end{align*}
  Then
  \begin{displaymath}
    f_{\ast}f^{\ast} \caS  = N_{0}^{\vee}f\cup \{(x,\xi ) \in
    T_{0}^{\vee}X\mid x=f(y),\ \exists (x,\xi ')\in \caS, df(y)^{t}\xi  =df(y)^{t}\xi' \}.
  \end{displaymath}
  Clearly, $f_{\ast}f^{\ast} \caS =f_{\ast}f^{\ast} f_{\ast}f^{\ast} \caS $.
  We call $f_{\ast}f^{\ast} \caS $ the \emph{saturation} of $\caS $ with
  respect to $f$. If $\caS =f_{\ast}f^{\ast} \caS $ we say that $\caS $ is
  \emph{saturated}. 
  If $Y$ is a smooth submanifold of $X$ and $f$ the corresponding
  closed immersion, we write $N^{\vee}_{0}Y=N^{\vee}_{0}f$.
\end{df}

The basic functoriality properties of currents and wave front are the
following (see \cite[chap. VIII, sect. 2]{Hormander:MR1065993}).

\begin{prop}\label{prop:11}
Let $f\colon Y\to X$  be a morphism of complex
  manifolds of relative dimension $e$, and let $\caS \subset T_{0}^{\vee}X$
  and $\caR\subset 
  T_{0}^{\vee}$ closed conical subsets.
  \begin{enumerate}
  \item If $T\in D^{r}_{X;\caS }$ and $N_{f}\cap \caS =\emptyset$, then there
    is a well defined pullback current $f^{\ast}T\in D^{r}_{Y;f^{\ast}
      \caS }$.
  \item If $T\in D^{r}_{Y;\caR}$, then $f_{\ast}T\in D_{X;f_{\ast}\caR}$.
  \end{enumerate}
\end{prop}
Let $\iota\colon A\hookrightarrow X$ be a smooth hypersurface and
$\caS \subset T^{\vee}_{0}X$ a closed conical subset.  
We will denote
\begin{displaymath}
  D^{\ast}_{X,A;\caS }=D^{\ast}_{X,A}\cap D^{\ast}_{X;\caS }\qquad
    D^{\ast}_{X/A;\caS }=\left.D^{\ast}_{X;\caS }\right/D^{\ast}_{X,A;\caS }.
\end{displaymath}
\begin{lem}\label{lemm:7} Let $\caR\subset T^{\vee}_{0}A$ be a closed
  conical subset. The morphism $\iota _{\ast}$ induces an isomorphism
  \begin{equation}\label{eq:102}
    \iota_{\ast}\colon D^{\ast}_{A;\caR}[-2](-1) \longrightarrow
    D^{\ast}_{X,A;\iota_{\ast}\caR}.
  \end{equation}
  Therefore, if $\caS \subset T^{\vee}_{0}X$ is saturated, we obtain an
  isomorphism 
  \begin{displaymath}
    \iota_{\ast}\colon D^{\ast}_{A;\iota^{\ast} \caS }[-2](-1) \longrightarrow
    D^{\ast}_{X,A;\caS }.
  \end{displaymath}
\end{lem}
\begin{proof}
  By Proposition \ref{prop:11}, the map \eqref{eq:102} is well
  defined. Since $A$ is smooth, by \cite{BloomHerrera:drcas} the map 
  \begin{displaymath}
    \iota_{\ast}\colon D^{\ast}_{A}[-2](-1) \longrightarrow D^{\ast}_{X,A}
  \end{displaymath}
  is an isomorphism. This implies directly that the map \eqref{eq:102}
  is injective. It follows easily from the definition of wave front
  set, that if $\WF(\iota _{\ast}T)\subset f_{\ast}\caR$ then
  $\WF(T)\subset \caR$ which implies surjectivity. 
\end{proof}

When taking the current associated to a differential form with
logarithmic singularities, it is easy to control the wave front
set. In fact,  the map $E^{\ast}_{X}(\log A)\to D^{\ast}_{X/A}$
factors as a composition
  \begin{displaymath}
    E^{\ast}_{X}(\log A)\longrightarrow
     D^{\ast}_{X/A;N^{\vee}_{0}A}\longrightarrow
    D^{\ast}_{X/A}
  \end{displaymath}

  Let $\iota' \colon B \to X$ be another smooth hypersurface such that
  $\caS \cap N^{\vee}_{0}B=\emptyset$. By Proposition \ref{prop:11} there
  is a map $(\iota')^{\ast}\colon D^{\ast}_{X;\caS }\to D^{\ast}_{B;(\iota
    ')^{\ast}\caS }$ and we define
  \begin{displaymath}
    \Sigma _{B}D^{\ast}_{X;\caS }=\ker((\iota ')^{\ast}).
  \end{displaymath}

  \begin{df}
    We say that $\caS $ and $B$ are in \emph{good position} if, for every $p\in
    B$ there is an open neighborhood $U \subset X$ of $p$ and a smooth
    retraction $r\colon U\to U\cap B$ such that
    \begin{displaymath}
      r^{\ast} ((\iota')^{\ast}\caS |_{U\cap B} )\subset \caS |_{U}.
    \end{displaymath}
  \end{df}

  \begin{lem}\label{lemm:8}
    If $\caS $ and $B$ are in good position,  then the map
    \begin{displaymath}
      (\iota')^{\ast}\colon D^{\ast}_{X;\caS }\to D^{\ast}_{B;(\iota
    ')^{\ast}\caS }
    \end{displaymath}
    is surjective.
  \end{lem}
  \begin{proof}
    By a partition of unity argument, the statement is local on
    $B$. Let $p\in B$ and $U$ and $r$ the neighborhood and smooth
    retraction that exist because $\caS $ and $B$ are in good
    position. Let $T\in D^{\ast}_{B;(\iota')^{\ast}\caS }$. Then
    \begin{displaymath}
      r^{\ast}T\in D^{\ast}_{U;r^{\ast} (\iota')^{\ast}\caS }\subset
      D^{\ast}_{U;\caS }, \text{
        and } (\iota ')^{\ast}r^{\ast}T=T
    \end{displaymath}
    proving surjectivity.
  \end{proof}

  We now put all the ingredients together.
  Let $X$ be a smooth projective complex
  variety, $\iota\colon A\hookrightarrow X$ and $\iota'\colon
  B\hookrightarrow X$ two smooth disjoint hypersurfaces of
  $X$ and $\caS \subset T_{0}^{\vee}X$ a closed conical subset that is, at
  the same time, saturated with respect to $\iota$ and in good position
  with respect to $B$. 
  We define
  \begin{displaymath}
    \Sigma _{B}D^{\ast}_{X/A;\caS } =
    \{T\in D^{\ast}_{X/A;\caS }\mid T|_{B}=0\}. 
  \end{displaymath}

\begin{thm}\label{thm:2} Let $X$, $A$, $B$ and $\caS $ be as before.
  Then the map
  \begin{equation}\label{eq:31}
    (\Sigma _{B}E^{\ast}_{X}(\log A),F)\longrightarrow
    (\Sigma _{B}D^{\ast}_{X/A;\caS },F)
  \end{equation}
  is a filtered quasi-isomorphism. 
\end{thm}
\begin{proof}
  By Lemma \ref{lemm:7}, since $\caS $ is saturated with respect to
  $\iota$, we have an isomorphism 
  \begin{displaymath}
    \iota _{\ast}\colon  D^{\ast}_{A;\iota^{\ast}\caS }[-2](-1)\longrightarrow
    D^{\ast}_{X,A;\caS }.
  \end{displaymath}
  Since $(D^{\ast}_{A;\iota ^{\ast}\caS },F)\to (D^{\ast}_{A},F)$ is a filtered
  quasi-isomorphism and the map $D^{\ast}_{A}[-2](-1)\to D^{\ast}_{X,A}$ is an
  isomorphism, we deduce that $(D^{\ast}_{X,A;\caS },F)\to
  (D^{\ast}_{X,A},F)$ is a filtered quasi-isomorphism.  
  We consider the commutative diagram with exact rows
  \begin{displaymath}
    \xymatrix{
      0\ar[r]& D^{\ast}_{X,A;\caS }\ar[r]\ar[d]&
      D^{\ast}_{X;\caS }\ar[r]\ar[d]&
      D^{\ast}_{X/A;\caS }\ar[r]\ar[d]&0\\
      0\ar[r]& D^{\ast}_{X,A}\ar[r]&
      D^{\ast}_{X}\ar[r]&
      D^{\ast}_{X/A}\ar[r]&0
    }
  \end{displaymath}
As we have discussed, the first vertical arrow is a filtered
quasi-isomorphism. By Proposition
\ref{prop:6} the second vertical arrow is also filtered quasi
isomorphism. We deduce that the third arrow also is one. Using now 
Proposition \ref{prop:5} we obtain that the map
\begin{displaymath}
  (E^{\ast}_{X}(\log A),F)\longrightarrow
  (D^{\ast}_{X/A;\caS },F)
\end{displaymath}
is a filtered quasi-isomorphism.
Consider next the commutative diagram with exact rows.
  \begin{displaymath}
    \xymatrix{      
      0\ar[r]& \Sigma _{B}E^{\ast}_{X}(\log A)\ar[d]\ar[r]&
      E^{\ast}_{X}(\log A)\ar[d]\ar[r]&
      E^{\ast}_{B}\ar[r]\ar[d]&0\\
      0\ar[r]& \Sigma _{B}D^{\ast}_{X,A;\caS }\ar[r]&
      D^{\ast}_{X,A;\caS }
      \ar[r]&
      D^{\ast}_{B;(\iota ')^{\ast}\caS }\ar[r]&0
    }
  \end{displaymath}
  Note that surjectivity of the map $D^{\ast}_{X,A;\caS }
  \to
  D^{\ast}_{B;(\iota ')^{\ast}\caS }$ is Lemma \ref{lemm:8}.  
  We already know that the second and third vertical arrows are
  filtered quasi-isomorphism, hence the first is also one, proving the
  result. 
\end{proof}

\subsection{Higher Chow groups}
\label{sec:higher-chow-groups}
We recall here the definition and main properties of the
\emph{higher Chow groups} defined by Bloch in \cite{Bloch:achK}.
Initially, they were defined using the chain complex associated to
a simplicial abelian group, but the description using the cubical complex is more user friendly to define the product structure. We stick to notations and conventions followed in \S 3 of \cite{Burgoswami:hait}.

Fix a base field $k$ and let $\P^1$ be the projective line over
$k$. Let $ \square = \P^1\setminus \{1\}\,(\cong \A^1).$
The cartesian
product $(\P^1)^{\cdot}$ has a cocubical scheme structure.
For $i=1,\dots,n$, we denote by $t_{i}\in (k\cup\{\infty\})\setminus
\{1\}$  the absolute coordinate of the $i$-th factor. Then the
coface maps
are defined as
\begin{align*}
  \delta_0^i(t_1,\dots,t_n) &= (t_1,\dots,t_{i-1},0,t_{i},\dots,t_n), \\
\delta_1^i(t_1,\dots,t_n) &= (t_1,\dots,t_{i-1},\infty,t_{i},\dots,t_n).
\end{align*}
Then, $\square^{\cdot}$ inherits a cocubical scheme structure from
that of $(\P^1)^{\cdot}$. 
An $r$-dimensional \emph{face} $F$ of $\square^n$ is any subscheme of the form
$\delta^{i_1}_{j_1}\cdots
\delta^{i_{n-r}}_{j_{n-r}}(\square^{r})$. By convention, $\square^n$ is a
face of dimension $n$. The codimension of an $r$-dimensional face of
$\square ^{n}$ is $n-r$.  

Let $X$ be an equidimensional quasi-projective scheme of
dimension $d$ over the field $k$. Let $Z^{p}(X,n)$ be the free
abelian group generated by the codimension $p$ closed irreducible
subvarieties of $X\times \square^{n}$, which intersect properly
$X\times F$ for every face $F$
of $\square^n$. We call the elements of $Z^p(X,n)$
\textit{admissible} cycles. 
The pull-back by the coface and codegeneracy maps of $\square^{\cdot}$
endow $Z^{p}(X,\cdot)$ with a cubical abelian group structure, given by
\begin{displaymath}
\begin{gathered}
  \delta ^{j}_{i}=(\delta ^{i}_{j})^{\ast},\\
  \delta= \sum_{i=1}^n \sum_{j=0,1}(-1)^{i+j}\delta_i^j.
\end{gathered}
\end{displaymath}
Note that the indexes have been raised or lowered to reflect the change from
cocubical to cubical structures.

Let
$(Z^{p}(X,*),\delta)$ be the associated chain complex  
 and consider the \emph{normalized} and \emph{refined normalized} chain complexes associated to $Z^p(X,*)$,
 \begin{eqnarray*}
Z^{p}(X,n)_{0}\coloneqq \bigcap_{i=1}^{n} \ker
\delta^1_i,\\
Z^p(X,n)_{00} \coloneqq \bigcap_{i=1}^{n} \ker \delta^1_i\cap
\bigcap_{i=2}^{n} \ker \delta^0_i.
\end{eqnarray*}
The differential of these normalized complexes are also denoted by $\delta
$. One can show that the inclusion
\begin{displaymath}
Z^p(X,n)_{00}\hookrightarrow Z^{p}(X,n)_{0}
\end{displaymath}
is a quasi-isomorphism of cubical chain complexes. An element in the above two complexes will be called a \textit{pre-cycle}, and
will be called a (higher) \textit{cycle} if it also satisfies $\delta(Z)=0$. 

\begin{df} Let $X$ be a quasi-projective equidimensional scheme  over a field
$k$. The \emph{higher Chow groups} defined by Bloch are
\begin{displaymath}
\CH^p(X,n)\coloneqq H_n(Z^{p}(X,*)_{0})\cong H_n(Z^{p}(X,*)_{00}).
\end{displaymath}
\end{df}
Since we will often come across the notion of proper intersection of
higher cycles in this paper, for the sake of easy reference, we recall
its definition.
\begin{df}\label{def7.5}
Let $X$ be a smooth quasi-projective scheme over $k$, and let
  $p,q,n,m\ge 0$ be non-negative
integers. If  $Z\in Z^{p}(X,n)$, 
$W\in Z^{q}(X,m)$, we say that $Z$ and $W$
\emph{intersect properly} if, for any
face $F$ of $\square ^{n+m}$,
\begin{displaymath}
  \codim_{X\times F}\left( \pi_{1}^{-1}|Z| \, \cap \pi_{2}^{-1}|W|\,
    \cap \, (X\times
  F)\right)\ge p+q,
\end{displaymath}
where
\begin{displaymath}
  \pi_1\colon X\times \square ^{n}\times \square^{m}\to X\times
  \square ^{n},\quad 
  \pi_2\colon X\times \square ^{n}\times \square^{m}\to X\times
  \square ^{m}
\end{displaymath}
are the projections.
\end{df}
Let $W\in Z^{q}(X,m)$ be an admissible cycle. We denote by
  $Z_{W}^{p}(X,n)\subset Z^{p}(X,n)$ the 
  subgroup generated by the codimension $p$ irreducible
  subvarieties $Z\subset X\times \square^{n}$, such that $Z$ and $W$
  intersect properly. Then it can be shown that the inclusions
  \begin{displaymath}
    Z_{W}^{p}(X,\ast)_0\hookrightarrow Z^{p}(X,\ast)_0,\qquad
    Z^{p}_{W}(X,\ast)_{00}\hookrightarrow Z^{p}(X,\ast)_{00}
\end{displaymath}
are quasi-isomorphisms.

\subsection{Survey of Deligne--Beilinson cohomology}
\label{sec-deligne-cohom}

As in \cite[Example 4.17]{Burgoswami:hait}, given a Dolbeault complex
$A$ we can associate to it a
diagram of complexes and morphisms 
  \begin{equation}\label{eq:16}
    \xymatrix{ 
      \DB_{\TW}(A,p) \ar@<5pt>[r]^{I} &
      \DB_{t}(A,p)\ar@<5pt>[r]^{H } \ar@<5pt>[l]^{E}&
      \DB(A,p)\ar@<5pt>[l]^G },      
  \end{equation}
  where the three complexes compute the Deligne cohomology of $A$ and 
  all the arrows are homotopy equivalences. The leftmost 
  complex has the advantage that, when $A$ is a Dolbeault algebra, has
  also a structure of 
  an associative and 
  graded commutative algebra. On the middle complex, we have several
  product structures, but none is at the same time graded commutative
  and associative.
  The rightmost 
  complex is the smallest one and gives a more concise description of
  Deligne cohomology but again has the disadvantage that the product
  is only associative up to homotopy.

  In particular, if $X$ is a smooth projective variety over $\C$, we
  can specialize diagram \eqref{eq:16} to the case $A=E^{\ast}_{X}$ to
  obtain a diagram 
  \begin{equation}\label{eq:30}
    \xymatrix{
      \DB_{\TW}(X,p) \ar@<5pt>[r] &
      \DB_{t}(X,p)\ar@<5pt>[r] \ar@<5pt>[l]& \DB(X,p)\ar@<5pt>[l]
    },     
  \end{equation}
  computing the real Deligne cohomology $H^\ast_{\fD}(X,\R(p))$ of
  $X$.
  We recall a few pieces of this diagram.   Denote by
  $L_{\R}=(L_{\R}^{\ast},d)$ the algebraic de Rham complex of
  $\A^{1}_{\R}$, that is,
  \begin{displaymath}
    L^{0}_{\R}=\R[\varepsilon ],\quad L^{1}_{\R}=\R[\varepsilon ]d\varepsilon, 
  \end{displaymath}
  where $\varepsilon $ is an indeterminate. For a Dolbeault complex
  $A$ we write
     \begin{equation}\label{eq:9}
     \fDTW(A,p)=\left\{\omega \in L^{\ast}_{\R}\otimes 
       A^{\ast }(p)_{\C}\,\middle |\,
       \begin{gathered}
       \omega |_{\varepsilon =0}\in  A^{\ast}(p)_{\R},\\
     \omega |_{\varepsilon =1}\in  F^{0}A^{\ast}(p)_{\C}.
   \end{gathered}
   \right\}
   \end{equation}
 and 
\begin{displaymath}
  \fD^{n}(A,p)=
  \begin{cases}\displaystyle
    A^{n-1}(p-1)_{\R}\cap
    \bigoplus_{\substack{p'+q'=n-1\\p'<p,\ q'<p}}A^{p',q'}_{\C},&\text{
      if }n < 2p,\\[10mm]
    \displaystyle
    A^{n}(p)_{\R}\cap
    \bigoplus_{\substack{p'+q'=n\\p'\ge p,\ q'\ge p}}A^{p',q'}_{\C},&\text{
      if }n \ge 2p.
  \end{cases}
\end{displaymath}
Note that, for $n<2p$ we can also write
\begin{equation}
  \label{eq:33}
  \fD^{n}(A,p)=\frac{A^{n-1}(p)_{\C}}{A^{n-1}(p)_{\R} \cap F^{0}A^{n-1}(p)_{\C}}. 
\end{equation}
We will denote by
\begin{equation}
  \label{eq:35}
  \pi _{p}\colon A^{n-1}(p)_{\C}\longrightarrow \fD^{n}(A,p)
\end{equation}
the projection map. 
Then, for $n<2p$, (see \cite{BurgosWang:hBC} paragraphs (6.1) and
(6.2)) the map $\fDTW^{n}(A,p)\to \fD^{n}(A,p)$ is given by 
\begin{equation}
  \label{eq:34}
  f(\varepsilon )\otimes \omega _{1}+g(\varepsilon )d\varepsilon
  \otimes \omega _{2}\mapsto \int_{0}^{1} g(\varepsilon
  )d\varepsilon\cdot \pi_{p} (\omega _{2}).
\end{equation}

\subsection{Goncharov regulator and higher archimedean
  height pairing} 
\label{subsec-higher-arc-pairing}

Here we give a quick revision of the cubical Goncharov regulator and
of the higher archimedean height pairing
for sake of ready reference. More details about the regulator can be
found in \cite[Section 5]{Burgoswami:hait}, more details about Green
currents and forms, in \cite[Section 6]{Burgoswami:hait} and about the
height pairing in \cite[Subsection
7.5]{Burgoswami:hait}. From now on we denote the
differential in the Thom-Whitney complex by $d_{\fD}$, to distinguish
it from the differential in the de-Rham complex.

In the paper \cite{Burgoswami:hait}, Goncharov regulator
\begin{displaymath}
  \caP\colon \CH^{p}(X,n)\longrightarrow H^{2p-n}_{\fD}(X,\R(p))
\end{displaymath}
is given by a morphism of complexes, also denoted $\caP$ 
\begin{displaymath}
  Z^{p}(X,\ast)_{0}\to
  \DB_{\TW,D}^{2p-\ast}(X,p).
\end{displaymath}
Recall the complex $L$ form Section
\ref{sec:higher-chow-groups} Let
$\lambda \in (L_{\C}\otimes E_{\P^{1}}(\log B))^{1}$ be the element given by
\begin{equation}\label{eq:48}
  \lambda =-\frac{1}{2}\left((\varepsilon +1)\otimes \frac{dt}{t}+
  (\varepsilon -1)\otimes \frac{d\bar t}{\bar t}
  +d\varepsilon\otimes \log t\bar t\right).
\end{equation}
Then $\lambda \in \DB_{\TW}^{1}(E^{\ast}_{\P^{1}}(\log B),1)$.

On $(\P^{1})^{n}\setminus B$, for $n\ge 0$, we consider the Wang forms
\begin{align*}
  W_{0}&=1\\
  W_{n}&=\pi _{1}^{\ast}\lambda \cdots \pi _{n}^{\ast}\lambda, \ n>0, 
\end{align*}
where $\pi _{i}\colon (\P^{1})^{n}\to \P^{1}$ is the projection onto
the $i$-th factor.
Clearly  $W_{n}\in \DB^{n}_{\TW}(\Sigma _{A}E^{\ast}_{(\P^{1})^{n}}(\log
B),n)$. See \cite[\S 5]{Burgoswami:hait} for the main properties of
these forms. By abuse of notation we will also denote by $W_{n}$ the
pull-back of $W_{n}$ to any variety of the form $X\times
(\P^{1})^{n}$. If $Z$ is an irreducible subvariety of $X\times
\square^{n}$ intersecting properly all the faces and $\widetilde Z$
is a resolution of singularities of the closure $\overline Z$, then
the pull-back of $W_{n}$ is locally integrable. 
Therefore, for any cycle $Z\in Z^{p}(X,\ast)_{0}$,
Writing
\begin{equation}
  \label{eq:49}
  \delta _{Z,\TW} \coloneqq 1\otimes \delta _{\overline {Z}}
  \in \mathfrak{D}_{\TW}^{2p}(D^{\ast}_{X\times (\P^{1})^{n}},p),
\end{equation}
we have a well defined current
\begin{displaymath}
\delta _{Z,\TW}\cdot W_{n}\in
\mathfrak{D}_{\TW}^{2p+n}(D^{\ast}_{X\times (\P^{1})^{n}},p+n). 
\end{displaymath}
Then, Goncharov regulator is given by
\begin{equation}\label{eq:50}
  \caP(Z)=(\pi _{X})_{\ast}(\delta _{Z,\TW}\cdot W_{n})
  \in \mathfrak{D}_{\TW}^{2p-n}(D^{\ast}_{X},p),
\end{equation}
where $\pi _{X}\colon X\times (\P^{1})^{n}\to X$ is the projection.
  
Given a cycle $Z\in Z^p(X,n)_0$, we call any current $g_Z\in
\fD^{2p-n-1}_{\TW,D}(X,p)$ a Green current for $Z$ if it satisfies.
\begin{displaymath}
\caP(Z)+d_{\fD}g_Z=[\omega_Z],\text{ for }\omega_Z\in \fD^{2p-n}_{\TW}(X,p).
\end{displaymath}
A class of Green currents is the class of a Green current in
\begin{displaymath}
\widetilde{\fD}^{2p-n-1}_{\TW,D}(X,p)\coloneqq \fD^{2p-n-1}_{\TW,
  D}(X,p)/\im d_{\fD},
\end{displaymath}
and is denoted by $\widetilde{g}_Z$. A pair $(Z,\widetilde{g}_Z)$,
where $\widetilde{g}_Z$ is a Green current for $Z$ is
called an arithmetic cycle, and is the building block to define higher arithmetic
Chow groups.

To define an intersection theory at the level of higher arithmetic
Chow groups, we need the notion of a Green form of logarithmic type
for a cycle $Z$. It acts as a bridge between the current $1\otimes
\delta_Z\in \fD^{2p}_{\TW,D}(X\times (\P^1)^n,p)$ and a smooth form
that lives in $\fD^{2p-n}_{\TW}(X,p)$, and computes the real Deligne
cohomology class $\caP(Z)$. For shorthand, in the next proposition we
denote $A\coloneqq (\P^1)^n\setminus \square^n$.
\begin{df}\label{GF}
Given a cycle $Z \in Z^p(X,n)_0$ and the pullback $|Z|_k$ of $|Z|$ in $X\times \square^k$ (see \S 6.2 of \cite{Burgoswami:hait} for exact definition of $|Z|_k$), a \emph{Green form of logarithmic
  type} for $Z$ is an $n$-tuple
\begin{displaymath}
  \fg_{Z}\coloneqq (g_n, g_{n-1},\cdots , g_0)\in \bigoplus^0_{k=n}\fD
^{2p-n+k-1}_{\TW,\log}(X\times \square ^{k}\setminus |Z|_{k},p)_0,
\end{displaymath}
Such that, if $n>0$,
\begin{enumerate}
\item\label{item:20} the equation $\delta _{Z}+d_{\fD}[g_n]=0$ holds in the complex
\begin{displaymath}
\fD^{2p}_{\TW,D, X\times (\P^1)^n/X\times A}(p)
\end{displaymath}
In other words, $g_{n}$ is a Green form for $Z$ in $X\times
  \square^{n}$.
\item\label{item:21} $(-1)^{n-k+1}\delta g_k+d_{\fD}g_{k-1}=0,\quad k=2,\cdots ,n$.
\item\label{item:22} $(-1)^{n}\delta g_1+d_{\fD}g_0\eqqcolon \omega(\fg_Z)\in
  \fD^{2p-n}_{\TW}(X,p)$. In other words, the form  $(-1)^{n}\delta g_1+d_{\fD}g_0$
  extends to a smooth form on the whole $X$. It can be shown that
  $\omega(\fg_Z)$ is closed and belongs to class $\caP(Z)$ in
  $H^{2p-n}_{\fD}(X,\R(p))$.
\end{enumerate}
If $n=0$ the previous conditions collapse to the condition
\begin{displaymath}
 \delta _{Z}+d_{\fD}[g_n]\in [\fD^{2p}_{\TW}(X,p)]. 
\end{displaymath}
If $Z\in Z^p(X,n)_{00}$ is a cycle in the refined normalized complex,
then a \emph{refined Green form} 
is defined as a Green form satisfying the stronger condition
\begin{equation}\label{eq:66}
  \fg_Z\in \bigoplus^0_{k=n}\fD
^{2p-n+k-1}_{\TW,\log}(X\times \square ^k\setminus |Z|'_{k},p)_{00},
\end{equation}
where $|Z|'_{k}=(\delta _{0}^{1})^{-1}\,\overset{n-k}{\dots}\,(\delta
_{0}^{1})^{-1}|Z|$.  
\end{df}
It can be shown that every class $\widetilde{g}_Z$ of Green currents
contains a Green form of logarithmic type (cf. \cite[propositions
6.12, 6.13]{Burgoswami:hait}).

Let $Z\in Z^p(X,n)_0$ and $W\in Z^q(X,m)_0$ be two cycles intersecting
properly in the sense of definition \ref{def7.5}. Then for choices of
classes of green currents $\widetilde{g}_Z$ and $\widetilde{g}_W$ for
$Z$ and $W$ respectively, we define the start product
\begin{df}\label{defstar}
Choosing any representative $g_{Z}$ of $\widetilde g_{Z}$ and a Green form
  $\fg_{W}=\{g'_m,\cdots, g'_0\}$ for $W$ contained in $\widetilde{g}_W$, we define the
  $\ast$-product of $\widetilde g_{Z}$ and $\widetilde g_{W}$ as
  \begin{displaymath}
    \widetilde g_{Z}\ast \widetilde g_{W}=
    \left (  (-1)^n\left(\sum^m_{j=0}(\pi_{X,\ast}\left(\delta_Z\cdot W_n\cdot g'_j\cdot W_j\right)\right)+g_Z\cdot \omega(\fg_W)\right)^{\sim},  
  \end{displaymath}
  where $\delta_Z\cdot W_n\cdot g'_j\cdot W_j$ is seen as a current in
  $X\times (\P^{1})^{n+m}$ and $\pi_X$ is the projection to $X$.
\end{df}
Of course, the $\ast$-product $\widetilde g_{Z}\ast \widetilde g_{W}$
depends on the choice of the Green currents $\widetilde g_{Z}$ and
$\widetilde g_{W}$ and not only on the cycles $Z$ and
$W$. Nevertheless, if the real regulators of $Z$ and $W$ are zero, we
can obtain an invariant from the $\ast$-product that only depends on
the cycles $Z$ and $W$. This is the higher analogue of the archimedean
component of the height pairing.  

\begin{df}\label{arcdf} Let $Z\in Z^p(X,n)_0$ and $W\in Z^q(X,m)_0$ be
  cycles intersecting properly, having real regulator classes zero,
  and $2(p+q-d-1)=n+m$. Then we can find Green currents for $Z$ and
  $W$ satisfying the conditions
  \begin{equation}\label{eq:88}
  d_{\fD}g_Z+\caP(Z)=d_{\fD}g_W+\caP(W)=0.
  \end{equation}
  and the 
  \emph{higher archimedean height pairing} is defined
  as
  \begin{displaymath}
   \langle Z, W\rangle _{\Arch}\coloneqq \left(p_{X,\ast}(g_Z\ast
   g_W)\right)^\sim \in H^1_{\DB}(\Spec(\C), \R(p+q-d)),  
  \end{displaymath}
  for any choice of Green current $g_{Z}$ for $Z$ and a Green current
  $g_{W}$ for $W$ satisfying \eqref{eq:88}. Here $p_X\colon
  X\rightarrow \Spec(\C)$ is the structural morphism. 
  \end{df}
  It can be shown (\cite[Proposition 7.20]{Burgoswami:hait}) that the
  definition is independent of the choice of Green currents $g_Z$ and
  $g_W$ satisfying condition \eqref{eq:88}.

From the fact that $\omega(\fg_W)$ has been chosen to be zero, we get 
\begin{displaymath}
\langle Z, W\rangle _{\Arch} =
(-1)^n\sum^m_{j=0}p_{\ast}\left(\delta_Z\cdot W_n\cdot g'_j\cdot
  W_j\right)^{\sim}, 
\end{displaymath}
where $p=p_{X}\circ \pi_X$. This pairing is graded commutative and
linear on both components.

\section{Oriented mixed Hodge structures and height}
\label{sec:mhs-arc}

\subsection{The height of a mixed Hodge structure}
\label{sec:height-mixed-hodge}

Let $V$ be a $\Q$-vector space. 
A mixed Hodge structure $(F,W)$ on $V$ induces a unique functorial bigrading
\cite[Theorem 2.13]{CKS:dhs}
\begin{equation}
    V_{\C } = \bigoplus_{a,b}\, I^{a,b}     \label{deligne-bigrading}
\end{equation}  
of the underlying complex vector space $V_{\C }$ such that
\begin{enumerate}
\item $F^a = \oplus_{\alpha\geq a,\beta}\, I^{\alpha,\beta}$;
\item $W_k = \oplus_{\alpha+\beta\leq k}\, I^{\alpha,\beta}$;
\item $\overline{I^{a,b}}
  \equiv I^{b,a} \mod\oplus_{\beta<b,\alpha<a}\, I^{\beta,\alpha}$.
\end{enumerate}
The $I^{a,b}$ is given by
\begin{equation}
  \label{eq:72}
  I^{a,b}=F^{a}\cap W_{a+b}\cap \left(\overline{F^{b}}\cap W_{a+b} +
    \overline{U^{b-1}_{a+b-2}}\right), 
\end{equation}
where
\begin{displaymath}
  U^{r}_{s}=\sum_{j\ge 0} F^{r-j}\cap W_{s-j}.
\end{displaymath}
\begin{df}
The bigrading \eqref{deligne-bigrading} will be called the
\emph{Deligne bigrading} of $(F,W)$. The associated semisimple endomorphism
$Y = Y_{(F,W)}$ of $V_{\C }$ which acts as multiplication by $p+q$ on
$I^{p,q}$ will be called the \emph{Deligne grading} of $(F,W)$.  
\end{df}
 We
will denote by $\Pi _{k}$ the projector over
$\Gr^{W}_{k}V_{\C}=\bigoplus_{a+b=k}I^{a,b}$ and $\Pi _{a,b}$ 
the projector over $I^{a,b}$. So, for instance, $\Pi_{k}$ is the
composition 
\begin{displaymath}
  V_{\C}\longrightarrow \Gr^{W}_{k}V_{\C} \hookrightarrow V_{\C}.
\end{displaymath}
Moreover the semisimple endomorphism $Y$ is given by
\begin{equation}\label{eq:37}
  Y=\sum _{k\in \Z} k \Pi _{k}.
\end{equation}
Let
\begin{equation}
     \gl(V_{\C})^{a,b} = \{\,\alpha\in \gl(V_{\C})
        \mid \alpha(I^{c,d})\subseteq I^{a+c,b+d}\,\}
     \label{gl-hodge-comp}
\end{equation}
be the Hodge decomposition of $\gl(V)$ and define
\begin{equation}
     \Lambda^{-1,-1} = \bigoplus_{a<0,b<0}\, \gl(V_{\C })^{a,b}.
     \label{lambda-def}
\end{equation}
Then, $\overline{\Lambda^{-1,-1}} = \Lambda^{-1,-1}$
\cite[Eq.~2.19]{CKS:dhs}.
For an element $\lambda \in \gl(V_{\C})$ we will denote $\lambda =\sum
\lambda ^{a,b}$ its decomposition into Hodge components. 

There exists a unique
real element $\delta = \delta_{(F,W)}\in\Lambda^{-1,-1}$ such that 
\begin{equation}
       \overline{Y_{(F,W)}} = e^{-2i\delta}\cdot Y_{(F,W)}  \label{delta-def}
\end{equation}
where $g\cdot \alpha \coloneqq \Ad(g)\alpha$ denotes the adjoint action of
$\GL(V_{\C })$
on $\gl(V_{\C })$, \cite[Prop 2.20]{CKS:dhs}. The element
$\delta$ defined by \eqref{delta-def}
will be called the \emph{Deligne splitting} of $(F,W)$.

For an element $g\in \GL(V_{\C})$ we will denote by $g\cdot F$ the
filtration given by $(g\cdot F)^{p}V_{\C}=g(F^{p}V_{\C})$. In general if
$(F,W)$ is a mixed Hodge structure on $V$, the pair of filtrations
$(g\cdot F,W)$ do not form a mixed Hodge structure.

\begin{lem}[{\cite[Lemma 4.11]{Pearlstein:vmhshf}}]
  \label{lambda-equiv} Let $(F,W)$ be a 
  mixed Hodge structure on $V$ and $\Lambda^{-1,-1}$ be the associated
  subalgebra 
  \eqref{lambda-def}.  Then, $\lambda\in\Lambda^{-1,-1}$ implies that
   $(e^{\lambda}\cdot F,W)$ is a mixed Hodge structure on $V$ and that
   \begin{displaymath}
     I^{p,q}_{(e^{\lambda}\cdot F,W)} = e^{\lambda}(I^{p,q}_{(F,W)}).
   \end{displaymath}
\end{lem}    

\par A choice of graded-polarization of $(F,W)$ determines a hermitian inner
product on $V_{\C }$ by declaring the bigrading $\oplus_{a,b}\, I^{a,b}$
to be orthogonal and defining the inner product on $I^{a,b}$ using the
isomorphism $I^{a,b}\cong H^{a,b}\Gr ^W_{a+b}$ and the standard Hodge inner
product on $\Gr ^W_{a+b}$.  In this way, we can attach a collection of heights
to $(F,W)$ via the norms of the Hodge components $\delta^{a,b}$ of $\delta$
\cite[\S 5.1]{pearlstein:sl2}.  To attach a signed height to $(F,W)$,
we need a notion of orientation.

\begin{df}\label{signed-height} Given a mixed Hodge structure
$H=(F,W)$ on $V$, define
\begin{displaymath}
  \max(H) = \max\{k\mid \Gr ^W_k(V)\neq 0\},\quad
     \min(H) = \min\{k\mid \Gr ^W_k(V)\neq 0\}.
   \end{displaymath}
   and define the \emph{length} of $H$ as
   \begin{displaymath}
     \ell(H) = \max(H)-\min(H).
   \end{displaymath}
We say that $H$ is \emph{oriented} if $\Gr ^W_{\max(H)}(V)$ and $\Gr
^W_{\min(H)}(V)$ are 
both of rank 1.  This implies that $\max(H)$ and $\min(H)$ are both
even and, writing $a=\max(H)/2$  and $c=\min(H)/2$, that 
\begin{equation}
  \label{eq:36}
  \Gr ^W_{\max(H)}(V)\cong\Q(-a),\quad
  \Gr ^W_{\min(H)}(V)\cong\Q(-c).
\end{equation}
If $H$ is oriented, an \emph{orientation} of $H$ consists of a
choice of Betti generators $\bfone_{H}$ of $\Gr ^W_{\max(H)}(V)$ and
$\bfone_{H}^{\vee}$ of
$\Gr ^W_{\min(H)}(V)$. Equivalently, an orientation is a choice of the
isomorphisms \eqref{eq:36}. 
Given an orientation of $H$ we define a signed height by the formula:
\begin{equation}
  \delta^{r,r}_H(e) = \Ht(H)e^{\vee},\qquad r = -\ell(H)/2,
  \label{signed-ht-def}
\end{equation}
where $e$ is the  element of $I^{a,a}\subset V_{\C }$ which projects to
$\bfone_{H}\in \Gr ^W_{\max(H)}(V)$ and $e^{\vee}$ is the image of
$\bfone_{H}^{\vee}$ in $W_{\min(H)}V_{\C}$. 
\end{df}

\begin{rmk} The height functions considered above only depend on the
underlying $\R $-mixed Hodge structure.    
\end{rmk}

\begin{df}\label{def:8}
  Let $H$ be an oriented mixed Hodge structure on $V$. We say that $H$
  is a \emph{generalized biextension} if $H$ has at most three non
  trivial weights.   
\end{df}
Therefore, if $H$ is a generalized biextension, there are three
integers $2a>b>2c$, and a pure
Hodge structure $H_{b}$ of weight $b$ such that
\begin{displaymath}
  \Gr_{k}^{W}(V)=
  \begin{cases}
    \Q(-a),&\text{ if }k=2a,\\
    H_{b},&\text{ if }k=b,\\
    \Q(-c),&\text{ if }k=2c,\\
    0,&\text{ otherwise.}
  \end{cases}
\end{displaymath}
Note that $H_{b}$ may be zero.

\begin{lem}\label{main-ht-formula} Let $H=(F,W)$ be a generalized
  biextension and $a,b,c$ as before.   
  Let $e\in I^{a,a}$ be the unique element that maps to the
  generator $\bfone_{H}$ and $e^{\vee}$ the image of
  $\bfone_{H}^{\vee}$ in $I^{c,c}$. Then,
\begin{displaymath}
     \Ht(H)e^{\vee}
     =\frac{1}{2}\im\left(\Pi _{2c}(e-\bar e)\right).   
\end{displaymath}
\end{lem}
\begin{proof} Write $k_{1}=2a$, $k_{2}=b$ and $k_{3}=2c$ for the
  different weights of $H$ and let $Y = Y_{(F,W)}$ and $\delta = \delta_{(F,W)}$. Since
  $\delta \in \Lambda ^{1,1}$, there is a decomposition
  $\delta = \delta _{1}+\delta _{2}+\delta _{3}$, with
  \begin{displaymath}
    \delta _{1} = \Pi _{k_{2}}\circ \delta \circ \Pi _{k_{1}},\quad
    \delta _{2} = \Pi _{k_{3}}\circ \delta \circ \Pi _{k_{2}},\quad
    \delta _{3} = \Pi _{k_{3}}\circ \delta \circ \Pi _{k_{1}}.
  \end{displaymath}
The decomposition \eqref{eq:37} and the fact that the projectors $\Pi _{k}$ are
orthogonal imply
\begin{displaymath}
  [Y,\delta _{1}]=(k_{2}-k_{1})\delta _{1},\quad
  [Y,\delta _{2}]=(k_{3}-k_{2})\delta _{2},\quad
  [Y,\delta _{3}]=(k_{3}-k_{1})\delta _{3}. 
\end{displaymath}
 In particular,
\begin{align*}
  [\delta ,[\delta,Y]]
  &=[\delta _{1}+\delta _{2}+\delta _{3},(k_{1}-k_{2})\delta
    _{1}+(k_{2}-k_{3})\delta _{2}+(k_{1}-k_{3})\delta _{3}]\\
  & = (k_{1}+k_{3}-2k_{2})\delta _{2}\circ \delta _{1}.    
\end{align*}
Therefore,
\begin{multline}
      \overline Y = e^{-2i\delta}\cdot Y
      = Y -2i[\delta,Y] -2[\delta ,[\delta ,Y]]=\\
      Y -2i((k_{1}-k_{2})\delta
      _{1}+(k_{2}-k_{3})\delta _{2}+(k_{1}-k_{3})\delta _{3})
      -2(k_{1}+k_{3}-2k_{2})\delta _{2}\circ \delta _{1}.  \label{bar-Y-eq}
\end{multline}
Since $e\in I^{a,a}$ is a lift of  $\bfone(-a)_{\Q}\in \Q(-a)_{\Q}$ and 
$\overline{ \bfone(-a)}_{\Q}=\bfone(-a)_{\Q}$, we can write 
\begin{equation}
     \bar e = e + a_{k_{2}} + a_{k_{3}} ,  \label{bar-e0-eq}
\end{equation}
where $Y(a_j) = ja_j$. 
We now compute
\begin{multline}\label{eq:38}
  \overline Y(e) = \overline{Y(\bar e)}
  = \overline{Y(e + a_{k_{2}} + a_{k_{3}})}\\
  = k_{1}\bar e+k_{2}\bar a_{k_{2}} + k_{3}\bar a_{k_{3}}
  = k_{1}e +k_{1}a_{k_{2}}+k_{1}a_{k_{3}}+k_{2}\bar a_{k_{2}} + k_{3}\bar a_{k_{3}}.
\end{multline}
On the other hand, by equation \eqref{bar-Y-eq},
\begin{equation}
  \label{eq:39}
  \overline Y(e)=k_{1}e-2i(k_{1}-k_{2})\delta _{1}(e)
  -2i(k_{1}-k_{3})\delta _{3}(e)-2(k_{1}+k_{3}-2k_{2})\delta
  _{2}(\delta _{1}(e)). 
\end{equation}
By \eqref{bar-e0-eq} we deduce
\begin{equation}\label{eq:42}
  \bar a_{k_{2}}=-a_{k_{2}}-a_{k_{3}}-\bar a_{k_{3}}.
\end{equation}
By equations \eqref{eq:38}, \eqref{eq:39} and \eqref{eq:42} and
using the splitting, we deduce the equations
\begin{align}
  (k_{1}-k_{2})a_{k_{2}}
  &=-2i(k_{1}-k_{2})\delta _{1}(e),\label{eq:44}\\
    (k_{1}-k_{2})a_{k_{3}}+(k_{3}-k_{2})\bar a_{k_{3}}
  &=
    2i(k_{3}-k_{1})\delta _{3}(e)-2(k_{1}+k_{3}-2k_{2})\delta
    _{2}(\delta_{1}(e)). \label{eq:45}
\end{align}
From equation \eqref{eq:44}, taking into
account that $k_{1}-k_{2}\not = 0$, we obtain
  \begin{equation}
    \delta _{1}(e) = \frac{i}{2}a_{k_{2}}.\label{eq:44'}
\end{equation}
Applying $\delta _{2}$ to equation \eqref{eq:44'} we get
\begin{equation}\label{eq:47}
  \delta _{2}(a_{k_{2}})=-2i\delta _{2}(\delta _{1}(e)).
\end{equation}
Computing $\overline Y(a_{k_{2}})$ in two ways as we have done with
$\overline Y(e)$ yields the equation
\begin{equation}
  \label{eq:46}
  -2i\delta _{2}(a_{k_2})=a_{k_3}+\bar a_{k_3}.
\end{equation}
Combining equations \eqref{eq:45}, \eqref{eq:47} and \eqref{eq:46}
gives
\begin{displaymath}
  \delta _{3}(e)=\frac{-1}{2}\frac{a_{k_{3}}-\bar a_{k_{3}}}{2i},
\end{displaymath}
which is equivalent to the lemma.
\end{proof}

\subsection{Some ancillary results}
We next study the effect of a morphism of mixed Hodge structures on
the height we have defined. To this end we first recall the compatibility of the
Deligne splitting with morphism of mixed Hodge structures.

\begin{lem} Let $A$ and $B$ be mixed Hodge structures with Deligne
splittings $\delta_A$ and $\delta_B$ respectively.  Let $f:A\to B$
be a morphism of mixed Hodge structures.  Then,
$f\circ\delta_A = \delta_B\circ f$.
\end{lem}
\begin{proof}  By \cite[Prop 2.20]{CKS:dhs} if $C$ is a mixed Hodge structure
then $\delta_C$ commutes with all $(r,r)$-morphisms of $C$.  Let
$C=A\oplus B$ and observe that $g(a,b) = (a,b+f(a))$ is a morphism of $C$.
Using the block structure of $\gl(C)=\gl(A\oplus B)$ it follows immediately
from \eqref{delta-def} that $\delta_C(a,b) = (\delta_A(a),\delta_B(b))$.
Writing out the $g\circ\delta_C = \delta_C\circ g$ shows that 
$f\circ\delta_A=\delta_B\circ f$.
\end{proof}

\begin{prop}\label{prop:biextension-functoriality}
Let $A$ and $B$ be oriented mixed Hodge structures
such that $\max(A)=\max(B)$ and $\min(A)=\min(B)$.  Let $f\colon A\to B$ be a
morphism of mixed Hodge structures which is injective on $\Gr ^W_{\max(A)}$
and $\Gr ^W_{\min(A)}$.  Then,
\begin{displaymath}
     \text{\rm ht}(A) d_{min}(f) = \text{\rm ht}(B) d_{max}(f)  ,
\end{displaymath}
where $f(\bfone_A) = d_{max}(f) \bfone_B$ and $f(\bfone_A^{\vee}) =
d_{min}(f) \bfone_B^{\vee}$. 
\end{prop}
\begin{proof} Let $e_A$ be a lift of $\bfone_A$ and $e_A^{\vee}$ the
  image of $\bfone_A^{\vee}$. Then, $f(e_A)=d_{max}(f)e_B$ where $e_B$
  is a lift of $\bfone_B$.
Likewise, $f(e_A^{\vee}) = d_{min}(f)e_B^{\vee}$ where $e_B^{\vee}$ is
the image of to $\bfone_B^{\vee}$.   Moreover, since $f$ is of type
$(0,0)$, then 
$f\circ\delta_A = \delta_B\circ f$  implies that 
$    f\circ\delta_A^{r,r} = \delta_B^{r,r}\circ f
$    
for any $r$.  Setting $r=(\min(A)-\max(A))/2$ it follows that
\begin{equation*}
\begin{matrix}
    f\circ\delta_A^{r,r}(e_A) & = & \delta_B^{r,r}\circ f(e_A) \\
    \parallel & & \parallel \\
    f(\text{\rm ht}(A)e_A^{\vee}) & = & \delta_B^{r,r}(d_{max}(f)e_B) \\
    \parallel & & \parallel \\
    \text{\rm ht}(A) d_{min}(f) e_B^{\vee} & = &
    \text{\rm ht}(B) d_{max}(f) e_B^{\vee}.
\end{matrix}
\end{equation*}
\end{proof}
\begin{ex}\label{example:usual-height-correspondence}
We put Proposition \ref{prop:biextension-functoriality} in practice
for usual cycles. Let $X$ and $Y$ be smooth projective varieties of
dimensions $d_X$ and $d_Y$ respectively. Let $Z\in Z^p_{\hom}(X)$,
$W\in Z^q_{\hom}(Y)$ and $\Gamma\in Z^{d_X+r}(X\times Y)$ be a
correspondence of \textit{degree} $r$, such that $p+q+r=d_Y+1$. We
assume that the pullbacks of $Z$ and $W$ intersect $\Gamma$ properly,
so that $\Gamma_\ast(Z)$ and $\Gamma^\ast(W)\coloneqq
\Gamma^t_{\ast}(W)$ are both defined at the level of cycles. Let
$B_{Z,\Gamma^\ast(W)}$ and $B_{\Gamma_{\ast}(Z),W}$ be oriented
biextensions as defined by Hain in \cite{Hain:Height}, of graded
weights $0,-1,-2$. One can show that $\Gamma$ defines a morphism of
Hodge structures between these biextensions
\begin{displaymath}
\Gamma_{Z,W}\colon B_{Z, \Gamma^\ast(W)}\rightarrow B_{\Gamma_\ast(Z), W},
\end{displaymath}
with $d_{max}(\Gamma_{Z,W})=d_{min}(\Gamma_{Z,W})=1$. Hence we get
\begin{displaymath}
\Ht(B_{Z, \Gamma^\ast(W)})=\Ht(B_{\Gamma_\ast(Z), W}).
\end{displaymath}
\end{ex}

For later use, we record the following:

\begin{lem}\label{rescale-by-N}  Let $N$ be a $(-1,-1)$-morphism of a mixed
Hodge structure $(F,W)$.  Then, 
$
   \delta_{(e^{tN}\cdot F,W)} = \delta_{(F,W)} + \im(t)N
$.
\end{lem}
\begin{proof}  By \cite[Prop. 2.20]{CKS:dhs}, $N$ and $\delta_{(F,W)}$ commute.
Therefore, using Lemma \eqref{lambda-equiv}:
\begin{align*}
    e^{(\bar t-t)N - 2i\delta_{(F,W)}}\cdot Y_{(e^{tN}\cdot F,W)}
    &= e^{\bar t N}e^{-2i\delta_{(F,W)}}e^{-tN}\cdot Y_{(e^{tN}\cdot F,W)}  \\
    &= e^{\bar t N}e^{-2i\delta_{(F,W)}}e^{-tN}e^{tN}\cdot Y_{(F,W)} \\
    &= e^{\bar t N}e^{-2i\delta_{(F,W)}}\cdot Y_{(F,W)} \\
    &= e^{\bar t N}\cdot \overline{Y_{(F,W)}} \\ 
    &= \overline{e^{t N}\cdot Y_{(F,W)}} \\
    &= \overline{Y_{(e^{tN}\cdot F,W)}}.
\end{align*}
Accordingly, by \eqref{delta-def} 
$-2i\delta_{(e^{tN}\cdot F,W)} = (\bar t - t)N -2i\delta_{(F,W)}$ which implies
the stated formula after dividing by $-2i$.
\end{proof}

\begin{cor}\label{ht-reparam} Let $N$ be a $(-1,-1)$-morphism of a
mixed Hodge structure $(F,W)$ and 
$r=(\min(F,W)-\max(F,W))/2$.  If $r<-1$  then
\begin{displaymath}
     \Ht(e^{tN}\cdot F,W) = \Ht(F,W)
   \end{displaymath}
for all $t\in\C $.
\end{cor}     
\begin{proof} By Lemma~\eqref{rescale-by-N},
$\delta_{(e^{tN}\cdot F,W)}^{r,r} = \delta_{(F,W)}^{r,r}$ and hence the two
mixed Hodge structures have the same height.
\end{proof}

\subsection{Dual of a mixed Hodge structure}\label{dualmhs} A real mixed
Hodge structure $A$ induces a mixed Hodge structure $A^{\ast}$ on the 
dual vector space $A_{\R}^*$ by the formula
\begin{equation}
     I^{a,b}_{A^*}
     = \{\, \lambda\in A_{\C }^* \mid \lambda(I^{c,d}) = 0,\hphantom{a}
         (c,d)\neq(-a,-b)\,\}.         \label{dual-mhs-ipq}
\end{equation}
If $\alpha\in \gl(A_{\C })$ then $\alpha^T\in \gl(A_{\C }^*)$ is the
linear map $(\alpha^T(\lambda))(v)=\lambda(\alpha(v))$ for all
$\lambda\in A_{\C }^*$ and $v\in A_{\C }$.  A short calculation
shows that if $\alpha\in \gl(A_{\C })^{a,b}$ then
$\alpha^T\in \gl(A_{\C }^*)^{a,b}$.  Tracing through
the definitions, one sees that the Deligne grading $Y_A$ of $A$ and
$Y_{A^*}$ of $A^*$ are related by the formula
\begin{equation}
        Y_{A^*} = -Y^T_A.     \label{dual-deligne-grading}
\end{equation}
It follows from equations \eqref{dual-deligne-grading} and
\eqref{delta-def} that
\begin{equation}
        \delta_{A^*} = -\delta_A^T.   \label{dual-delta}
\end{equation}
Indeed, since
$\ad(X_1^T)\cdots\ad(X_{r-1}^T)X_r^T
=(-1)^{r-1}\{\ad(X_1)\cdots\ad(X_{r-1})X_r\}^T$ it follows
that
\begin{align*}
    e^{-2i\,\ad(-\delta_A^T)} Y_{A^*}
    &= e^{-2i\,\ad(-\delta_A^T)}(-Y_A^T) \\  
    &= -\sum_{m\geq 0}\, \frac{1}{m!}(2i\,\ad(\delta_A^T))^{m}Y_A^T \\
    &= -\sum_{m\geq 0}\,
        \frac{(-2i)^{m}}{m!}((\ad(\delta_A)^{m})Y_A)^T \\ 
    &= -\left(\exp(-2i\,\ad(\delta_A))Y_A\right)^T 
     = -\overline{Y_A}^T = \overline{-Y_A^T} = \overline{Y_{A^*}}.
\end{align*}
since the operations of transpose and complex conjugation commute. Therefore,
$\delta_{A^*} = -\delta_A^T$ by \eqref{delta-def}.

Now if $H$ is a generalized biextension as defined in Definition
\ref{def:8}, then its dual $H^*$ is also a generalized biextension
with 
\begin{displaymath}
  \Gr_{k}^{W}(V^*)=
  \begin{cases}
    \Q(c),&\text{ if }k=-2c,\\
    H^*_{b},&\text{ if }k=-b,\\
    \Q(a),&\text{ if }k=-2a,\\
    0,&\text{ otherwise.}
  \end{cases}
\end{displaymath}
We have the following relation between the heights of $H$ and $H^*$:
\begin{prop}\label{prop:dual-height}
Let $H$ be a generalized biextension. Then
\begin{displaymath}
\Ht(H^*)=-\Ht(H).
\end{displaymath}
\end{prop}
\begin{proof}
  By the definition of the dual of an oriented biextension, the
  generators of $H$ ad $H^{\ast}$ satisfy
  \begin{displaymath}
    \langle \bfone_{H},\bfone_{H^{\ast}}^{\vee}\rangle=1,\qquad
    \langle \bfone_{H}^{\vee},\bfone_{H^{\ast}}\rangle=1.
  \end{displaymath}
Let $e_H$ be an element of $I^{a,a}_{H}\subset V_{\C}$ which projects
to $\bfone_{H}\in \Gr ^W_{2a}(V)$ and $e^{\vee}_H$ is the image of 
$\bfone_{H}^{\vee}$ in $W_{2c}V_{\C}$. Correspondingly, for $H^*$ we
have elements $e_{H^*}$ and $e_{H^*}^{\vee}$. These elements also
satisfy
\begin{displaymath}
  \langle e_{H^*}, e^{\vee}_H\rangle =1,\qquad
  \langle e^{\vee}_{H^*}, e_H\rangle=1.
\end{displaymath}
Also, since $\delta^{r,r}_{H^*}=-(\delta^{r,r}_H)^T$,
we get 
\begin{align*}
     \Ht(H^*)e^{\vee}_{H^*}=  \delta^{r,r}_{H^*}(e_{H^*})
    &=  -(\delta^{r,r}_H)^T(e_{H^*}).
    \end{align*}
 Hence
 \begin{displaymath}
 \Ht(H^*) = \langle -(\delta^{r,r}_H)^T(e_{H^*}), e_{H}\rangle.
 \end{displaymath}
Finally, using the action of $(\delta^{r,r}_H)^T$, we get
\begin{displaymath}
\Ht(H^*)=-\Ht(H)\langle e_{H^*}, e^{\vee}_{H}\rangle=-\Ht(H).
\end{displaymath}
\end{proof}

\section{Mixed Hodge structures associated to higher cycles.}
\label{sec: mixed hodge-higher cycles}

In this section we define extension classes for higher cycles $Z\in
Z^p(X,n)_{00}$ in the refined normalized complex. For two higher cycles $Z\in Z^p(X,n)_{00}$ and $W\in Z^q(X,m)_{00}$, with $2(p+q-d-1)=n+m$, we construct, under certain assumptions, an oriented mixed Hodge structure diagram (figure \ref{fig:oriented_diagram}) which captures both the extension related to cycle $Z$ and the dual to the extension related to $W$. In an even more special situation for $n=m=1$, this diagram defines an oriented biextension.

\subsection{Two divisors on \texorpdfstring{$(\P^{1})^n$}{P1n}} 
\label{sec:two-divisors-p1n}
\begin{df}
  On $(\P^1)^n$, we define the following divisors:
  \begin{align*}
    A &= \{(t_1,\cdots , t_n)\mid \exists i, t_i=1\},\\
    B & =\{(t_1,\cdots , t_n)\mid \exists i, t_i\in \{0, \infty\}\}.
  \end{align*}
Then $A\cup B$ is a simple normal crossing divisor. Moreover
\begin{displaymath}
  (\P^1)^n\setminus A=\square^n,\quad (\P^1)^n\setminus B=(\C^\ast)^n
\text{ and } B\cap \square^n=\partial \square^n.
\end{displaymath}
For any variety $X$ we also denote
\begin{displaymath}
  A_X\coloneqq X\times A,\qquad B_X:= X\times B.
\end{displaymath}
\end{df}

The following cohomology groups are easy to compute.
\begin{equation}\label{eq:2}
  H^r((\P^1)^n\setminus A, B) =
  \begin{cases}
    0,& \text{ if }r\not = n,\\
    \Q(0),& \text{ if }r= n.
  \end{cases}
\end{equation}
\begin{equation}\label{eq:3}
  H^r((\P^1)^n\setminus B, A) =
  \begin{cases}
    0,& \text{ if }r\not = n,\\
    \Q(-n),& \text{ if }r= n.
  \end{cases}
\end{equation}
In order to fix the isomorphism \eqref{eq:3} we choose the
generator of $H^n((\P^1)^n\setminus B, A;n)_{\Q}$, that is represented
by the differential form 
\begin{equation}
  \label{eq:6}
  (-1)^{n}\frac{dt_{1}}{t_{1}}\wedge\dots \wedge \frac{dt_{n}}{t_{n}}
  \in F^{0}\Sigma _{A}E^{n}_{(\P^{1})^{n}}(\log B;n),
\end{equation}
where $t_{i}$ is the coordinate of the $\P^{1}$ in position $i$.
This choice also fixes the isomorphism \eqref{eq:2}. The reason of the
sign $(-1)^{n}$ is to make it compatible with the normalizations
chosen in \cite{Burgoswami:hait}. See for instance Proposition
\ref{propreg} below.     
The K\"unneth formula and the computations \eqref{eq:2} and
\eqref{eq:3} produce, for $a,r\in \Z$, isomorphisms of mixed
Hodge structures
\begin{align}
  \label{eq:4}
  H^r(X\times (\P^1)^n\setminus A_X, B_X;a)
  &\cong H^{r-n}(X,a),\\
  \label{eq:5}
H^r(X\times (\P^1)^n\setminus B_X, A_X;a)&\cong H^{r-n}(X, a-n).
\end{align}
Since $A_X$ and $B_X$ are in product situation (see \cite[Lemma 6.1.1
and Remark 6.1.2]{BKV:Fihnf}), the above isomorphisms are compatible
with duality
\begin{multline*}
H^r(X\times (\P^1)^n\setminus A_X, B_X, \Q(p))
\cong \\ \left(H^{2d+2n-r}\left(X\times (\P^1)^n\setminus B_X, A_X,
    \Q(d+n-p)\right)\right)^{\vee}.
\end{multline*}
We fix the isomorphism \eqref{eq:4} using the generator \eqref{eq:6}
and Proposition \ref{prop:3}.

\begin{df}\label{def:3}
  For any $a,r,n\in \Z$, we denote by
  \begin{displaymath}
    \Psi \colon H^{r}\colon
    H^r(X\times (\P^1)^n\setminus A_X, B_X;a) \longrightarrow H^{r-n}(X,a)
  \end{displaymath}
  the isomorphism determined by the generator \eqref{eq:6}. This
  isomorphism sends the class of a closed form 
\begin{displaymath}
  \omega \in \Sigma _{B_{X}}E^{r}_{X\times (\P^1)^n}(\log A_{X})
\end{displaymath}
 to the class represented by the current
\begin{displaymath}
  (-1)^{n}(\pi _{X})_{\ast}\left [\omega \wedge
    \frac{dt_{1}}{t_{1}}\wedge\dots \wedge
    \frac{dt_{n}}{t_{n}}\right]=\left[
    \frac{(-1)^{n}}{(2\pi i)^{n}}\int_{(\P^{1})^{n}}
    \omega \wedge
    \frac{dt_{1}}{t_{1}}\wedge\dots \wedge
    \frac{dt_{n}}{t_{n}}
  \right],
\end{displaymath}
where $\pi _{X}\colon X\times (\P^{1})^{n}$ is the first projection. 
\end{df}

\subsection{The extension associated to a higher cycle}
\label{sec:extens-assoc-high}
In this section we show how to associate, to a cycle $Z\in
Z^p(X,n)_{00}$, $n\geq 1$, an
extension
\begin{displaymath}
  e_Z\in \text{Ext}^1_{\Q-\MHS}(\Q(0), H^{2p-n-1}(X;p)).
\end{displaymath}

By definition $Z$ is a codimension $p$ algebraic cycle in $X\times
(\P^1)^n\setminus A_{X}$, which intersects properly all the faces of
$B_X\setminus
(A_X\cap B_X)$. We write
\begin{displaymath}
  B_X=B_{01}\cup\cdots \cup B_{0n}\cup B_{\infty 1}\cup \cdots \cup
  B_{\infty n},
\end{displaymath}
as the decomposition of $B_{X}$ into irreducible components.

Since
$Z\in Z^p(X,n)_{00}$, we have $Z\cdot (B_{ij}\setminus (A_X\cap
B_{ij}))$ well defined. Moreover, $Z$ being a higher cycle, we have
\begin{equation}\label{eq:1}
  Z\cdot (B_{ij}\setminus (A_X\cap B_{ij}))=0,\ \forall
  i=0,\infty,\  \forall j=1,\cdots, n. 
\end{equation}
We denote by $\overline Z$ the closure of $Z$ as an algebraic cycle in
$X\times (\P^{1})^{n}$. There is a cycle class with support
\begin{displaymath}
  \cl(\overline Z)\in H^{2p}_{|\overline Z|}(X\times (\P^1)^n;p)_{\Q}
\end{displaymath}
and, by restriction, a class
\begin{displaymath}
  \cl(Z)\in H^{2p}_{|Z|}(X\times (\P^1)^n\setminus A_{X};p)_{\Q}.
\end{displaymath}
Now we have the following
\begin{prop}\label{propclass}
Under the above setting, there is a unique cycle class
\begin{displaymath}
  [Z]\in H^{2p}_{|Z|\setminus A_X}\left(X\times (\P^1)^n\setminus A_X,
    B_X\setminus A_X;p\right)_{\Q},
\end{displaymath}
that is sent to $\cl(Z)$ under the obvious map 
\begin{displaymath}
  H^{2p}_{|Z|\setminus A_X}\left(X\times (\P^1)^n\setminus A_X,
    B_X\setminus A_X; p\right)\to
  H^{2p}_{|Z|\setminus A_X}\left(X\times (\P^1)^n\setminus A_X; p\right).
\end{displaymath}
\end{prop}
\begin{proof}
  Consider the long exact sequence of relative cohomology with supports
\begin{multline*}
\cdots \rightarrow H^{2p-1}_{(|Z|\cap B_X)\setminus A_X}(B_X\setminus
A_X;p)\rightarrow H^{2p}_{|Z|\setminus A_X}(X\times
(\P^1)^n\setminus A_X, B_X\setminus A_X;p)\\
\rightarrow H^{2p}_{|Z|\setminus A_X}(X\times (\P^1)^n \setminus
A_X;p)\rightarrow H^{2p}_{(|Z|\cap B_X)\setminus A_X} (B_X\setminus
A_X;p)\rightarrow\cdots
\end{multline*} 
The proof will follow if we show

\begin{enumerate}
\item\label{it1}
$H^{2p-1}_{(|Z|\cap B_X)\setminus A_X}(B_X\setminus A_X;p)=0,$
\vspace{0.1cm}
\item\label{it2}
$\cl(Z)\mapsto 0$, in $H^{2p}_{(|Z|\cap B_X)\setminus
  A_X}(B_X\setminus A_X;p)$.
\end{enumerate}
Notice that for \ref{it1}, we cannot use semi-purity directly since
$B_X\setminus A_X$ is not smooth. Instead we use the following lemma.
\begin{lem}\label{lemIC}
  Let $D$ be a complex space that can be covered by a finite number of
  smooth closed subvarieties. That is $D=\cup^r_{i=1}D_i$, with
  $D_{i}$ Zariski closed and smooth. Put $D_I=\cap_{i\in I}D_i$ for $I\subseteq
\{1,\cdots, r\}$, assume that $D_I$ be smooth for each $I$, and let $Z$ be a Zariski closed subset such that
$Z\cap D_I$ has codimension $p$ for all $I$. Then
\begin{displaymath}
  H^k_{Z}(D;p)=0, \text{ for all }k<2p
\end{displaymath}
and the map
\begin{displaymath}
  H^{2p}_{Z}(D;p)\longrightarrow
  \bigoplus_{i=1}^{r} H^{2p}_{Z\cap D_{i}}(D_{i};p)
\end{displaymath}
is a monomorphism.
\end{lem}
\begin{proof} The Mayer-Vietoris property for closed coverings gives a   
  first quadrant spectral sequence
  \begin{displaymath}
   E_{1}^{a,b}=\bigoplus_{|I|=a+1}H^{b}_{Z\cap
     D_I}(D_{I};p)\Longrightarrow H^{a+b}_{Z}(D;p). 
  \end{displaymath}
Each
$D_I$ is smooth and $\codim(Z\cap D_I)=p$. Hence using semi-purity we
conclude that $H^{b}_{Z\cap D_I}(D_{I};p)=0$ for $b< 2p$. Since
$a\geq 0$, the first statement follows. The second statement is just
the fact that edge morphism of a spectral sequence is a monomorphism. 
\end{proof}

The first statement of Lemma \ref{lemIC}, implies directly condition \ref{it1}.

The property \eqref{eq:1} implies that the the
class $\cl(Z)$ is sent to $0$ in all the groups $H^{2p}_{|Z|\cap
  B_{ij}}(B_{ij}\setminus A_{X};p)$. Therefore condition \ref{it2}
follows from the second statement of Lemma \ref{lemIC}.
\end{proof}
\begin{lem}\label{lemcy0}
For $Z\in Z^p(X,n)_{00}$, the image of the class $[Z]$ in
\begin{displaymath}
  H^{2p}(X\times (\P^1)^n\setminus A_X, B_X;p),
\end{displaymath}
is zero.
\end{lem}
\begin{proof}
  By the isomorphism \eqref{eq:4} we know that the mixed Hodge
  structure
  \begin{displaymath}
    H^{2p}(X\times(\P^1)^n\setminus A_X, B_X;p)
    \cong H^{2p-n}(X;p)
  \end{displaymath}
  is pure of weight $-n$. 
  Since the image of the class $[Z]$ belongs to
  \begin{displaymath}
    F^0H^{2p-n}(X;p)_{\C}\cap H^{2p-n}(X;p)_{\R}.   
  \end{displaymath}
  Since in a pure Hodge structure of weight $-n<0$ this group is zero, 
  we conclude the result.
\end{proof}
There is a long exact sequence of mixed Hodge structures
\begin{multline}\label{eq:62}
  0\rightarrow
  H^{2p-n-1}(X;p)\rightarrow
  H^{2p-1}(X\times (\P^1)^n\setminus A_X\cup |Z|, B_X;p)
  \rightarrow\\
H^{2p}_{|Z|\setminus A_X}(X\times (\P^1)^n\setminus A_X, B_X;p)
\rightarrow H^{2p}(X\times (\P^1)^n\setminus A_X, B_X; p)
\rightarrow \cdots, 
\end{multline}
where the zero on the left hand side follows from
\begin{displaymath}
  H^{2p-1}_{|Z|}(X\times (\P^1)^n\setminus A_X, B_X;p)=0
  \qquad (\text{semi-purity}). 
\end{displaymath}
By Proposition \ref{propclass} and Lemma \ref{lemcy0}, 
the cycle class $[Z]$ defines a map
\begin{equation}\label{eq:55}
  \phi_Z\colon \Q(0)\longrightarrow H^{2p}_{|Z|\setminus
  A_X}(X\times (\P^1)^n\setminus A_X, B_X;p),
\end{equation}
whose image of $\phi _{Z}$
in $H^{2p}(X\times (\P^1)^n\setminus A_X, B_X;p)$ is
zero. Therefore, pulling back the above long exact sequence through $\phi_Z$, we get
an extension  
\begin{equation}\label{eq:53}
0\longrightarrow H^{2p-n-1}(X;p)\longrightarrow E_Z\longrightarrow
\Q(0)\longrightarrow 0.
\end{equation}
By abuse of notation, we also denote as 
\begin{displaymath}
  E_Z\coloneqq \left[0\rightarrow H^{2p-n-1}(X;p)\rightarrow
    E_Z\rightarrow \Q(0)\rightarrow 0\right]
\end{displaymath}
the class of this extension in $\Ext^1_{\Q-\MHS}\left(\Q(0),
  H^{2p-n-1}(X;p)\right)$. 

\subsection{Differential forms attached to the extension
  \texorpdfstring{$E_{Z}$}{ez}}\label{subsubsecdiff}
The extension $E_{Z}$ induces an extension
\begin{displaymath}
  E_{Z,\R}\in \Ext^1_{\R-\MHS}\left(\R(0),
    H^{2p-n-1}(X;p)\right). 
\end{displaymath}
For shorthand we write $H=H^{2p-n-1}(X;p)$, that is a mixed Hodge
structure pure of weight $-n$. 
Recall that there is an isomorphism
\begin{equation}\label{eq:17}
  \Ext^1_{\R-\MHS}\left(\R(0),
    H\right )\xrightarrow{\simeq}
    \frac{H_{\C}}
    {F^{0}_{\C}+H_{\R}}.
\end{equation}
This isomorphism works as follows. Let $E\in
\Ext^1_{\R-\MHS}$, so $E$ is the class of a short exact
sequence
\begin{displaymath}
  0\to H \to E \to \R(0)\to 0.
\end{displaymath}
Let $\bfone(0)$ be the canonical generator of $\R(0)$. Choose $v \in
F^{0}E$ an element that is sent to $\bfone(0)$. Then $h=(v-\overline
v)/2$ 
is sent to zero in $\R(0)$ and therefore belongs to $H$. The class of
$h$ in the quotient at the right hand side of \eqref{eq:17} does not
depend on the
choice of $v$ and represents the image of $E$ under the isomorphism
\eqref{eq:17}. In this section, given an element $h\in H_{\C}$, we
will denote by 
\begin{equation}
  \label{eq:29}
  \widetilde h \in \frac{H_{\C}}
    {F^{0}_{\C}+H_{\R}}
\end{equation}
its class in the quotient.

We will now construct several differential forms related to the
extension 
$E_{Z,\R}$ and, in particular a representative of its class. To this
end we will use the complexes of differential forms with zeros and
logarithmic poles
\begin{displaymath}
  \Sigma _{B_X}E^\ast_{X\times (\P^1)^n}(\log A_X;
  p), \text{ and } \Sigma _{B_X}E^\ast_{X\times (\P^1)^n}(\log
  A_X\cup |Z|;p).
\end{displaymath}
The relevance of these complexes is clear because, for instance the
class $[Z]$ belongs to
\begin{displaymath} 
  F^{0}H^{2p}_{|Z|}(X\times (\P^1)^n\setminus A_{X},B_{X};p)_{\C}
\end{displaymath}
And the underlying cohomology group can be computed using the simple of
the morphism of complexes
\begin{equation}\label{eq:22}
  \Sigma _{B_X}E^\ast_{X\times (\P^1)^n}(\log A_X;
p)\xrightarrow{\iota} \Sigma _{B_X}E^\ast_{X\times (\P^1)^n}(\log
A_X\cup |Z|;p).
\end{equation}

\begin{prop}\label{prop:4}
  Let $X$ and $Z$ be as in the previous section. Then there are
  differential forms
  \begin{enumerate}
  \item $\eta_{Z}\in F^{0}\Sigma _{B_X}E^{2p-1}_{X\times (\P^1)^n}(\log
  A_X\cup |Z|;p)$ such that $d\eta_{Z}=0$ so the pair $(0,\eta_{Z})$
  is a cycle in the simple $\simple(\iota )$ and the corresponding
  class satisfies
  \begin{equation}
    \label{eq:19}
    \{(0,\eta_{Z})\}=[Z]\in H^{2p}_{|Z|}(X\times (\P^1)^n\setminus A_{X},B_{X};p)_{\C}.
  \end{equation}
  Moreover, on the complex of currents $D^{\ast}_{X\times
    (\P^{1})^{n}/ A_{X}}$ there is an equality of 
  currents
  \begin{equation}
    \label{eq:18}
    d[\eta_{Z}]+\delta _{Z}=0.
  \end{equation}
\item $\theta_{Z} \in F^{-n}\Sigma _{B_X}
  E^{2p-1}_{X\times (\P^1)^n}(\log
  A_X;p)$ with $d\theta_{Z} =0$ and $\overline \theta_{Z} =-\theta_{Z}
  $. Moreover, if we denote by $\widetilde {\{\theta_{Z} \}}$ the image of
  the class $\{\theta_{Z} \}$ under the composition
  \begin{displaymath}
    H^{2p-1}(X\times (\P^{1})^{n}\setminus A_{X},B_{X};p)_{\C}
    \xrightarrow{\simeq}
    H_{\C}
    \to \frac{H_{\C}}{F^{0}H_{\C}+H_{\R}}
    =\Ext^1\left(\R(0),H\right),
  \end{displaymath}
  where we have used again the shorthand $H=H^{2p-n-1}(X;p)$, then
  \begin{equation}
    \label{eq:20}
    \widetilde {\{\theta_{Z} \}}=E_{Z,\R}.
  \end{equation}
\item $g_{Z}\in F^{-1}\cap \overline F^{-1} \Sigma
  _{B_X}E^{2p-2}_{X\times (\P^1)^n}(\log
  A_X\cup |Z|;p)$ satisfying $\overline g_{Z}=-g_{Z}$ and
  \begin{equation}
    \label{eq:21}
    d g_{Z} = \frac{1}{2}(\eta_{Z}-\overline \eta_{Z})-\theta _{Z},
  \end{equation}
  \end{enumerate}
\end{prop}
\begin{rmk}\label{rem:1}
  Before starting the proof, we recall how the notation in Definition
  \ref{def:1} works. The conditions
  \begin{displaymath}
    g_{Z}\in F^{-1}\cap \overline F^{-1} \Sigma _{B_X}E^{2p-2}_{X\times (\P^1)^n}(\log
  A_X\cup |Z|;p), \text{ and }\overline g_{Z}=-g_{Z}
  \end{displaymath}
  are equivalent to 
  \begin{displaymath}
    g_{Z}\in \Sigma _{B_X}E^{p-1,p-1}_{X\times (\P^1)^n}(\log
    A_X\cup |Z|), \text{ and }\fancyconj{g_{Z}}{\dR}=(-1)^{p-1}g_{Z}, 
  \end{displaymath}
  where $\fancyconj{g_{Z}}{\dR}$ is the original conjugation of
  differential forms.
\end{rmk}
\begin{proof}[Proof of Proposition \ref{prop:4}]
  We first note that the equality \eqref{eq:20} is a consequence of
  \eqref{eq:19} and \eqref{eq:21}. Recall the explicit construction
  the isomorphism \eqref{eq:17} at the beginning of the
  section. The mixed Hodge structure  $E_{Z}$ is a substructure of
  $H^{2p-1}(X\times (\P^{1})^{n}\setminus A_{X}\cup |Z|,B_{X};p)$.
  Condition \eqref{eq:19} implies that the class $\{\eta_{Z}\}$
  belongs to $F^{0}E_{Z,\C}$ and is a choice of the class $v$. Then
  equation \eqref{eq:21} implies that $\{\theta _{Z}\}$ agrees with
  the class $(v-\overline v)/2$, and we deduce \eqref{eq:20}.  

  The class $[Z]$ belongs to $F^{0}H^{2p}_{|Z|}(X\times
  (\P^1)^n\setminus A_{X},B_{X};p)_{\C}$, and we compute  
  the underlying cohomology group using the simple of
morphism $\iota $ in \eqref{eq:22}.  
Therefore, there should be an element $(\alpha _{1},\beta _{1})\in
F^{0}\simple(\iota)$ that represents $[Z]$.

By Lemma \ref{lemcy0} the form $\alpha _{1}$ has to be exact. Since by
Corollary \ref{cor:2}
the differential $d$ is strict with respect to the Hodge filtration
 we deduce that there is
\begin{displaymath}
  \alpha _{2}\in F^{0} E^{2p-1}_{X\times (\P^1)^n}(\log A_X;p)
\end{displaymath}
with $d \alpha _{2}=\alpha _{1}$. Writing $\beta =\beta _{1}-\alpha
_{2}$ we deduce that $[Z]$ is represented by $(0,\beta)=(\alpha
_{1},\beta _{1})-d(\alpha _{2},0)$ with 
\begin{displaymath}
  \beta\in F^0\Sigma_{B_X}E^{2p-1}_{X\times (\P^1)^n}(\log A_X\cup |Z|;p).
\end{displaymath}
Since the class $\{(0, \overline{\beta})\}$ also agrees with $[Z]$, we
get
\begin{displaymath}
  \{(0, \beta-\overline{\beta})\}=0.
\end{displaymath}
Hence
\begin{multline*}
\{\beta- \overline{\beta}\}\in W_{-1}H^{2p-1}(X\times (\P^1)^n\setminus A_X\cup |Z|, B_X;p)\\
=H^{2p-1}(X\times (\P^1)^n\setminus A_X, B_X;p)=H.
\end{multline*}
Since this last mixed Hodge structure is pure of weight $-n-1$,
we can decompose
\begin{equation}\label{eq:24}
  \{\beta- \overline{\beta}\}/2 =
  c-\overline c +t,
\end{equation}
with
\begin{displaymath}
  c\in F^{0}H_{\C},\qquad \overline c\in \overline F^{0}H_{\C},
  \qquad
  t\in F^{-n}H_{\C}
\end{displaymath}
and $\overline t = -t$. The class $c$ can be represented by a cycle
\begin{displaymath}
  \gamma \in F^{0}\Sigma_{B_X}E^{2p-1}_{X\times (\P^1)^n}(\log A_X;p).
\end{displaymath}
Hence
$\overline \gamma $ represents $\overline c$. Next choose a representative
\begin{displaymath}
  \theta _{1}\in F^{-n}\Sigma_{B_X}E^{2p-1}_{X\times (\P^1)^n}(\log A_X;p)
\end{displaymath}
of $t$. As a form in $\Sigma_{B_X}E^{2p-1}_{X\times (\P^1)^n}(\log
A_X)$, it has components of bidegree $(a,2p-1-a)$ for $a\ge p-n$. We
observe that  
$-\overline {\theta _{1}}$ also represents $t$. Hence,
there is an $u\in \Sigma_{B_X}E^{2p-2}_{X\times (\P^1)^n}(\log A_X;p)$
such that $du=\theta_{1} +\overline \theta _{1}$. Since the bidegrees
of $\theta_{1}$ and $\overline \theta _{1}$ only overlap in the range
\begin{equation}\label{eq:23}
  (p-n,p+n-1),\dots,(p+n-1,p-n)
\end{equation}
we see that some components of $du$ will kill some components of
$\theta _{1}$.
Let $F^{n-1}u$ denote
the sum of the components of $u$ of bidegree $(a,b)$ with $a\ge
p+n-1$. Then $\theta _{2}\coloneqq \theta _{1}-dF^{n-1}u$ only has
components of bidegrees in the range \eqref{eq:23}. This implies that
$\overline \theta _{2}$ belongs to $F^{-n}$

Writing $\theta_{Z}
=(\theta _{2}-\overline \theta _{2})/2$ we obtain a differential form
satisfying
\begin{displaymath}
  \theta_{Z} \in F^{-n}\Sigma _{B_X} E^{2p-1}_{X\times (\P^1)^n}(\log
  A_X;p),\quad d\theta_{Z} =0 \text{ and }\overline \theta_{Z} =-\theta_{Z} 
\end{displaymath}
and still representing $t$.

The decomposition \eqref{eq:24} implies that there is form
\begin{displaymath}
  g_1\in \Sigma_{B_X}E^{2p-2}_{X\times (\P^1)^n}(\log A_X\cup |Z|;p)
\end{displaymath}
such that
\begin{displaymath}
  dg_1=\frac{1}{2}\left((\beta-\overline{\beta})-(\gamma -\overline{\gamma })\right)-\theta_{Z}
\end{displaymath}
and $\overline{g}_1=-g_1$. We decompose $g_{1}$ in bidegrees
\begin{displaymath}
  g_{1}=g_{1}^{p-1,p-1}+F^{0}g_{1}+\overline F^{0}g_{1}.
\end{displaymath}
and define
\begin{displaymath}
  g_{Z}=g_{1}^{p-1,p-1}\quad\text{ and }\quad
  \eta_{Z}=\beta -2\gamma -2dF^{0}g_{1}.
\end{displaymath}
By construction, equation \eqref{eq:21} is satisfied. Therefore
$g_{Z}$ satisfies all the conditions of the theorem. On the other hand
\begin{displaymath}
  (0,\eta_{Z})=(0,\beta )+d(-2\gamma ,2F^{0}g_{1}),
\end{displaymath}
so $\eta_{Z}$ satisfies condition \eqref{eq:19}. As explained in the
beginning, this implies that $\theta _{Z}$ satisfies equation
\eqref{eq:20}. 

It remains to show equation \eqref{eq:18}. The argument is adapted
from \cite[Theorem 4.4]{Burgos:Gftp}. By construction of the class
$[Z]$ we see that forgetting the vanishing at $B_{X}$, the pair
$(0,\eta_{Z})$ represents the class $\cl(Z)\in
H^{2p}_{|Z|}(X\times (\P^1)^n\setminus A_{X};p)_{\C}$. Using
resolution of singularities we can construct a
smooth complex variety $\widetilde X$, a normal crossing divisor $D$
and a codimension $p$ cycle $Z'$ with $|Z'|$ smooth and intersecting
transversely all intersections of components of $D$ and a birational map $\pi
\colon \widetilde X\to X\times (\P^{1})^{n}$, such that $\pi
_{\ast}Z'=Z$, $D$ being the union of the exceptional divisor of $\pi$ and
the preimage of $A_{X}$. The cohomology
group $H^{2p}_{|Z'|}(\widetilde X\setminus D;p)_{\C}$ can be
computed as the simple of the morphism of complexes
\begin{displaymath}
  D^{\ast}_{\widetilde X/D} (p)\xrightarrow{\iota '}
  D^{\ast}_{\widetilde X/(D\cup Z')}(p).
\end{displaymath}
moreover there is a morphism of complexes $\simple(\iota )\to
\simple(\iota' )$ given by the commutative diagram
\begin{displaymath}
  \xymatrix{
    E^\ast_{X\times (\P^1)^n}(\log A_X; p)\ar[r]^-{\iota}
    \ar[d]
    & \Sigma _{B_X}E^\ast_{X\times (\P^1)^n}(\log A_X\cup |Z|;p)
    \ar[d]\\
  D^{\ast}_{\widetilde X/D} (p)\ar[r]^-{\iota '}
  & D^{\ast}_{\widetilde X/(D\cup Z')}(p).    
  }
\end{displaymath}
In the complex $\simple(\iota )$ the class $\cl(Z')$ is represented by
the pair $(\delta _{Z'},0)$. Therefore there are currents $u,v$ such that
\begin{displaymath}
  (\delta _{Z'},0)-(0,[\pi ^{\ast}\eta_{Z}])=d(u,v)=(du,u-dv).
\end{displaymath}
Hence
\begin{displaymath}
  \delta _{Z'} = du,\qquad [\pi ^{\ast}\eta_{Z}]=dv-u,
\end{displaymath}
which implies the result, thanks to the projection formula.
\end{proof}

\subsection{The class of the extension and Goncharov regulator.}
\label{sec:class-extens-gonch}
In this section we will use the form $\theta_Z $ to relate the class
of $E_{Z}$ with the Goncharov regulator $\caP(Z)$ of section
\ref{subsec-higher-arc-pairing}.

\begin{prop}\label{propreg} Let $\caP$ be the cubical Goncharov regulator
  normalized as in \cite[Definition 5.1]{Burgoswami:hait} and $\Psi $
  the isomorphism of Definition \ref{def:3}. 
  Under the isomorphism
  \begin{equation}\label{eq:27}
    H^{2p-n}_{\fD}(X,\R(p))
    \xrightarrow{\cong} 
    \frac{H^{2p-n-1}(X, \C)}{F^pH^{2p-n-1}(X, \C)
      + H^{2p-n-1}(X,\R(p))},
  \end{equation}
 the class $\caP(Z)$ is mapped to $\widetilde{\Psi (\theta _{Z})}$.
\end{prop}
\begin{proof}
  In this proof, to compute real Deligne cohomology we use the Thom-Whitney
  Deligne complex $\DB_{\TW}$ of Section \ref{sec:higher-chow-groups}
  (see \cite[Definition 4.14]{Burgoswami:hait}). This
  complex has the advantage to have a well defined graded commutative
  and associative product.
  
  From the forms constructed in Proposition \ref{prop:4} we can define
  the following Thom-Whitney versions, to complement $\delta _{Z,\TW}$
  given by equation \eqref{eq:49}.
  \begin{equation}
    \label{eq:26}
  \begin{aligned}
    g_{Z,\TW}
    & \coloneqq \epsilon\otimes \eta_Z+(1-\epsilon)\otimes
      (\eta_Z+\overline{\eta}_Z)/2+d\epsilon\otimes g_Z\\
    &\phantom{\coloneqq d\epsilon \otimes \theta _{Z}}\in
      \mathfrak{D}_{\TW}^{2p-1}(\Sigma _{B_{X}}E^{\ast}_{X\times
      (\P^1)^n}(\log 
  A_X\cup |Z|),p),\\ 
    \theta _{Z,\TW}
    &  \coloneqq d\epsilon \otimes \theta _{Z}
      \in \mathfrak{D}_{\TW}^{2p}(\Sigma _{B_{X}}E^{\ast}_{X\times (\P^1)^n}(\log
  A_X\cup |Z|),p).
  \end{aligned}    
  \end{equation}

  Equations \eqref{eq:18} and \eqref{eq:21} and the fact that
  $\overline{\delta _{Z}}=\delta _{Z}$ imply that
  \begin{equation}\label{eq:25}
    d[g_{\TW,Z}]=-\delta_{\TW,Z}+ [\theta_{\TW,Z}].
  \end{equation}

  Equations \eqref{eq:50} and \eqref{eq:25}, together with equation
  \cite[(5.7)]{Burgoswami:hait} and the fact that $g_{Z,\TW}$ vanishes
  when restricted to $B_{X}$
  imply the equality of cohomology classes
  \begin{displaymath}
    \caP(Z)=\{(\pi _{X})_{\ast}[\theta_{Z,\TW}\cdot W_{n}]\}.
  \end{displaymath}
  So we are left to compare the classes $\{(\pi _{X})_{\ast}[\theta_{Z,\TW}\wedge
  W_{n}]\}$ with $\{\Psi (\theta _{Z})\}$. To this end we will use the explicit
  description of Wang forms in \cite[Definition
  6.5]{BurgosWang:hBC}. We note that the form denoted by $W_{n}$ here
  is the form $(-1)^{n}W_{n}^{3}$ in \cite{BurgosWang:hBC}. 
  
  Using \eqref{eq:34}, the image of $\caP(Z)$ is represented by the form
  \begin{equation}\label{eq:28}
    \sum_{i=1}^{n}\int_{0}^{1}\frac{(-1)^{n}(\epsilon +1)^{i}(\epsilon
      -1)^{n-i}}{2^{n}i!(n-i)!} d\epsilon \cdot
(\pi_{X}) _{\ast}[\theta _{Z}\wedge P^{i}_{n}],
  \end{equation}
  Where $P^{i}_{n}=\sum_{\sigma \in \mathfrak{S}_{n}}(-1)^{\sigma
  }P^{i}_{n,\sigma }$ and, for a permutation $\sigma \in
  \mathfrak{S}_{n}$. 
\begin{displaymath}
P^i_{n,\sigma }=\frac{dt_{\sigma(1)}}{t_{\sigma(1)}}\wedge\cdots
\wedge \frac{dt_{\sigma(i)}}{t_{\sigma(i)}}\wedge
\frac{d\overline{t}_{\sigma(i+1)}}{\overline{t}_{\sigma(i+1)}}\wedge\cdots
\wedge \frac{d\overline{t}_{\sigma(n)}}{\overline{t}_{\sigma(n)}}. 
\end{displaymath}
We now use that
\begin{gather*}
  [P^{i}_{n,\sigma }] = -[P^{n}_{n,\sigma }] + \text{boundaries}
  +\text{currents in }B_{X},\\
  (-1)^{\sigma }P^{n}_{n,\sigma
  }=\bigwedge_{i=1}^{n}\frac{dt_{i}}{t_{i}},
\end{gather*}
that
\begin{displaymath}
  \sum_{i=0}^{n}\frac{n!(-1)^{n}(\epsilon +1)^{i}(\epsilon
    -1)^{n-i}(-1)^{n-i}}{2^{n}i!(n-i)!}
  =\frac{(-1)^{n}}{2^{n}}(\epsilon +1-(\epsilon -1))^{n}
 =(-1)^{n},
\end{displaymath}  
and that the form $\theta _{Z}$ vanishes on $B_{X}$ to deduce that the
current \eqref{eq:28} is cohomologous to
\begin{displaymath}
  (-1)^{n}(\pi _{X})_{\ast}\left[\theta _{Z}\wedge
  \bigwedge_{i=1}^{n}\frac{dt_{i}}{t_{i}}\right] =
\Psi (\theta _{Z}).
\end{displaymath}
\end{proof}

\begin{cor}\label{thetazero}
Let $Z\in Z^p(X,n)_{00}$, be a cycle such that its real regulator
class is zero. Then we can choose $g_{Z}$, $\eta_{Z}$ and $\theta
_{Z}$ as in Proposition \ref{prop:4} with the additional property
$\theta _{Z}=0$. Therefore
\begin{displaymath}
  dg_Z=\frac{1}{2}(\eta_Z-\overline{\eta}_Z).
\end{displaymath}
\end{cor}
\begin{proof}
  Let $g'_{Z}$, $\eta '_{Z}$ and $\theta '_{Z}$ a choice of forms as
  in Proposition \ref{prop:4}.  If the real regulator class of $Z$ is
  zero, then Proposition
\ref{propreg} implies that the cohomology class of $\theta_Z'$ belongs to 
\begin{displaymath}
F^0H^{2p-1}(X\times (\P^1)^n\setminus A_{X}, B_{X};p)+
H^{2p-1}(X\times (\P^1)^n\setminus A_{X}, B_{X} ;p)_{\R}. 
\end{displaymath}
Hence there exist differential forms 
\begin{align*}
  h_1&\in F^0\Sigma_{B_X}E^{2p-1}_{X\times (\P^1)^n}(\log A_X; p),\\
  h_2&\in \Sigma_{B_X}E^{2p-1}_{X\times (\P^1)^n,\R}(\log A_X; p),\\
  \gamma &\in \Sigma_{B_X}E^{2p-2}_{X\times (\P^1)^n}(\log A_X; p),
\end{align*}
with $h_{1}$ and $h_{2}$ closed, 
such that
\begin{displaymath}
  \theta' _{Z} = h_{1} + h_{2} + d\gamma.
\end{displaymath}
We write $\gamma _{1}=(\gamma
-\overline \gamma )/2$ and we decompose
\begin{displaymath}
  \gamma_{1} = \gamma_{1} ^{p-1,p-1} +F^{0}\gamma_{1} +\overline F^{0}\gamma_{1}.
\end{displaymath}
Then $\overline F^{0}\gamma_{1} = -
\overline{F^{0}\gamma_{1}}$. Moreover, since $\overline{\theta
  _{Z}'}=-\theta _{Z}'$,
\begin{displaymath}
  d \gamma _{1}^{p-1,p-1} = \theta '_{Z}- \frac{1}{2}\left((h_{1}+2 dF^{0}\gamma
    _{1})-\overline{(h_{1}+2 dF^{0}\gamma _{1})}\right). 
\end{displaymath}
Thus, if we write
\begin{displaymath}
 g _Z=g'_Z+\gamma _{1}^{p-1,p-1},\qquad \eta
_{Z}=\eta _{Z}'-h_{1}-2F^{0}\gamma _{1}\qquad 
\theta _{Z}=0, 
\end{displaymath}
then, is easy to verify that the triple $\eta _{Z},\theta _{Z},\gamma _{Z}$
satisfies the properties of Proposition \ref{prop:4}.
\end{proof}
\begin{rmk}\label{rmkreg0}
When the real regulator class of a higher cycle $Z\in
Z^p(Z,n)_{00}$ is zero, and the forms $\eta_{Z}$ and $g_{Z}$ are as in
Corollary \ref{thetazero}, then $\eta_Z=2\partial g_Z$.
\end{rmk}

\subsection{Comparison with \texorpdfstring{\cite{Burgoswami:hait}}{BGG}} This subsection
acts as a bridge between the Hodge theoretic forms obtained above, and
the higher Green forms and currents used in \cite{Burgoswami:hait}. We
will use it later to connect the higher archimedean height
pairing to the height of a mixed Hodge structure
associated to a pair of higher cycles.
We will follow the notations of \cite{Burgoswami:hait}.

For each $n$, consider the complex given by
\begin{displaymath}
  \tau \DB_{\TW,\A}^{\ast,-s}(X,p)=\tau _{\le 2p}\DB^{\ast}_{\TW}(E^{\ast}_{X\times
    (\P^{1})^{s}}(\log B),p).
\end{displaymath}
It has a cubical structure and we can form the associated refined
normalizes double complex $\tau \DB_{\TW,\A,\log}^{\ast,\ast}(X,p)_{00}$ and
the corresponding total complex
$\tau \DB_{\TW,\A,\log}^{\ast}(X,p)_{00}$. See \cite[5.2]{Burgoswami:hait}
for more details.

There is a quasi-isomorphism
\begin{displaymath}
  \tau_{\le 2p}
\DB_{\TW}^{\ast}(X,p)\hookrightarrow \tau \DB_{\TW,\A,\log}^{\ast}(X,p)_{00}
\end{displaymath}
that is given by the inclusion as the column $n=0$.

Let $Z$, $\theta _{Z,\TW}$ and $g_{Z,\TW}$ be as in the previous
section and write
\begin{displaymath}
  \underline{\theta_{Z}}=(\pi _{X})_{\ast}[\theta_{Z,\TW}\cdot W_{n}]
  \in \DB_{\TW}^{2p-n}(X,p)=\DB_{\TW,\A,\log}^{2p-n,0}(X,p)_{00}.
\end{displaymath}
In the complex $\tau \DB_{\TW,\A,\log}^{\ast}(X,p)_{00}$, the forms
$\theta _{Z,\TW}$ and $\underline{\theta_{Z}}$ are cohomologous as
both represent the class $\{\caP(Z)\}$.
Therefore we obtain an element 
\begin{displaymath}
(\alpha_n,\cdots, \alpha_0)\in \DB_{\TW,\A,\log}^{2p-n-1}(X,p)_{00},
\end{displaymath}
satisfying
\begin{equation}\label{eq:51}
(0,\cdots, \underline{\theta_Z})-(\theta_{Z,\TW},0,\cdots ,
0)=d(\alpha_n,\cdots, \alpha_0).
\end{equation}
We obtain an $n$-tuple of forms
\begin{displaymath}
  \fg_{Z}\coloneqq (g_{Z,\TW}+\alpha _{n},\dots,\alpha _{0}).
\end{displaymath}
\begin{lem}\label{BGgreen}
The $n$-tuple of forms
\begin{displaymath}
  \fg_{Z}\in \bigoplus^{0}_{i=n} \DB_{\TW}^{2p-n+i-1}(
  E^{\ast}_{X\times (\P^{1})^{i}}(\log A\cup |Z|_{i}),p)_{00}
\end{displaymath}
is a refined Green form for $Z$, as in
\cite[Definition 6.5]{Burgoswami:hait}.
\end{lem}
\begin{proof}
  Equations \eqref{eq:51} and \eqref{eq:25} when written
  componentwise, imply the conditions of
  \cite[Definition 6.5]{Burgoswami:hait}. 
\end{proof}
\begin{rmk}\label{Improvedgreen}
Notice that in the Green form $\fg_{Z}$ only the component over
$X\times (\P^{1})^{n}$ has singularities along
$|Z|$, while the rest are smooth on $X\times \square^i$, with
logarithmic singularities along $A_X$. 
\end{rmk}

After constructing a higher Green form out of $g_{Z,\TW}$ we also
construct a Green current. 
Let $\underline{g_{Z}}\coloneqq (\pi _{X})_{\ast}[g_{Z}\cdot W_n]\in
\DB^{2p-n-1}_{\TW,D}(X,p)$. Then, in the complex
$\DB^{2p-n-1}_{\TW,D}(X,p)$ the equation 
\begin{displaymath}
d\underline{g_{Z}}=-\caP(Z)+\underline{\theta_{Z}},
\end{displaymath}
is satisfied.  Hence $\underline{g_{Z}}$ is a Green current for the
cycle $Z$ as in \cite[Definition 6.1]{Burgoswami:hait}). 

Let now $W$ be a cycle in $Z^q(X,m)_{00}$, which intersects $Z$
properly and $\underline{g_{W}}$ a Green current for $W$ in the
Thom-Whitney complex. We now can give a second
(and simplified) definition of star product: 

\begin{df}\label{2ndstar} Let $g_{Z,\TW}$, $\underline{g_{Z}}$ and
  $\underline{g_{W}}$ be as before. Then we define the product
  \begin{displaymath}
    \underline{g_{Z}}\ast_{2}\underline{g_{W}}=
    (-1)^n(\pi _{X})_{\ast}\left(\delta_{Z,\TW}\cdot W_m\cdot g_{W,\TW}\cdot
      W_n\right)+\underline{g_{Z}}\cdot \underline{\theta_{W}} 
  \end{displaymath}
\end{df}
We note here that the products are taking place in the ambient space
$X\times (\P^1)^m\times (\P^1)^n$, and the notations should be
interpreted accordingly. For example $g_{\TW,Z}$ really means the
pullback of this form in the ambient space. We avoid the pullback
notations to simplify the exposition. This note will hold true
whenever we take product between elements in a priori different
spaces.

We next show that the star product $\ast_{2}$ is compatible with the
star product $\ast$ in \cite[\S 6.4]{Burgoswami:hait},

\begin{prop}\label{prop2ndGreen}
Let $\fg'_W$ be a Green form for $W$ in the Thom-Whitney complex, such
that $\underline{g_{W}}^\sim=[\fg_W]^\sim$. Then for any Green current
$g_Z$ of $Z$, we have 
\begin{displaymath}
\left(g_Z\ast_2 \underline{g_W}\right)^\sim=(g_Z\ast \fg'_W)^\sim.
\end{displaymath}
\end{prop}
\begin{proof}
  Since the product $(g_Z\ast \fg'_W)^\sim$ is independent on the
  choice of $\fg'_W$ we can make a particular choice. 
  We consider the elements $(\alpha _{m},\dots,\alpha _{0})$
  satisfying \eqref{eq:51}. We write
  \begin{displaymath}
    \alpha =\sum _{i=0}^{m} (\pi _{X})_{\ast}(\alpha _{i}\cdot W_{i}).
  \end{displaymath}
  Then $\alpha $ is closed. Indeed by \eqref{eq:51} and
  \cite[(5.7)]{Burgoswami:hait}
  \begin{multline*}
    d\alpha =\sum _{i=0}^{m} (\pi _{X})_{\ast}(d\alpha _{i}\cdot
    W_{i}) + \sum _{i=0}^{m} (-1)^{2p-i-1}(\pi _{X})_{\ast}(\alpha _{i}\cdot
    dW_{i})=\\-(\pi _{X})_{\ast}(\theta _{W,\TW}\cdot
    W_{n})+\underline{\theta _{W}} =0.
  \end{multline*}
  we define
  \begin{displaymath}
    \fg'_{W}=(g_{W,\TW}+\alpha _{m},\alpha _{m-1},\dots,\alpha
    _{1},\alpha _{0}-\alpha ).
  \end{displaymath}
  With this choice
  \begin{displaymath}
    [\fg'_{W}]=(\pi _{X})_{\ast}(g_{W,\TW}\cdot W_{n})+
    \sum _{i=0}^{m} (\pi _{X})_{\ast}(\alpha _{i}\cdot W_{i}) -
    \alpha = \underline{g_{W}}.
  \end{displaymath}
  Moreover
  \begin{multline*}
    (-1)^{n}(g_{Z}\ast \fg'_{W}-g_{Z}\ast_{2} \underline{g_{W}})=\\
    \sum_{i=0}^{m}(\pi _{X})_{\ast}(\delta_{Z,\TW}\cdot W_m\cdot \alpha _{i}\cdot
    W_{i}) 
    -(\pi _{X})_{\ast}(\delta_{Z,\TW}\cdot W_m\cdot \alpha)=0.
  \end{multline*}
  proving the proposition.
\end{proof}
As a consequence we obtain the following formula for the higher
archimedean height pairing of Definition \ref{arcdf}.
\begin{cor}\label{cor:refined-height}
If $Z\in Z^p(X,n)_{00}$ and $W\in Z^q(X,m)_{00}$ be two higher cycles whose real regulator classes are zero with $2(p+q-d-1)=n+m$, then 
\begin{displaymath}
\langle Z,W\rangle_{\Arch} = (-1)^n(p)_{\ast}\left(\delta_{Z,\TW}\cdot
  W_m\cdot g_{W,\TW}\cdot 
      W_n\right)^\sim ,
\end{displaymath}
where $p\colon X\times (\P^1)^n\times (\P^1)^m\rightarrow\Spec(\C)$ is
the structural morphism. 
\end{cor}
\begin{proof}
The key point is that we can use the second definition of Green current using Proposition \ref{prop2ndGreen} for the particular choice of Green form for $W$, since higher archimedean height pairing is independent of the choice of Green form for a higher cycle. Next the real regulator class of $W$ being zero allows us to choose $\theta_W=0$ by Corollary \ref{thetazero}. This concludes the proof.
\end{proof}

\subsection{The dual extension}\label{dualext}
Let now $q\ge 0$ and $m\ge 1$ be integers and let $W\in
Z^{q}(X,m)_{00}$ be a cycle. We apply the construction of sections
\ref{sec:extens-assoc-high} and \ref{subsubsecdiff} to this
setting, obtaining an extension $E_{W}$ and the corresponding
differential forms. 
We can dualize the extension $E_{W}$ to get a dual extension
\begin{displaymath}
  E_{W}^{\vee}=\Hom_{\MHS}(E_{W},\Q(0)).
\end{displaymath}
This extension is given by the short exact sequence
\begin{displaymath}
  0\longrightarrow \Q(0)\longrightarrow E^{\vee}_W
  \longrightarrow H^{2d-2q+m+1}(X;d-p)\longrightarrow 0,
\end{displaymath}
dual to \eqref{eq:53}. By construction $E_{W}$ is a sub-mixed Hodge
structure of
\begin{equation}\label{eq:52}
 H^{2q-1}(X\times (\P^1)^m\setminus A_X\cup |W|, B_X;p) 
\end{equation}
By duality we would like to see $E_{W}^{\vee}$ as a quotient
mixed Hodge structure. 
A naive idea would be to think that $E_{W}^{\vee}$ should be a
quotient of
\begin{displaymath}
  H^{2d-2q+m+1}(X\times (\P^1)^m\setminus B_X, A_X\cup |W|;d+m-q).
\end{displaymath}
But the problem is that the above group does not need to be the dual
to \eqref{eq:52} because $B_{X}$ and $A_X\cup |W|$ may fail to be
in a local product situation. 
To remedy this situation, we consider a composition of blow ups as in
the next lemma.

\begin{lem}\label{lemm:4} There exists a proper transform 
  \begin{displaymath}
  \pi\colon \caX_W\rightarrow X\times (\P^1)^m,
\end{displaymath}
that is a composition of blow-ups 
with smooth centers whose image in $X\times (\P^1)^m$ is contained in
$|W|\cap B_X $, such that, if we denote by
$\widehat{W}$, $\widehat{A}_X$ 
and  $\widehat{B}_X $ the strict transforms of $|W|$, $A_{X}$ and
$B_{X}$ respectively, and by $D$ the exceptional divisor, then
\begin{enumerate}
\item The strict transforms $\widehat{W}$ and $\widehat{B}_X$ do not
  meet.
\item The divisor $\widehat{A}_X\cup D \cup \widehat{B}_X $ is a
  simple normal
  crossing divisor.
\end{enumerate}
The previous conditions imply that the pair of closed subsets $\widehat{A}_X\cup D$ and
$\widehat{B}_X $ are in local product situation and  the same is true
for the pair $\widehat{A}_X\cup D\cup \widehat{W}$ and
$\widehat{B}_X $.
\end{lem}
\begin{proof}
  Let $\caI_{W}$ be the ideal sheaf of $|W|$ and $\caI_{B}$ the ideal
  sheaf of $B_{X}$ by blowing up $\caI_{W}+\caI_{B}$ we obtain a
  proper transform $X_{1}\to X\times (\P^1)^m$ such that the strict
  transform of $|W|$ and $B_{X}$ do not meet (\cite[Chapter II,
  Exercise 7.12]{Hartshorne:ag}). This proper transform is 
  an isomorphism outside $|W|\cap B_X $ but $X_{1}$ is possibly
  singular. By the use of strong resolution of
  singularities in the elimination of indeterminacies, there is
  proper transform 
  \begin{displaymath}
    \pi\colon \caX_W\rightarrow X\times (\P^1)^m,
  \end{displaymath}
  that is a composition of blow-ups 
  with smooth centers whose image in $X\times (\P^1)^m$ is contained in
  $|W|\cap B_X $, with a map $\caX_W\to X_{1}$ making the diagram
  \begin{displaymath}
    \xymatrix{
      \caX_W\ar[r]\ar[dr]& X_{1}\ar[d]\\
      & X\times (\P^1)^m
    }
  \end{displaymath}
  commutative and satisfying the conditions of the lemma.
\end{proof}

Let $\pi\colon \caX_W\rightarrow X\times (\P^1)^m$ be a map provided
by Lemma \ref{lemm:4}, 
\begin{notation}\label{def:5}
In the sequel we will use the following shorthands.
\begin{alignat*}{2}
  \square^{m} &= (\P^{1})^{m}\setminus A, &\quad
  \dsquare^{m} &= ((\P^{1})^{m}\setminus A,B),\\
  \square_{X}^{m} &= X\times (\P^{1})^{m}\setminus A_{X}, &\quad
  \dsquare_{X}^{m} &= (X\times (\P^{1})^{m}\setminus A_{X},B_{X})\\
  \widetilde {\square_{X}^{m}}&=\caX_{W}\setminus
  \widehat{A}_{X}&\qquad
  \widetilde {\dsquare_{X}^{m}}&=(\caX_{W}\setminus
  \widehat{A}_{X},\widehat{B}_{X}).
\end{alignat*}
and the dual ones
\begin{alignat*}{2}
  G^{m} &= (\P^{1})^{m}\setminus B, &\quad
  \G^{m} &= ((\P^{1})^{m}\setminus B,A),\\
  G_{X}^{m} &= X\times (\P^{1})^{m}\setminus B_{X}, &\quad
  \G_{X}^{m} &= (X\times (\P^{1})^{m}\setminus B_{X},A_{X})\\
  \widetilde {G_{X}^{m}}&=\caX_{W}\setminus
  \widehat{B}_{X},&\qquad
  \widetilde {\G_{X}^{m}}&=(\caX_{W}\setminus
  \widehat{B}_{X},\widehat{A}_{X}).
\end{alignat*}
Moreover, in the relative schemes like $\dsquare_{X}^{m}$, the
notation $(\dsquare_{X}^{m}\setminus S,T)$ will mean
\begin{displaymath}
  (X\times (\P^{1})^{m}\setminus A_{X}\cup S,B_{X}\cup T).
\end{displaymath}
\end{notation}
We have the following
\begin{lem}\label{5lemblowup1}
The cohomology of $\caX_W$ satisfies
\begin{enumerate}
\item\label{it3} the morphism
  \begin{displaymath}
    H^r(X\times (\P^1)^m\setminus A_X\cup |W|, B_X )\xrightarrow{\pi ^{\ast}}
    H^{r}(\caX_W\setminus \widehat{A}_X\cup D\cup\widehat{W}, \widehat{B}_X),
  \end{displaymath}  
 is an isomorphism for all $r\geq 0$;
\item\label{it4} the morphism
  \begin{displaymath}
    H^r(X\times (\P^1)^m\setminus A_X, B_X)\longrightarrow
    H^r(\caX_W\setminus \widehat{A}_X\cup D, \widehat{B}_X) 
  \end{displaymath}
  is an isomorphism for $r\leq 2q$, and injective for $r=2q+1$.
\end{enumerate}
\end{lem}
\begin{proof}
Since the map $\pi $ gives isomorphisms
\begin{align*}
  \caX_W\setminus \widehat{A}_X\cup D\cup\widehat{W}&\cong
  X\times (\P^1)^m\setminus A_X\cup |W| \\
  \widehat{B}_{X}\setminus \widehat{A}_X\cup D\cup\widehat{W}&\cong
  B_X\setminus A_X\cup |W|,
\end{align*}
we get \ref{it3} immediately.

For \ref{it4}, let $C$ be the center of the blow ups. by the same
reason as before, $\pi ^{\ast}$ gives isomorphisms
\begin{displaymath}
    H^r(X\times (\P^1)^m\setminus A_X\cup C, B_X)\xrightarrow{\cong}
  H^r(\caX_W\setminus \widehat{A}_X\cup D, \widehat{B}_X).
\end{displaymath}
Moreover, using Notation \ref{def:5} we have a diagram of mixed Hodge structures with exact rows
and commutative squares 
{\tiny
\begin{displaymath}
  \xymatrix@C=1em{
    H^{r-1}(\square^{m}_X)\ar[d]^-{\circled{1}}\ar[r]
    & H^{r-1}(B_X\setminus A_X)\ar[d]^-{\circled{2}}\ar[r]&
   H^r(\dsquare^{m}_X)\ar[d]^-{\circled{3}}\ar[r]
   & H^r(\square^{m}_{X})\ar[d]^-{\circled{4}}\ar[r]
   & H^r(B_X\setminus A_X)\ar[d]^-{\circled{5}}\\ 
    H^{r-1}(\widetilde {\square_{X}^{m}}\setminus D)\ar[r] &
   H^{r-1}(\widehat{B}_X\setminus \widehat{A}_X\cup D)\ar[r] &
 H^{r}(\widetilde {\dsquare_{X}^{m}}\setminus D) \ar[r] &
   H^{r}(\widetilde {\square_{X}^{m}}\setminus D)\ar[r] &
   H^{r}(\widehat{B}_X\setminus \widehat{A}_X\cup D)} 
\end{displaymath}
}
Since $C$ has codimension $\geq q+1$ in $X\times (\P^1)^m$ and $C\cap
B_W$ has codimension $\geq q+1$ in $B$, the arrows
$\circled{1}$, $\circled{2}$, $\circled{4}$ and $\circled{5}$ 
are isomorphisms. Hence the arrow $\circled{3}$ is also an isomorphism for $r\leq 2q$. For $r=2q+1$, the
arrows $\circled{1}$ and $\circled{2}$ are isomorphisms, while the arrow $\circled{4}$ is
injective. Hence the arrow $\circled{3}$ is also injective.
\end{proof}

\begin{cor}\label{cor:3}
  The morphism
    \begin{displaymath}
      H^r(\dsquare^{m}_{X})\longrightarrow
    H^r(\widetilde{\dsquare^{m}_{X}} \setminus D) 
  \end{displaymath}
  is an isomorphism for $r\le 2q$ and injective for $r=2q+1$. Dually,
  the map
  \begin{displaymath}
    H^{s}(\widetilde{\G^{m}_{X}},D)\longrightarrow
    H^{s}(\G^{m}_{X})
  \end{displaymath}
  is an isomorphism for $s\ge 2d+2m-2q$ and surjective for
  $s=2d+2m-2q-1$.
\end{cor}
We now consider the commutative diagram with exact rows
  \begin{displaymath}
    \xymatrix@C=.6em{H^{2q-1}(\dsquare^{m}_{X})\ar[d]^{\cong}\ar@{^{(}->}[r]
      & H^{2q-1}(\dsquare^{m}_{X} \setminus |W|)\ar[d]^{\cong}\ar[r]
      & H^{2q}_{|W|}(\dsquare^{m}_{X})\ar[d]^{\cong}\ar[r]
      & H^{2q}(\dsquare^{m}_{X})\ar[d]^{\cong}\\ 
      H^{2q-1}(\widetilde{\dsquare^{m}_{X}}\setminus D)\ar@{^{(}->}[r]
      & H^{2q-1}(\widetilde{\dsquare^{m}_{X}}\setminus D\cup
      \widehat{W})\ar[r]
      & H^{2q}_{\widehat{W}}(\widetilde{\dsquare^{m}_{X}}\setminus
      D) \ar[r]
      & H^{2q}(\widetilde{\dsquare^{m}_{X}}\setminus D)
 }         
\end{displaymath}
where the vertical arrows are isomorphisms thanks to lemma \ref{5lemblowup1},

In the bottom row of the above diagram, all the relevant relative
schemes are in a product situation. Hence the dual to this bottom row,
after twisting by $\Q(-d-m)$ to make the twist disappear, writing
$d_{W}=\dim(W)=d+m-q$, and taking into account that
$H^{2d_{W}}(\widehat{W})=H^{2d_{W}}(\widehat{W},
\widehat{W}\cap(D\cup \widehat{A}_X))$,
reads
\begin{displaymath}
  H^{2d_{W}}(\widetilde {\G^{m}_{X}},D)
  \to
  H^{2d_{W}}(\widehat{W})
  \to
  H^{2d_{W}+1}(\widetilde {\G^{m}_{X}},D\cup \widehat{W})
  \twoheadrightarrow
  H^{2d_{W}+1}(\widetilde {\G^{m}_{X}},D).
\end{displaymath}
After unfolding Notation \ref{def:5} we obtain
\begin{multline}\label{eq:56}
H^{2d_{W}}(\caX_W\setminus \widehat{B}_X, \widehat{A}_X\cup
D)\rightarrow H^{2d_{W}}(\widehat{W})\rightarrow\\ 
H^{2d_{W}+1}(\caX_W\setminus \widehat{B}_X, \widehat{A}_X\cup D\cup
\widehat{W})\rightarrow H^{2d_{W}+1}(\caX_W\setminus \widehat{B}_X,
\widehat{A}_X\cup D)\rightarrow 0.
\end{multline}
Just as a sanity check, note that in this exact sequence the first
arrow is well defined
because $\widehat{W}\cap \widehat{B}_X=\emptyset$ and there is a
zero at the end because $\dim\widehat{W}=d_{W}$.
We now use that
\begin{displaymath}
  H^{2d_{W}}(\widehat{W}, \widehat{W}\cap(\widehat{A}_X\cup D))\cong H^{2d_{W}}(\widehat{W}),
\end{displaymath}
since $\dim (\widehat{W}\cap(\widehat{A}_X\cup D))< d_{W}$.

The class of $W$ produces a morphism of mixed Hodge structure
\begin{equation}\label{eq:63}
  \phi^\vee_W\colon H^{2d_{W}}(\widehat{W}; d_{W})\longrightarrow \Q(0),
\end{equation}
that is the dual of the map \eqref{eq:55}. The fact that the image of
the class $[W]$ in $H^{2q}(X\times (\P^1)^m\setminus A_X, B_X;p)$ is
zero implies that 
\begin{displaymath}
  \phi^\vee_W\left(H^{2d_{W}}(\caX_W\setminus \widehat{B}_X, \widehat{A}_X\cup D;d_{W})\right)=0.
\end{displaymath}
Hence, taking the push-forward through $\phi^\vee_W$ of the exact
sequence \eqref{eq:56}, we obtain a short exact sequence 
\begin{displaymath}
0\rightarrow \Q(0)\rightarrow E^\vee_W\rightarrow H^{2d_{W}+1}(\caX_W\setminus \widehat{B}_X,
\widehat{A}_X\cup D;d_{W}) \rightarrow 0.
\end{displaymath}
By Lemma \ref{5lemblowup1}~\ref{it4}, the fact that $\widehat{B}_X$
and $\widehat{A}_X\cup D$ are in local product situation and the
isomorphism \eqref{eq:5} we have
\begin{align*}
  H^{2d_{W}+1}(\caX_W\setminus \widehat{B}_X,
  \widehat{A}_X\cup D;d_{W})
  &= H^{2q-1}(\caX_W\setminus\widehat{A}_X\cup
    D,\widehat{B}_X;p)^{\vee}\\
  &= H^{2q-1}(X\times(\P^{1})^{m}\setminus A_X,B_X;p)^{\vee}\\
  &= H^{2d_{W}+1}(X\times(\P^{1})^{m}\setminus B_X,A_X;d_{W})\\
  &= H^{2d-(2q-m-1)}(X;d-q).
\end{align*}
Therefore the above short exact sequence can be written as
\begin{equation}\label{eq:58}
  0\rightarrow \Q(0)\rightarrow E^\vee_W\rightarrow
  H^{2d-(2q-m-1)}(X;d-q) \rightarrow 0.
\end{equation}
By construction this exact sequence is 
the dual sequence to \ref{eq:53}. As with $E_Z$, we keep using the notation $E^\vee_W$ to actually denote its class in
\begin{displaymath}
  \Ext^1_{\MHS}\left(H^{2d+m-2q+1}(X;d-q), \Q(0)\right).
\end{displaymath}

\subsection{Oriented MHS attached to a pair of higher cycles}
\label{sec:orient-mhs-attach}
Let $n,m\ge 1$, and $p,q\ge 0$ be integers with
\begin{equation}
  \label{eq:57}
  2(p+q-d-1)=n+m.
\end{equation}
Let $Z\in Z^p(X,n)_{00}$,
and $W\in Z^q(X,m)_{00},$ be two cycles in the refined normalized
complex  intersecting properly.
We want to attach an oriented rational mixed Hodge structure to this
pair. This mixed Hodge structure is similar to the one constructed by
Hain in \cite{Hain:Height}, with one significant difference: In the
case for usual cycles homologous to zero, proper intersection and the
numerical relation $p+q=d+1$ mean that the supports of the cycles are
disjoint, which is no longer the case here. So one should expect the
new mixed Hodge structure to reflect this phenomenon. Moreover, the use of proper modification 
in order to use duality
will add another technical difficulty.

Let
\begin{align*}
  \pi _{1}\colon X \times (\P^{1})^{n}\times
  (\P^{1})^{m}&\longrightarrow 
  X \times (\P^{1})^{n}\\
  \pi _{2}\colon X \times (\P^{1})^{n}\times (\P^{1})^{m}
  &\longrightarrow X \times (\P^{1})^{m} 
\end{align*}
be the two projections. Then the fact that $Z$ and $W$ meet properly
means precisely that $p_{1}^{-1}(|Z|)\cap p_{2}^{-1}(|W|)\cap X\times
\square^{n+m}$ has codimension $p+q$ and intersects properly all the
faces of $\square^{n+m}$. Hence there is a well defined intersection
pre-cycle
\begin{displaymath}
  Z\cdot W\in Z^{p+q}(X,n+m)_{0}
\end{displaymath}
Since $Z$ and $W$ are cycles in the refined normalized complex, the
same is true for $Z\cdot W$.

Let $\pi \colon \caX_{W}\to X\times (\P^{1})^{m}$ be a proper
modification as in Lemma \ref{lemm:4} applied to $W$. Let $C\subset
|W|$ be the support of the center of $\pi $. Then $\pi $ is an
isomorphism outside $C$. On $\caX_{W}$, $\widehat{W}$,
$\widehat{A}_{X}$ and $\widehat{B}_{X}$ are the strict transforms of
$|W|$, $A_{X}$ and $B_{X}$ and $D$ is the exceptional divisor. 
 
We will assume the following technical conditions.
\begin{assumption}\label{def:ass}
  The intersection $\pi _{1}^{-1}(|Z|)\cap \pi _{2}^{-1}(C)=\emptyset$.
\end{assumption}

\begin{rmk}
  The assumption \ref{def:ass} is more and more restrictive the bigger
  $n$ and $m$ are. In the case $n=m=1$, this condition is satisfied
  generically but it is not the case for higher values of $n$ and
  $m$. 
\end{rmk}

The sought mixed Hodge structure will appear in a diagram that
contains at the same time the exact sequence \eqref{eq:53} for the
cycle $Z$ and the dual exact sequence \eqref{eq:58} for the cycle
$W$. For the main diagram to fit in one page we need to complement Notation
\ref{def:5}.
\begin{notation}\label{not:1}
  We have already introduced the projections $\pi _{1}$ and $\pi _{2}$
  and consider also the projection
  \begin{displaymath}
    \pi _{3}\colon \caX_{W}\times (\P^{1})^{n}\longrightarrow
    \caX_{W}.
  \end{displaymath}
  Moreover we also consider the proper transform
  \begin{displaymath}
    \pi '\colon \caX_{W}\times (\P^{1})^{n}
    \longrightarrow
    X\times (\P^{1})^{n}\times (\P^{1})^{m}.
  \end{displaymath}
  Note that this map involves a change in the order of the variables. 
  We write,
  \begin{alignat*}{2}
    A_{1}&=\pi _{1}^{-1}A_{X},&\qquad A_{2}&=\pi_{2} ^{-1}A_{X},\\
    B_{1}&=\pi _{1}^{-1}B_{X},&\qquad B_{2}&=\pi_{2} ^{-1}B_{X},\\
    \overline {A}_{2}&=\pi _{3}^{-1}\widehat{A}_{X},&
    \qquad \overline {B}_{2}&=\pi_{3}^{-1}\widehat{B}_{X},\\
    \overline{A}_{1}&=(\pi')^{-1}A_{1},&
    \qquad \overline {B}_{1}&=(\pi')^{-1}B_{1},\\
    \overline D &=\pi _{3}^{-1}D,&\qquad C_{2}&=\pi _{2}^{-1}(C),\\
    \overline{Z}&=|Z|\times (\P^{1})^{m},&\qquad
    \overline{W}&=\pi _{3}^{-1}\widehat{W}.
  \end{alignat*}
  Note that the spaces marked with an overline are subsets of
  $\caX_{W}\times (\P^{1})^{n}$ while the others are 
  subsets of $X\times (\P^{1})^{n}\times (\P^{1})^{m}$. 
  We will also consider the relative schemes
  \begin{align*}
    \dsquare_{X}^{n,m}
    &=\dsquare_{X}^{n}\times_{X}\G^{m}_{X}=
    (X\times (\P^{1})^{n}\times (\P^{1})^{m}\setminus A_{1}\cup
      B_{2},B_{1}\cup A_{2}),\\
    \widetilde{\dsquare_{X}^{n,m}}
    &=\dsquare_{X}^{n}\times_{X}\widetilde {\G^{m}_{X}}=
      (\caX_{W}\times (\P^{1})^{n}\setminus \overline{A}_{1}\cup
      \overline{B}_{2},\overline{B}_{1}\cup\overline{A}_{2}),\\
    \underline{Z}
    &=(\overline{Z}
      \setminus \overline{A}_{1}\cup
      \overline{B}_{2},\overline{B}_{1}\cup\overline {A}_{2})\subset
      \widetilde{\dsquare_{X}^{n,m}},\\
    \underline{W}&=(\overline{W}\setminus
    \overline{A}_{1},
    \overline{B}_{1}\cup\overline {A}_{2})\subset
    \widetilde{\dsquare_{X}^{n,m}}.
  \end{align*}
  The relative schemes will always be denoted, either with double
  line typography or with an underline.
  Finally we write $\underline{S}=\underline{Z}\cap
  \underline{W}$. Note that, by Assumption \ref{def:ass} the relative
  schemes $\underline Z$ and $\underline S$ can be seen as subschemes
  of either $\dsquare_{X}^{n,m}$ or
  $\widetilde{\dsquare_{X}^{n,m}}$. Note also that in the definition
  of $\underline{W}$, the divisor
  $\overline B_{2}$ does not appears because $\widehat{W}$ and
  $\widehat B_{X}$ are disjoint.
\end{notation}
We consider the commutative diagram
with exact rows and columns of figure \ref{fig:maindiag}. In that
diagram, we have omitted $\overline{D}$ in the last column because, by
Assumption \ref{def:ass}, it
is disjoint
with $\overline Z$.
\begin{figure}[ht]
  \centering
\begin{displaymath}
  \xymatrix@C=.8em{
    &\circled{1} & \circled{2} & \circled{3}& \\
    \circled{4}\ar[r]&
    H^{r}(\widetilde{\dsquare_{X}^{n,m}},\overline{D}) \ar[u] \ar[r]&
    H^{r}(\widetilde{\dsquare_{X}^{n,m}}\setminus \overline{Z},\overline{D}) \ar[u] \ar[r]&
    H^{r+1}_{\underline Z}(\widetilde{\dsquare_{X}^{n,m}}) \ar[u] \ar[r]&
    \circled{5}\\
    \circled{6}\ar[r]&
    H^{r}(\widetilde{\dsquare_{X}^{n,m}},\overline{W}\cup\overline{D}) \ar[u] \ar[r]&
    H^{r}(\widetilde{\dsquare_{X}^{n,m}}\setminus \overline{Z},
    \overline{W}\cup\overline{D}) \ar[u] \ar[r]&
    H^{r+1}_{\underline Z}(\widetilde{\dsquare_{X}^{n,m}},\overline{W}) \ar[u] \ar[r]&
    \circled{7}\\
    \circled{8}\ar[r]&
    H^{r-1}(\underline{W},\overline{W}\cap \overline{D}) \ar[u] \ar[r]&
    H^{r-1}(\underline{W} \setminus S,\overline{W}\cap \overline{D})
    \ar[u] \ar[r]&
    H^{r}_{\underline{S}}(\underline{W}) \ar[u] \ar[r]&
    \circled{9}\\
    &\circled{10}\ar[u] &\circled{11}\ar[u] &\circled{12}\ar[u]&  
  }
\end{displaymath}  
  \caption{The main diagram}
  \label{fig:maindiag}
\end{figure}

We now analyze the different terms in that diagram for $r=2p+m-1$. We
start with the
top left corner
\begin{align*}
  H^{2p+m-1}(\widetilde{\dsquare_{X}^{n,m}},\overline{D})
  &=H^{2p+m-1}((\widetilde {\G_{X}^{m}},D)\times \dsquare^{n})\\
  &=H^{2p+m-n-1}(\widetilde{\G_{X}^{m}},D)\\
  &=H^{2p+m-n-1}(\G_{X}^{m})\\
  &=H^{2p-n-1}(X;-m).
\end{align*}
The first equality is true at the level of relative schemes. The second
equality follows from \eqref{eq:4} and K\"unneth formula. Since, by
\eqref{eq:57}, $2p+m-n-1=2d+2m-2q+1$, hence the third equality follows form
Corollary \ref{cor:3}. The last one follows from \eqref{eq:5}. This
computation means in particular that the composition
\begin{equation}\label{eq:59}
  H^{2p+m-1}(\widetilde{\dsquare_{X}^{n,m}},\overline{D})
  \xleftarrow{\cong }H^{2p+m-1}(\dsquare_{X}^{n,m},C_{2})
  \longrightarrow
  H^{2p+m-1}(\dsquare_{X}^{n,m})
\end{equation}
is an isomorphism. The fact that the composition \eqref{eq:59} is an
isomorphism, together with the fact that $\overline D$ and $Z$ are
disjoint by Assumption \ref{def:ass} imply that the compositions   
\begin{multline}\label{eq:60}
  H^{2p+m-1}(\widetilde{\dsquare_{X}^{n,m}}\setminus \underline Z,\overline{D})
  \xleftarrow{\cong }H^{2p+m-1}(\dsquare_{X}^{n,m}\setminus \underline{Z},C_{2})
  \longrightarrow\\
  H^{2p+m-1}(\dsquare_{X}^{n,m}\setminus \underline{Z})
  \xrightarrow{\cong}
  H^{2p-1}(\dsquare_{X}^{n}\setminus Z;-m)
\end{multline}
and
\begin{equation}
  \label{eq:61}
  H^{2p}_{Z}(\dsquare _{X}^{n};-m)\longrightarrow 
  H^{2p+m}_{\underline Z}(\dsquare_{X}^{n,m})
  \longrightarrow 
  H^{2p+m}_{\underline Z}(\widetilde{\dsquare_{X}^{n,m}})
\end{equation}
are isomorphisms. So we can identify the top row of the diagram with
the exact sequence \eqref{eq:62}. In fact this argument also
implies that $\circled{4}$ is zero and that the image of the class of
$Z$ in $\circled{5}$ is also zero.

Since $2d_{W}=2d+2m-2q=2p-2+m-n$, using the isomorphisms 
\begin{multline*}
  H^{2d_{W}+1}((\widetilde{\G_{X}^{m}},\widehat{W}\cup D)),
  \xrightarrow{\cong}\\
  H^{2p+m-1}((\widetilde{\G_{X}^{m}},\widehat{W}\cup D)\times \dsquare^{n})
  =
  H^{2p+m-1}(\widetilde{\dsquare_{X}^{n,m}},\underline{W}\cup\overline{D})
\end{multline*}
and
\begin{multline*}
  H^{2d_{W}}(\widehat{W},\widehat{W}\cap D)=H^{2d_{W}}(\widehat{W})
  \xrightarrow{\cong}\\
  H^{2p+m-2}((\widehat{W},\widehat{W}\cap D)\times \dsquare^{n})=
  H^{2p+m-2}(\underline{W},\underline{W}\cap \overline{D})
\end{multline*}
we can identify the first column of the diagram with the exact
sequence  \eqref{eq:56}. By dimension reasons, these
identifications also imply that $\circled{1}$ is zero and that the
image of $\circled{10}$ twisted by $\Q(d_{W})$ under the map $\phi_{W}^{\vee}$ in
\eqref{eq:63} is zero.  
 
Note that the group $\circled{9}$ agrees with $\circled{1}$ and the
group $\circled{12}$ agrees with $\circled{4}$ so they both vanish.

Next we face the technical problem that, in general, the groups
$H^{\ast}_{\underline{S}}(\underline{W})$ are difficult to
control. Even if $S$ is one point, if $W$ is singular, it can be very
complicated. So in order to proceed we need to add another technical
assumption. Afterwards we will give an example of geometrical
conditions that assure the fulfillment of the technical assumption.

\begin{assumption}\label{def:6} Assume that the main diagram satisfies the
  following conditions:
  \begin{enumerate}
  \item\label{item:4} the image of the class of $Z$ in
    $H^{2p+m}_{\underline{S}}(\underline{W})$ is zero;
  \item\label{item:5} the map $\phi_{W}^{\vee}$ sends the image of
    $H^{2p+m-2}_{\underline{S}}(\underline{W};d_{W})$ to zero;
  \item\label{item:6} the mixed Hodge structure
    $H^{2p+m-1}_{\underline{S}}(\underline{W})$ has weights contained
    in the interval $[2p+m-n-1,2p+2m-1]$.
  \end{enumerate}
\end{assumption}

\begin{prop}\label{prop:8}
  Let $S_{0}$ be the union of components of $S$
  that are not contained in $\overline A_{1}\cup A_{2}$. If the
  conditions
  \begin{enumerate}
  \item\label{item:1} the subset $S_{0}$ is contained in
    $\widehat W_{\sm}$, the open subset of smooth points;
  \item\label{item:2} the pair of subsets $S_{0}$ and $\widehat{W}\cap(\overline
    B_{1}\cup \overline D\cup \overline A_{2})$ are in local product
    situation inside $\widehat{W}$;
  \item\label{item:3} we are in the symmetric situation $n=m$;
  \end{enumerate}
  are satisfied, then the conditions of Assumption \ref{def:6} are
  also satisfied. 
\end{prop}
\begin{proof}
  By resolving singularities of $\widehat W$ and using Lemma
  \ref{lemm:1}, the conditions \ref{item:1} and \ref{item:2} of the
  proposition imply that 
  \begin{displaymath}
    H^{r}_{\underline S}(\underline{W})=
    H^{2d+2n+2m-2q-r}(S_{0}\setminus \overline B_{1}\cup \overline D\cup
    \overline A_{2},\overline A_{1};d+n+m-q)^{\vee}.
  \end{displaymath}
  Since $\dim S_{0}=(n+m)/2-1$, by 
  \cite[Chapter IV, Proposition (3.5)]{GNPP} the cohomology of $S$ has
  weights in the interval $[0,n+m-2]$. Therefore 
  the weights of $H^{r}_{\underline S}(\underline{W})$
  are contained in the interval $[2p,2p+n+m-2]$. If we add the
  condition $n=m$, then this interval is contained in the interval of
  Assumption \ref{def:6}~\ref{item:6}.

  The class of $Z$ in  $H^{2p+m}_{\underline
    Z}(\widetilde{\dsquare_{X}^{n,m}})$ has weight $2p+2m$ (recall the
  isomorphism \eqref{eq:5}) since $H^{2p+m}_{\underline S}(\underline
  W)$ has weight at most $2p+2m-2$ (here as well we are using $n=m$)
  condition \ref{def:6}~\ref{item:4} follows.

  Using again $n=m$, the group
  $H^{2p+m-2}_{\underline{S}}(\underline{W};d_{W})$ has weights in
  the interval $[2,2m]$. Since the image of the map 
  $\phi _{W}^{\vee}$ has weight zero, we deduce condition
  \ref{def:6}~\ref{item:5}. 
\end{proof}

\begin{df}
  Let $n=m\ge 1$ and $p,q\ge 0$ satisfying $p+q=d+n+1$ and let $Z\in
  Z^{p}(X,n)_{00}$ and $W\in Z^{q}(X,n)_{00}$ be cycles satisfying
  assumptions \ref{def:ass} and \ref{def:6}. Then the \emph{oriented
    mixed Hodge structure diagram} associated to $Z,W$ is the diagram
  obtained from the main diagram in Figure \ref{fig:maindiag} by first
  twisting by $\Q(p+n)$, then taking the pullback by $\phi _{Z}$ and
  then the push-forward by $\phi _{W}^{\vee}$ twisted by
  $\Q(n+1)$. This diagram is depicted in figure
  \ref{fig:oriented_diagram}. 
  \begin{figure}[ht]
    \centering
    \begin{displaymath}
      \xymatrix{& 0 & 0 & 0\\
        0\ar[r] & H^{2p-n-1}(X; p)\ar[u]\ar[r]
        & E_Z\ar[u]\ar[r]
        & \Q(0)\ar[u]\ar[r] & 0\\ 
        0\ar[r] & E_W^\vee(n+1)\ar[r] \ar[u]
        & B_{Z,W} \ar[r] \ar[u] & C_{Z,W}\ar[u]\ar[r] & 0\\
        0\ar[r] & \Q(n+1)\ar[r]\ar[u] & D_{Z,W}\ar[r]\ar[u] &
        H^{2p+n-1}_{\underline{S}}(\underline{W};p+n)\ar[u]\ar[r] & 0\\
        & 0\ar[u] & 0\ar[u] & 0\ar[u]}      
    \end{displaymath}    
    \caption{Oriented mixed Hodge structure diagram}
    \label{fig:oriented_diagram}
  \end{figure}
\end{df}

\begin{rmk}
  In general, if we switch $Z$ and $W$, we do not obtain the dual of
  the diagram in figure \ref{fig:oriented_diagram}. The first problem
  is obvious: Assumption \ref{def:ass} is not symmetric. But even if
  assumptions \ref{def:ass}, \ref{def:6} and the symmetric assumptions 
  are satisfied, the two obtained diagrams may not be dual of each
  other if $Z$ and $W$ are not in local product situation. Later we
  will investigate in more detail the duality of this diagram in a
  particular case.
\end{rmk}

\subsection{The case \texorpdfstring{$n=m=1$}{nm1}}
\label{sec:case-n=m=1}

Due to the technical difficulties arising from the intersection $\pi
_{1}^{-1}(|Z|)\cap \pi _{2}^{-1}(|W|)$ we will concentrate on the case
$n=m=1$. Then equation  \eqref{eq:57} reads
\begin{equation}
  \label{eq:67}
  p+q = d+2.
\end{equation}
Proper intersection means that the intersection $\pi
^{-1}_{1}(|Z|)\cap \pi _{2}^{-1}(|W|)\cap X\times \square^{2}$ is a
finite set of points. 

To ease the analysis, we make the following stronger assumption.  

\begin{assumption}\label{def:7} We assume that $n=m=1$ and that the
  whole intersection $S=\pi _{1}^{-1}(|Z|)\cap
  \pi _{2}^{-1}(|W|)\subset X\times(\P^{1})^{2}$ is a finite set of
  points. Moreover,
  \begin{enumerate}
  \item\label{item:7} the subsets $S$ and $A_{1}\cup A_{2}\cup B_{1}\cup B_{2}$ are
    disjoint; 
  \item\label{item:8} the subset $S$ is contained in $\pi ^{-1}(|Z|_{\sm})\cap \pi
    _{2}^{-1}(|W|_{\sm})$ and 
    the intersection $\pi ^{-1}(|Z|)\cap \pi _{2}^{-1}(|W|)$ is
    transverse at every point of $S$.
  \end{enumerate}
  In particular $\pi ^{-1}(|Z|)$ and $\pi _{2}^{-1}(|W|)$ are in local
  product situation.  
\end{assumption}

Assumption \ref{def:7} implies that we can define the diagram in
figure \ref{fig:oriented_diagram} and also the same diagram with $Z$
and $W$ swapped. 

\begin{prop}\label{prop:9}
  Assumption \ref{def:7} implies assumptions \ref{def:ass} and
  \ref{def:6} for the pair $Z$, $W$ and for the reversed pair $W$,
  $Z$.   
\end{prop}
\begin{proof}
  By condition \ref{def:7}~\ref{item:7} $\pi_{1} ^{-1}(|Z|)$ and $\pi
  _{2}^{-1}(|W|\cap B_{X})$ are disjoint. Therefore Assumption
  \ref{def:ass} is satisfied. Since $S$ is a finite set of points
  contained in the smooth part of $\underline W$, the dimension of
  $\underline W$ is $d+2-q=p$, and $2p+m-1=2p$, we deduce that
  \begin{displaymath}
    H^{2p+m}_{\underline S}(\underline W)=H^{2p+m-2}_{\underline S}(\underline W)=0,
  \end{displaymath}
  and that $H^{2p+m-1}_{\underline S}(\underline W)$ is pure of weight
  $2p$. Hence Assumption \ref{def:6} is also satisfied.

  Since Assumption \ref{def:7} is symmetric with respect to the swap
  of $Z$ and $W$, we deduce assumptions \ref{def:ass} and \ref{def:6}
  for the pair reversed. 
\end{proof}

Next we modify the main diagram in figure \ref{fig:maindiag} to
achieve two goals. First we want it to be symmetric under the swap of
$Z$ and $W$, and second, we want the strict transforms of $Z$ and $W$
to be smooth in order to easily use differential forms on them.

Using the same method as in Lemma \ref{lemm:4}. we can find a 
proper transform $\pi _{Z}\colon \caX_{Z}'\to X\times \P^{1}$, with
centers contained in $|Z|_{\sing} \cup (|Z|\cap B_{X})$, with
exceptional divisor $D_{Z}$ such that
\begin{enumerate}
\item the strict transform $\widehat Z$ of $|Z|$ is smooth and does
  not meet $\widehat {B}_{X}$;
\item the divisor $\widehat {A}_{X}\cup D_{Z}\cup \widehat{B}_{X}$ is
  a simple normal crossing divisor.
\end{enumerate}
Similarly we construct the proper transform $\pi _{W}\colon
\caX_{W}'\to X\times \P^{1}$ and define 
\begin{displaymath}
\caX_{Z,W}\coloneqq \caX'_Z\times_{X}\caX'_{W},
\end{displaymath}
which is smooth under Assumption \ref{def:7}. We denote the union of
the centers
of blow ups for $\caX'_W$ and $\caX'_Z$ to be $C_W$ and $C_Z$
respectively. Let  
\begin{displaymath}
\pi'\colon \caX_{Z,W}\rightarrow X\times \P^1\times \P^1
\end{displaymath}
be the proper morphism induced by the maps $\pi _{Z}$ and $\pi _{W}$, and let
\begin{displaymath}
\pi_1'\colon \caX_{Z,W}\rightarrow \caX'_Z,\qquad
\pi_2'\colon \caX_{Z,W}\rightarrow \caX'_W, 
\end{displaymath}
be the projections. We summarize the different maps in the following diagram.
\begin{equation}\label{eq:75}
  \xymatrix{& \caX_{Z,W}\ar[dl]_{\pi _{1}'}\ar[dd]^{\pi '}\ar[dr]^{\pi _{2}'} &\\
    \caX'_{Z}\ar[dd]_{\pi _{Z}}&& \caX'_{W}\ar[dd]^{\pi _{W}}\\
    &X\times \P^{1} \times \P^{1} \ar[dl]_{\pi _{1}} \ar[dr]^{\pi _{2}}&\\
    X\times \P^{1} \ar[dr]&& X\times \P^{1}\ar[dl]\\
    &X&
  }
\end{equation}
We adapt Notation \ref{not:1} to this case,
and introduce 
\begin{notation}\label{not:2}
  \begin{alignat*}{2}
    A_{1}&=\pi _{1}^{-1}A_{X},&\qquad A_{2}&=\pi_{2} ^{-1}A_{X},\\
    B_{1}&=\pi _{1}^{-1}B_{X},&\qquad B_{2}&=\pi_{2} ^{-1}B_{X},\\
    \overline {A}_{1}&=(\pi _{1}')^{-1}\widehat{A}_{X},&
    \qquad \overline {B}_{1}&=(\pi_{1}')^{-1}\widehat{B}_{X},\\
    \overline {A}_{2}&=(\pi' _{2})^{-1}\widehat{A}_{X},&
    \qquad \overline {B}_{2}&=(\pi_{2}')^{-1}\widehat{B}_{X},\\
    \overline D_Z &=(\pi_{1}')^{-1}D_Z, &\qquad \overline D_W &=(\pi _{2}')^{-1}D_W,\\
    C_{1}&=\pi _{1}^{-1}(C_Z),&\qquad C_{2}&=\pi_{2}^{-1}(C_W),\\
    \overline{Z}&=(\pi _{1}')^{-1}\widehat{Z},&\qquad
    \overline{W}&=(\pi _{2}')^{-1}\widehat{W}.
  \end{alignat*}
  Here we have denoted by $\widehat A_{X}$ and $\widehat B_{X}$ the
  strict transforms of $A_{X}$ and $B_{X}$ in both blow-ups,
  $\caX_{Z}$ and $\caX_{W}$. 
  Note that the spaces marked with an overline are subsets of
  $\caX_{Z,W}$ while the others are 
  subsets of $X\times \P^{1}\times \P^{1}$. 
 As before, we will consider the relative schemes
  \begin{align*}
    \dsquare_{X}
    &=\dsquare_{X}\times_{X}\G_{X}=
    (X\times \P^{1}\times \P^{1}\setminus A_{1}\cup
      B_{2},B_{1}\cup A_{2}),\\
    \underline{\caX_{Z,W}}
    &=(\caX_{Z,W}\setminus \overline{A}_{1}\cup
      \overline{B}_{2}\cup\overline{D}_{Z},\overline{B}_{1}\cup
      \overline{A}_{2}\cup  \overline{D}_{W}),\\
    \underline{Z}
    &=(\overline{Z}
      \setminus \overline{A}_{1}\cup
      \overline{B}_{2}\cup \overline{D}_{Z},\overline {A}_{2})\subset
      \underline{\caX_{Z,W}},\\
    \underline{W}&=(\overline{W}\setminus
    \overline{A}_{1},\overline{B}_{1}\cup
      \overline{A}_{2}\cup  \overline{D}_{W})
                   \subset
    \underline{\caX_{Z,W}}.
  \end{align*}
  Finally we write $S=Z\cap W$. Note that, by Assumption \ref{def:7},
  the subset $S$ can be seen as a the relative scheme $\underline S
  \coloneqq (S\setminus  \emptyset,\emptyset)$ that is a relative subscheme of  
  either $\dsquare_{X}$ or
  $\underline{\caX_{Z,W}}$. As before,
  $\overline B_{2}$ does not appear in the definition of
  $\underline{W}$  because $\widehat{W}$ and 
  $\widehat B_{X}$ are disjoint. Similarly, $\overline B_{1}$ does not
  appear in the definition of $\underline Z$. 
\end{notation}

In figure \ref{fig:maindiag_b}, there is a more symmetric version of
the main diagram in figure
\ref{fig:maindiag}. 
\begin{figure}[ht]
  \centering
\begin{displaymath}
  \xymatrix@C=.8em{
    &0  & 0 & 0 & \\
    0\ar[r]&
    H^{2p}(\underline{\caX_{Z,W}}) \ar[u] \ar[r]&
    H^{2p}(\underline{\caX_{Z,W}}\setminus \overline{Z}) \ar[u] \ar[r]&
    H^{2p+1}_{\underline Z}(\underline{\caX_{Z,W}}) \ar[u] \ar[r]&
    \circled{5}\\
    0\ar[r]&
    H^{2p}(\underline{\caX_{Z,W}},\overline{W}) \ar[u] \ar[r]&
    H^{2p}(\underline{\caX_{Z,W}}\setminus \overline{Z},
    \overline{W}) \ar[u] \ar[r]&
    H^{2p+1}_{\underline Z}(\underline{\caX_{Z,W}},\overline{W}) \ar[u] \ar[r]&
    \circled{7}\\
    0\ar[r]&
    H^{2p-1}(\underline{W}) \ar[u] \ar[r]&
    H^{2p-1}(\underline{W} \setminus \overline{S}) \ar[u] \ar[r]&
    H^{2p}_{\underline{S}}(\underline{W}) \ar[u] \ar[r]&
    0\\
    &\circled{10}\ar[u] &\circled{11}\ar[u] &0\ar[u]&  
  }
\end{displaymath}  
  \caption{A symmetric version of the main diagram for $n=m=1$.}
  \label{fig:maindiag_b}
\end{figure}
The analysis of the main diagram carries through, with small
modifications to the diagram in figure \ref{fig:maindiag_b}. For
instance, using Lemma \ref{lemm:2}, the fact that $\codim C_{1}\ge
p+1$ and $\dim C_{2}\le p-1$, yield
\begin{align*}
  H^{2p}(\underline{\caX_{Z,W}})
  &= H^{2p}(\dsquare_{X}\setminus C_{1},C_{2})\\
  &= H^{2p}(\dsquare_{X})\\
  &= H^{2p-2}(X;-1)\\
\end{align*}

As in the proof of Proposition  \ref{prop:9}, the group $H^{2p}_{\underline
  S}(\underline W)$ is pure of weight  $2p$. In fact more is true. If
$s=\#S$ is the number of points in the intersection, then 
there is a canonical isomorphism
\begin{displaymath}
  H^{2p}_{\underline
  S}(\underline W)\cong \Q(-p)^{\oplus s}.
\end{displaymath}
Thus, after pulling back through the class of $Z$, taking the
pushforward with respect to the class of $W$ and twisting by
$\Q(p+1)$, we obtain, form figure \ref{fig:maindiag_b}, the
particular case of figure \ref{fig:oriented_diagram} depicted in
figure \ref{fig:biextensionnm1}.
\begin{figure}[ht]
  \centering
\begin{displaymath}
  \xymatrix{& 0 & 0 & 0\\
    0\ar[r] & H^{2p-2}(X;p)\ar[u]\ar[r]
    & E_Z\ar[u]\ar[r]
    & \Q(0)\ar[u]\ar[r] & 0\\ 
    0\ar[r] & E_W^\vee(2)\ar[r] \ar[u]
    & B_{Z,W} \ar[r] \ar[u] & C_{Z,W}\ar[u]\ar[r] & 0\\
    0\ar[r] & \Q(2)\ar[r]\ar[u] & D_{Z,W}\ar[r]\ar[u] &
    \Q(1)^{\oplus s}\ar[u]\ar[r] & 0\\
    & 0\ar[u] & 0\ar[u] & 0\ar[u]}       
   \end{displaymath}     
  \caption{The biextension diagram for $n=m=1$.}
\label{fig:biextensionnm1}  
\end{figure}

\begin{prop}\label{prop:10}
  With Assumption \ref{def:7}, the dual of the
  diagram of figure \ref{fig:biextensionnm1}, twisted by $\Q(2)$,
  agrees with the similar
  diagram with the role of $Z$ and $W$ reversed. In particular 
  \begin{displaymath}
    B_{W,Z}=B_{Z,W}^{\vee}(2),\qquad C_{W,Z}=D_{Z,W}^{\vee}(2),
    \qquad D_{W,Z}=C_{Z,W}^{\vee}(2).
  \end{displaymath}
\end{prop}
\begin{proof}
  Since, by condition condition \eqref{eq:67}, we have
  $(p+1) + (q+1) -2 = d+2 = \dim(\caX_{Z,W})$, and by Assumption
  \ref{def:7} all the  
  subspaces appearing in the diagram in figure
  \ref{fig:maindiag_b} are in local product situation, if we take that 
  diagram, twist it by $\Q(p+1)$, then
  take the dual and finally twist by $\Q(2)$, we obtain the analogous
  diagram, with $Z$ and $W$ swapped and
  twisted by $\Q(q+1)$. For instance, the central term of the first
  diagram twisted by $\Q(p+1)$ is 
 \begin{equation}\label{eq:71}
   H^{2p}(\underline{\caX_{Z,W}}\setminus \overline Z, \overline W;
   p+1), 
 \end{equation}
 and $B_{Z,W}$ as a sub quotient this mixed
 Hodge structure. The dual of this cohomology group, twisted by 
 $\Q(2)$ is 
 \begin{displaymath}
   H^{2q}(\underline{\caX_{Z,W}}\setminus \overline W,\overline Z; q+1).  
 \end{displaymath}
 From this the sought duality follows easily.
\end{proof}

From the diagram of figure \ref{fig:biextensionnm1}, and the fact that
all the maps there are 
morphisms of mixed Hodge structures, we deduce the next result.

\begin{cor}\label{cor:4}
  If Assumption \ref{def:7} is satisfied, then the mixed Hodge
  structure $B_{Z,W}$ has weights $-4$, $-2$ and $0$ and the graded
  pieces are  
\begin{displaymath}
\begin{aligned}
\Gr^W_0B_{Z,W}&= \Q(0),\\
\Gr^W_{-2}B_{Z,W}&= H^{2p-2}(X, \Q(p))\oplus \Q(1)^{\oplus_s},\\
\Gr^W_{-4}B_{Z,W}&= \Q(2).
\end{aligned}
\end{displaymath}
Therefore it is a generalized biextension. Moreover, if $H^{2p-2}(X,\Q(p))$
is of Hodge--Tate type, the same is true for $B_{Z,W}$. 
\end{cor}

\begin{rmk}
  In the case $n=m=1$,
  the duality in Proposition \ref{prop:10} is not only a duality of mixed Hodge
  structures, as we will see in the proof of the next
  proposition, this duality preserves the orientation. This is in
  contrast with the case $n=m=0$ as shown in \cite[Proposition
  3.3.4]{Hain:Height}.
\end{rmk}

\begin{prop}
   With Assumption \ref{def:7}, we have
   \begin{displaymath}
     \Ht(B_{Z,W})=-\Ht(B_{W,Z}).
   \end{displaymath}
 \end{prop}
 \begin{proof}
   By Proposition \ref{prop:dual-height} we only need to show that the
   duality between $B_{Z,W}$ and $B_{W,Z}$ preserves the
   orientation. The mixed Hodge structure $B_{Z,W}$ is a subquotient
   of $H^{2p}(\underline{\caX_{Z,W}}\setminus \overline{Z},
    \overline{W}; p+1)$, Hence its elements can be represented by
    differential forms in
    \begin{displaymath}
      E_1=\Sigma _{\overline{B}_{1}\cup
      \overline{A}_{2}\cup  \overline{D}_{W}\cup \overline W
    }E^{2p}_{\caX_{Z,W}}(\log \overline{A}_{1}\cup
      \overline{B}_{2}\cup\overline{D}_{Z}\cup \overline Z; p+1),
    \end{displaymath}
    while the elements in $B_{W,Z}$ can be represented by forms in
    \begin{displaymath}
      E_2=\Sigma _{\overline{A}_{1}\cup
      \overline{B}_{2}\cup\overline{D}_{Z}\cup \overline
      Z} E^{2q}_{\caX_{Z,W}}(\log \overline{B}_{1}\cup 
      \overline{A}_{2}\cup  \overline{D}_{W}\cup \overline W; q+1).
    \end{displaymath}
    The duality is given by the map
    \begin{displaymath}
      \langle \alpha ,\beta \rangle =\int _{\caX_{Z,W}}\alpha \wedge
      \beta. 
    \end{displaymath}
    The class of $Z$ is represented by a differential form $\nu
    _{Z}\in E_1$ and its dual class can be represented by a
    differential form $\mu _{Z}\in E_{2}$. Similarly we have
    differential forms $\nu _{W}$ and $\mu _{W}$. These forms satisfy
    \begin{displaymath}
      \int _{\caX_{Z,W}} \nu _{Z}\wedge \mu _{Z}= \int _{\caX_{Z,W}}
      \nu _{W}\wedge \mu _{W}.
    \end{displaymath}
    The orientation of $B_{Z,W}$ is given by the classes $(\nu
    _{Z},\mu _{W})$ and the orientation of $B_{W,Z}$ by the classes 
    $(\nu_{W},\mu _{Z})$.
    Since
    \begin{displaymath}
      \langle \nu _{Z},\mu _{Z}\rangle =\int _{\caX_{Z,W}} \nu _{Z}\wedge \mu _{Z}=1
    \end{displaymath} 
    and 
    \begin{equation}\label{eq:93}
      \langle \mu _{W},\nu _{W}\rangle =\int _{\caX_{Z,W}} \mu
      _{W}\wedge \nu _{W}=(-1)^{4pq}=1 
    \end{equation}
    we obtain that the duality preserves orientations and hence the
    result.  Note that in equation \eqref{eq:93} we are using that
    $n=m=1$, that implies that the forms $\mu _{W}$ and $\nu _{W}$ have even
    degree. In the case $n=m=0$ the differential forms have odd
    degree, hence the similar duality would not be compatible with the
    orientations. 
 \end{proof}

\section{Invariants attached to the mixed Hodge structure
  \texorpdfstring{$B_{Z,W}$}{BZW}}\label{sec5}

In this section we suppose that Assumption \ref{def:7} is satisfied and
compute the Deligne splitting $\delta $ of $B_{Z,W}$ (see
\eqref{delta-def}). This map characterizes $B_{Z,W}$ as a real mixed
Hodge structure.

\subsection{A decomposition of the Deligne splitting of
  \texorpdfstring{$B_{Z,W}$}{BZW}}
\label{sec:decomp-texorpdfstr}

Since we will be considering different mixed Hodge structures we will
use the following variant of the notation in Section
\ref{sec:mhs-arc} to keep track of them.

\begin{notation}
For a MHS $H$, we will denote the Deligne bigrading as
$H_{\C}=\bigoplus_{r,s} I_{H}^{r,s}$, and will denote the various
projections to the individual $I_{H}^{r,s}$ by
$\Pi_{I_{H}^{r,s}}$. Similarly the projection to the piece $\bigoplus_{p+q=k}I^{p,q}_{H}$ of pure
weight $k$ will be denoted $\Pi _{H,k}$. Also, the Deligne splitting
of $H$ will be denoted $\delta _{H}$.
\end{notation}

After Corollary \ref{cor:4}, the Deligne bigrading of $B\coloneqq B_{Z,W}$ (See
\eqref{deligne-bigrading}) has the shape
\begin{displaymath}
\begin{gathered}
B_{\C}=I_{B}^{0,0}\oplus\left(\bigoplus_{a+b=-2}
  I_{B}^{a,b}\right)\oplus I_{B}^{-2,-2} 
\end{gathered}
\end{displaymath}
Similarly, the bigradings of $C\coloneqq C_{Z,W}$, $D\coloneqq
D_{Z,W}$, $E_{Z}$ and $E^{\vee}_{W}$ are given by
\begin{displaymath}
\begin{alignedat}{2}
C_{\C}&=I^{0,0}_C\oplus I^{-1,-1}_C,\quad &
D_{\C}&=I^{-1,-1}_D\oplus I^{-2,-2}_D,\\
E_{Z,\C}&=I^{0,0}_{E_{Z}}\oplus \bigoplus_{a+b=-2} I^{a,b}_{E_{Z}}, \quad
& E_{W,\C}^{\vee}(2)&= \bigoplus_{a+b=-2} I^{a,b}_{E_{W}^{\vee}(2)}\oplus
I^{-2,-2}_{E_{W}^{\vee}(2)}.
\end{alignedat}
\end{displaymath}
Since $H^{2p-2}(X, \Q(p))$ and $\Q(1)^{s}$ are pure Hodge structures
of weight $-2$, their Deligne bigradings are given by
\begin{displaymath}
  H^{2p-2}(X,\Q(p))_{\C}=\bigoplus_{a+b=-2} I^{a,b}_{1},\qquad
  \Q(1)_{\C}^{s}=I^{-1,-1}_{2}.
\end{displaymath}
The functoriality of the Deligne bigrading and the diagram of figure \ref{fig:biextensionnm1}, give us canonical
identifications
\begin{equation}
  \label{eq:69}
  \begin{aligned}
    I^{0,0}_{B}&=I^{0,0}_{C}=I^{0,0}_{E_{Z}},\\
    I^{a,b}_{E_{Z}}&=I^{a,b}_{E_{W}^{\vee}(2)}=I^{a,b}_{1},\quad
    \text{ for }a+b=-2,\\
    I^{-1,-1}_{C}&=I^{-1,-1}_{D}=I^{-1,-1}_{2},\\
    I^{-1,-1}_{B}&=I^{-1,-1}_{C}\oplus I^{-1,-1}_{E_{Z}},\\
    I^{a,b}_{B}&=I^{a,b}_{E_{Z}},\quad \text{ for }a+b=-2,\ a\not =
    -1,\\
    I^{-2,-2}_{B}&=I^{-2,-2}_{D}=I^{-2,-2}_{E^{\vee}_{W}(2)}.
  \end{aligned}
\end{equation}
In terms of the graded pieces of the weight filtration we obtain identifications
\begin{equation}
  \label{eq:70}
  \begin{aligned}
    \Gr_{0}^{W}B&=\Gr_{0}^{W} C=\Gr_{0}^{W} E_{Z}=\Q(0),\\
    \Gr_{-2}^{W}B&=\Gr_{-2}^{W}C\oplus
    \Gr_{-2}^{W}E_{Z}\\
    &=\Gr_{-2}^{W}D\oplus \Gr_{-2}^{W}E^{\vee}_{W}(2)=
    H^{2p-2}(X,\Q(p))\oplus\Q(1)_{\C}^{s},\\
    \Gr_{-4}^{W}B&=\Gr_{-4}^{W}D\oplus \Gr_{-4}^{W}E^{\vee}_{W}(2)=\Q(2).
  \end{aligned}
\end{equation}
As in the proof of Lemma \ref{main-ht-formula} there is a
decomposition
\begin{displaymath}
  \delta _{B}= \delta _{1}+\delta _{2}+\delta _{3},
\end{displaymath}
with
\begin{displaymath}
  \delta _{1}\colon \Gr_{0}^{W}B\to \Gr_{-2}^{W}B,\quad 
  \delta _{2}\colon \Gr_{-2}^{W}B\to \Gr_{-4}^{W}B,\quad 
  \delta _{3}\colon \Gr_{0}^{W}B\to \Gr_{-4}^{W}B. 
\end{displaymath}
Using the identifications \eqref{eq:70}, we can write
\begin{displaymath}
  \delta _{1}=\delta _{E_{Z}}+\delta _{C},\qquad
  \delta _{2}=\delta _{E_{W}^{\vee}(2)}+\delta _{D}.
\end{displaymath}
Moreover, $\delta _{3}=\delta _{B}^{-2,-2}$ as in Definition
\ref{signed-height}. Therefore, if $e$ and $e^{\vee}$ are the
generators of $I_{B}^{0,0}$ and $I^{-2,-2}_{B}$ given by the
orientation of $B_{Z,W}$, then the height of $B$ is determined by the
equation 
\begin{displaymath}
  \delta _{3}(e)=\Ht(B)e^{\vee}.
\end{displaymath}
In conclusion, the Deligne splitting $\delta _{B}$ is characterized by
the invariants $\delta _{E_{Z}}$, $\delta _{C_{Z,W}}$, $\delta
_{E_{W}^{\vee}(2)}$, $\delta _{D_{Z,W}}$ and $\Ht(B)$. By duality, the
invariant $\delta_{E_{W}^{\vee}(2)}$ is determined by $\delta
_{E_{W}}$ and $\delta _{D_{Z,W}}$ by $\delta _{C_{W,Z}}$. So we will
concentrate in the computation of the invariants $\delta _{E_{Z}}$,
$\delta _{C}$ and $\Ht(B)$. 
By Lemma \ref{main-ht-formula} and equation \eqref{eq:44'}, we get
\begin{align}
\delta_{E_Z}(e)&=\frac{i}{2}\Pi_{E_Z,-2}\left(\overline{e}-e\right)
=\frac{i}{2}\Pi_{E_Z,-2}\left(\overline{e}\right),\label{eq:110}\\
\delta_{C}(e)&=\frac{i}{2}\Pi_{I^{-1,-1}_C}\left(\overline{e}-e\right)
=\frac{i}{2}\Pi_{I^{-1,-1}_C}\left(\overline{e}\right),\label{eq:80}\\
  \Ht(B)e^{\vee}&=-\frac{1}{2}\im
  \left(\Pi_{I^{-2,-2}_B}\left(\overline{e}\right)\right). \label{eq:81}
\end{align}
In this section we will concentrate in the computation of 
$\delta _{E_{Z}}(e)$, $\delta _{C_{W,Z}}(e)$ and $\Ht(B)$. Moreover, we will
show that, when the regulators of $Z$ and $W$ are zero, the height $\Ht(B)$ is
given by the higher archimedean height pairing.

\subsection{Computation of \texorpdfstring{$\delta_{E_Z}(e)$}{dEZe}.}
\label{sec:comp-delt}
We first compute
$\delta_{E_Z}(e)$ using the mixed Hodge structure arising from
\ref{eq:53}. To this end, we will find an element $e\in I^{0,0}_{E_{Z}}$
that is mapped to the standard generator of $\Q(0)$.
Most of the job has been done in Section
\ref{subsubsecdiff}. 
Let $\eta_{Z}$, $g_{Z}$ and $\theta _{Z}$ be the differential forms
provided by Proposition \ref{prop:4}. In particular,
\begin{displaymath}
\eta_Z\in F^0\Sigma_{B_X}E^{2p-1}_{X\times \P^1}(\log A_X\cap |Z|; p),
\end{displaymath}
with $d\eta_Z=0$. We claim that the class of $\eta_{Z}$,
\begin{displaymath}
  \{\eta_{Z}\}\in H^{2p-1}(X\times \P^{1}\setminus
  A_{X}\cup|Z|,B_{X};p) 
\end{displaymath}
gives us the sought element $e$.

By Proposition \ref{prop:4},  
the pair $(0,\eta_Z)$ is a cycle in the simple complex
associated to the morphism \eqref{eq:22}, representing
the cohomology class of $Z$. By the construction of $E_{Z}$, this
implies that $\{\eta_{Z}\}$ belongs to $E_{Z}$ and that it is
mapped to the standard generator of $\Q(0)$. We still need to show that
this class belongs to $I^{0,0}_{E_{Z}}$. By \eqref{eq:72} and the
shape of $E_{Z}$,
 \begin{displaymath}
  I^{0,0}_{E,Z}=F^{0}\cap \left( \overline {F^{0}} +
    \overline{F^{-1}}\cap W_{-2}\right).
\end{displaymath}
 By the construction of
$\eta_{Z}$, the class $\{\eta_{Z}\}$ belongs to $F^{0}$.
From the equation
\begin{equation}\label{eq:73}
  dg_Z=\frac{1}{2}(\eta_Z-\overline{\eta}_Z)-\theta_Z,
\end{equation}
with $\theta_Z \in
F^{-1}\Sigma_{B_X}E^{2p-1}_{X\times\P^1}(\log
A_X; p)$, and the fact that $H^{2p-1}(X\times\P^1\setminus
A_{X},B_{X};p)=H^{2p-2}(X;p)$ is pure of weight $-2$, 
we conclude that the cohomology class $\{\eta_Z\}$
belongs to $\overline {F^{0}}+\overline{F^{-1}}\cap W_{-2}$.
Hence $e\coloneqq\{\eta_Z\}\in I^{0,0}_{E_Z}$ is the
generator we are looking for.

Using again equation \eqref{eq:73} and the fact that the class
$\{\theta _{Z}\}$ of $\theta _{Z}$ belongs to $W_{-2}$, we
deduce that
\begin{equation}\label{eq:90}
  \delta _{E_{Z}}(e)=\frac{i}{2}\Pi _{E_Z,-2}(\overline e-e) =
  -i\widetilde{\theta _{Z}}=-i \widetilde{\Psi(\theta _{Z})},
\end{equation}
where, in the last equation, we are using the map $\Psi $ from
Definition \ref{def:3} to identify $H^{2p-1}(X\times\P^1\setminus
A_{X},B_{X};p)$ with $H^{2p-2}(X;p)$.  
Recall that, by Proposition \ref{propreg}, the class  
$\widetilde{\Psi(\theta_Z)}$ represents the Goncharov regulator of
$Z$. So, essentially, the invariant $\delta _{E_{Z}}(e)$ is the
regulator of $Z$. Note that the factor $i$ comes from the fact that in
the chosen normalization, the regulator is purely imaginary, while
the map $\delta $ is chosen to be real.

\begin{rmk}\label{Ezn}
Although we have written the above computation for $n=1$ to keep parity with
the rest of section \ref{sec:decomp-texorpdfstr}, since Section
\ref{subsubsecdiff} is valid for general $n\geq 1$, the same is true
for the above computation.
\end{rmk}

We now make the computation in the mixed Hodge structure $B_{Z,W}$ as
the techniques involved will be used latter in the computations of
the other invariants. As before, the key point is 
to find the generator $e$ of $I^{0,0}_{B}$. We see $B_{Z,W}$ as a
subquotient of
\begin{displaymath}
H^{2p}(\underline{\caX_{Z,W}}\setminus \overline Z, \overline W; p+1).
\end{displaymath}
Hence we will work on the smooth projective variety
$\caX_{Z,W}$ introduced in Section \ref{sec:case-n=m=1}. 

\begin{notation}\label{def:9} We choose $(t_{1},t_{2})$ affine
  coordinates of $\square^{2}$. We denote
  \begin{displaymath}
    \frac{dt_{1}}{t_{1}}, \frac{dt_{2}}{t_{2}}
    \in F^{0}\Sigma _{A}E^{1}_{(\P^{1})^{2}}(\log B;1).
  \end{displaymath}
  Recall, as Example \ref{exm:5}, that this implies
  \begin{equation}\label{eq:78}
    \overline{\left(\frac{d t_{1}}{t_{1}}\right)} =
   -\frac{d \bar t_{1}}{\bar t_{1}}
  \end{equation}
  Moreover, 
  when working with differential forms on the smooth projective variety
$\caX_{Z,W}$, that come from other spaces in diagram \eqref{eq:75}, we
will not write down explicitly the pullback map. For instance we will
denote by $\eta_{Z}$ the differential form $(\pi _{Z}\circ \pi
'_{1})^{\ast}\eta_{Z}$. Similarly $dt_{1}/t_{1}$ and $dt_{2}/t_{2}$
will also denote differential forms on $\caX_{Z,W}$. 
\end{notation}

We have the following characterization of $I^{0,0}_{B}$.
\begin{lem}\label{lemm:5}
An element $\xi\in B_{Z,W}$ belongs to $I_{B}^{0,0}$ if and only if
\begin{enumerate}
\item the condition 
$\xi \in F^0B_{Z,W}$ holds;
\item the image of $\xi$ in $E_Z$ belongs to $I^{0,0}_{E_Z}$.
\end{enumerate}
\end{lem}
\begin{proof}
The implication `only if' is clear from the fact that $\xi\in
I^{0,0}_{B}$ implies that $\xi\in F^{0}B_{Z,W}$, and that
$\rho\colon B_{Z,W}\rightarrow E_Z$ is a morphism of mixed Hodge
structures. To show the if part, we note first that $\ker(\rho)=
D_{Z,W}$. Since $D_{Z,W}$ is an extension of $\Q(1)^{s}$ by $\Q(2)$, 
we get
\begin{multline*}
  D_{Z,W}\subset F^{-2}\cap W_{-4}B_{Z,W}+F^{-1}\cap W_{-2} B_{Z,W}\\
  \subset F^{-2}\cap W_{-3}B_{Z,W}+F^{-1}\cap W_{-2} B_{Z,W}.
\end{multline*}
By assumption, $\xi\in F^0B_{Z,W}$, and we need to check that
\begin{displaymath}
  \overline{\xi}\in F^0B_{Z,W}+F^{-1}\cap W_{-2} B_{Z,W}+F^{-2}\cap W_{-3}B_{Z,W}.
\end{displaymath}
Now since also by assumption, $\rho(\xi)\in I^{0,0}_{E_Z}$,  we obtain
a $\xi'\in I_{B}^{0,0}$, such that 
$\rho(\xi')=\rho(\xi)$. Hence $\rho(\xi-\xi')=0$. Since $\rho$ is a
real map, we get $\rho
(\overline{\xi}-\overline{\xi'})=0$. Thus
$\overline{\xi}-\overline{\xi'}\in \ker(\rho)=D_{Z,W}$ and 
\begin{displaymath}
\overline{\xi}\in \overline{\xi'}+D_{Z,W}\subset F^0B_{Z,W}+F^{-1}\cap
W_{-2}B_{Z,W}+F^{-2}\cap W_{-3}B_{Z,W}, 
\end{displaymath}
as required. Hence $\xi\in I^{0,0}_{B}$, and the lemma follows.
\end{proof}
Now we have the following
\begin{prop}\label{propinv1}
Let $\eta_Z$ be as above, and write, using Notation \ref{def:9}, 
\begin{displaymath}
  \nu_Z\coloneqq -\eta_Z\wedge \frac{dt_2}{t_2}\in
  E^{2p}_{\caX_{Z,W}}(\log
    \overline {A_{2}}\cup \overline {B_{1}}\cup 
    \overline{D_{Z}}\cup \overline{Z};p+1).
\end{displaymath}
Then the cohomology class 
$\{\nu_Z\}$ is the generator $e$ of $I^{0,0}_B$ that is sent
to the canonical generator of $\Q(0)$. 
\end{prop}
\begin{proof}
  We first have to show that $\nu _{Z}$ belongs to
  \begin{displaymath}
    F^{0}\Sigma _{\overline {A_{1}}\cup \overline {B_{2}}\cup
      \overline{D_{W}}\cup \overline{W}}E^{2p}_{\caX_{Z,W}}(\log
    \overline {A_{2}}\cup \overline {B_{1}}\cup 
    \overline{D_{Z}}\cup \overline{Z};p+1).
  \end{displaymath}
  For this, the only point that has to be checked is that $\nu
  _{Z}|_{\overline {W}}$
  vanishes. As differential form $\nu _{Z}$ belongs to $F^{p+1}$, but
  \begin{displaymath}
    \dim (\overline W) = d+2-q = p.
  \end{displaymath}
  Therefore, $\nu _{Z}|_{\overline {W}}=0$. Since $\eta_{Z}$ is
  closed, the same is true for $\nu _{Z}$. By the explicit
  description of the isomorphism $\eqref{eq:3}$, we see that the class
  $\{\nu _{Z}\}$ is sent to $\{\eta_{Z}\}$. In
  particular to the canonical generator of $\Q(0)$.
  It remains to be
  shown that it belongs to $I^{0,0}_{B}$. The map $B_{Z,W} \to E_{Z}$
  sends that class $\{\nu _{Z}\}$ to the class $\{\eta_{Z}\}$ that
  belongs to $I^{0,0}_{E_{Z}}$. By Lemma \ref{lemm:5},
 $\{\nu _{Z}\}$  belongs to $I^{0,0}_{B}$ completing the proof.
\end{proof}
To compute $\delta _{E_{Z}}(e)$ using $\nu_{Z}$, it is easier to first
project to the
cohomology group 
\begin{displaymath}
  H^{2p}(\underline{\caX_{Z,W}}\setminus
  \overline{Z};p+1), 
\end{displaymath}
that is, we remove the condition of vanishing along $W$. In the complex
  \begin{displaymath}
    \Sigma _{\overline {A_{1}}\cup \overline {B_{2}}}E^{\ast}_{\caX_{Z,W}}(\log
    \overline {A_{2}}\cup \overline {B_{1}}\cup 
    \overline{D_{Z}}\cup \overline{Z};p+1),
  \end{displaymath}
    equations \eqref{eq:73} and \eqref{eq:78}, and the fact that
    $\eta_{Z}$ has odd degree, imply that 
  \begin{equation}\label{eq:76}
    \frac{1}{2}(\nu_Z-\overline{\nu}_Z)-\left(-\theta_Z\wedge
      \frac{dt_2}{t_2}\right) 
    =d\left(-g_Z\wedge \frac{dt_2}{t_2}
      +\frac{1}{2}(\log(t_2\overline{t}_2))\overline{\eta}_Z\right).
  \end{equation}
  From this equation, we conclude again that the invariant $\delta _{E_{Z}}(e)$
  is given by equation \eqref{eq:90}.

 \subsection{Computation of \texorpdfstring{$\delta_{C}(e)$}{dcs}}
 \label{sec-deltaint} 
 Since the form $\nu _{Z}$ represents the generator $e\in I_{B}^{0,0}$
 its image in $C_{Z,W}$ represents the generator $e\in I^{0,0}_{C}$.
 To compute this image, we project to the cohomology group $H^{2p+1}_{\underline
   Z}(\underline{\caX_{Z,W}},\overline{W})$. The class of $\nu _{Z}$
 is sent to the class of $(0,\nu _{Z})$. We know that the class of
 \begin{equation}\label{eq:82}
   \lambda _{Z}\coloneqq \frac{i}{2}(0,\overline \nu _{Z}-\nu _{Z})
 \end{equation}
 is sent to zero in the cohomology group $H^{2p+1}_{\underline
   Z}(\underline{\caX_{Z,W}})$. Therefore, according to equation
 \eqref{eq:80}, in order to compute $\delta _{C}(e)$, we need to find
 a preimage of  the class of $\lambda _{Z}$ in the group
 $H^{2p}_{\underline{S}}(\underline{W})$. Using Proposition
 \ref{prop0}, the fact that $\overline W$ is smooth and the standard
 description of the connection morphism associated to a short exact
 sequence, the method to find this preimage is the
 following. First we find a primitive of $\lambda _{Z}$ in the complex that
 computes the cohomology $H^{\ast}_{\underline
   Z}(\underline{\caX_{Z,W}})$, then we restrict this primitive to
 the relative scheme $\underline W$ and the class of this restriction
 will agree with
 $\delta _{C}(e)$.
 By equation \eqref{eq:76}, we have
 \begin{displaymath}
   \lambda _{Z}= d\left(i\theta _{Z}\wedge
   \frac{dt_{2}}{t_{2}},-ig_{Z}\wedge\frac{dt_{2}}{t_{2}}
   +\frac{i}{2}\log(t_{2}\overline t_{2})\eta_{Z}\right)
 \end{displaymath}
 Therefore, by the previous discussion, the class $\delta _{C}(e)$ is represented by
 \begin{equation}\label{eq:83}
   \left.\left(i\theta _{Z}\wedge
   \frac{dt_{2}}{t_{2}},-ig_{Z}\wedge\frac{dt_{2}}{t_{2}}
   +\frac{i}{2}\log(t_{2}\overline
   t_{2})\eta_{Z}\right)\right |_{\underline{W}}. 
 \end{equation}
 In order to compute explicitly the cohomology class represented by
 this form, we use that $S$ is disjoint with $\overline A_{1}\cup
 \overline B_{1}\cup \overline A_{2}\cup \overline D_{W}$, therefore
 \begin{equation}\label{eq:79}
   H^{\ast}_{\underline S}(\underline W)=H^{\ast}_{S}(\overline W).
 \end{equation}
 We write $S=\{(x_{j},t_{1,j},t_{2,j})\}_{j=1,\dots,s}$ and
 denote by $e_{j}$ the Betti generator of the term $\Q(1)_{\Q}$ corresponding to
 the point $(x_{j},t_{1,j},t_{2,j})$, for $j=1,\dots,s$. 
 We also
 denote by $\mu _{Z,j}$ the multiplicity of the cycle $Z$ in the
 component of $Z$ containing $(x_{j},t_{1,j})$. Using equation
 \eqref{eq:18}, we have the residue  
 computation 
 \begin{multline*}
   d\left[ \left. \left(-ig_{Z}\wedge\frac{dt_{2}}{t_{2}}
   +\frac{i}{2}\log(t_{2}\overline
   t_{2})\eta_{Z}\right)\right|_{\underline W} \right]=\\
\left[\left. i\theta _{Z}\wedge
   \frac{dt_{2}}{t_{2}}\right|_{\underline W}\right] -
 \frac{i}{2}\sum_{j=1}^{s} \log(t_{2,j}\overline t_{2,j})\mu
 _{Z,j}\delta _{(x_{j},t_{1,j},t_{2,j})}.
 \end{multline*}
 Since $\overline W$ is smooth, we can compute the cohomology
 \eqref{eq:79} using currents. 
 From the residue computation it follows that 
\begin{equation}\label{eq:91}
  \delta_{C}(e)=
\frac{i}{2}\sum_{j=1}^{s} \log(t_{2,j}\overline t_{2,j})\mu
 _{Z,j}\delta _{(x_{j},t_{1,j},t_{2,j})}=
  \frac{1}{4\pi}\sum^r_{j=1}\mu _{Z,j}\log(t_{2,j}\overline{t}_{2,j})e_{j}.
\end{equation}
In the second equality we are using the implicit de Rham generator 
carried by $\log(t_{2,j}\overline t_{2,j})$:
\begin{displaymath}
  \log(t_{2,j}\overline t_{2,j})=\log(t_{2,j}\overline t_{2,j})\otimes
  \bfone(1)_{\C} =\frac{1}{2\pi i} \log(t_{2,j}\overline t_{2,j})\otimes
  \bfone(1)_{\Q}. 
\end{displaymath}
As expected, the invariant $\delta_{C}(e)$ is real. 

\subsection{Computation of \texorpdfstring{$\text{ht}(B)$}{htBZW}}
\label{subsecgamma}
Since we will need to consider also the dual construction, we denote
by $e_{Z,W}$ the generator of $I^{0,0}_{B_{Z,W}}$ previously denoted
by $e$ and by $e_{Z,W}^{\vee}$ the generator of
$I^{-2,-2}_{B_{Z,W}}$. By Proposition \ref{propinv1}, we know that
$e_{Z,W}$ is represented by $\nu _{Z}$. By equation \eqref{eq:81} we
have that
\begin{equation}\label{eq:84}
    \Ht(B)e_{Z,W}^{\vee}=-\frac{1}{2}\im
  \left(\Pi_{I^{-2,-2}_B}\left(\overline{e}_{Z,W}\right)\right).
\end{equation}

We consider the dual mixed Hodge structure $B_{W,Z}(-2)=B_{Z,W}^{\vee}$
with decomposition, 
\begin{displaymath}
  B_{W,Z}(-2)_{\C} =J^{2,2}\bigoplus
  \left(\bigoplus_{l+s=2}J^{l,s}\right)\bigoplus J^{0,0}.
\end{displaymath}
Let $e_{W,Z}(-2)$ be the generator of $J^{2,2}$ that is mapped to the
generator $\bfone(-2)_{\Q}$ of $\Q(-2)_{\Q}$. It is constructed as in Section
\ref{sec:comp-delt} with $Z$ and $W$ swapped. It satisfies
the conditions 
\begin{align}
  \langle e_{W,Z}(-2), e_{Z,W}^{\vee}\rangle &=1,\label{eq:85}\\
  e_{W,Z}(-2) & \in \left( \bigoplus_{a+b = -2}  I^{a,b}_{B} \oplus
                I^{0,0}_{B}\right)^{\perp}.\label{eq:86}
\end{align}
Equations \eqref{eq:84}, \eqref{eq:85} and \eqref{eq:86} imply that
\begin{displaymath}
  \Ht(B)=-\frac{1}{2}\im \langle e_{W,Z}(-2), \overline
  e_{Z,W}\rangle. 
\end{displaymath}
The class $e_{Z,W}$ is represented by the form
\begin{displaymath}
\nu _{Z} \in F^{0}\Sigma _{\overline {A_{1}}\cup \overline {B_{2}}\cup
      \overline{D_{W}}\cup \overline{W}}E^{2p}_{\caX_{Z,W}}(\log
    \overline {A_{2}}\cup \overline {B_{1}}\cup 
    \overline{D_{Z}}\cup \overline{Z};p+1).
\end{displaymath}
while the class $e_{W,Z}$ is represented by
\begin{displaymath}
  \nu _{W}=-\eta_{W}\wedge\frac{dt_{1}}{t_{1}}\in
  F^{0}\Sigma _{\overline {A_{2}}\cup \overline {B_{1}}\cup
      \overline{D_{Z}}\cup \overline{Z}}E^{2q}_{\caX_{Z,W}}(\log
    \overline {A_{1}}\cup \overline {B_{2}}\cup 
    \overline{D_{W}}\cup \overline{W};q+1).
  \end{displaymath}
  Note that the subset where $\nu _{Z}$ may have logarithmic
  singularities agrees with the subset where $\nu _{W}$ vanishes and
  reciprocally. Therefore the differential form $\nu _{Z}\wedge \nu
  _{W}$ is locally integrable in $\caX_{Z,W}$, and the duality
  pairing is given by
  \begin{displaymath}
    \langle e_{W,Z}(-2), \overline e_{Z,W}\rangle =
    \frac{1}{(2\pi i)^{2}} (p_{\caX_{Z,W}}) _{\ast}[\nu_{W}\wedge
    \overline \nu_{Z}]=
    \frac{1}{(2\pi
      i)^{d+4}}\int_{\caX_{Z,W}}\nu_{W}\wedge \overline \nu_{Z},
  \end{displaymath}
  where $p_{\caX_{Z,W}}\colon \caX_{Z,W}\to \Spec(\C)$ is the
  structural map.
  In consequence, the height of
  $B_{Z,W}$ is given by
  \begin{multline}
    \label{eq:87}
    \Ht(B)=-\frac{1}{2}\im\frac{1}{(2\pi
      i)^{p+q+2}} \int_{\caX_{W,Z}}\nu_{W}\wedge \overline \nu_{Z}
    \\=\frac{1}{2}\im\frac{1}{(2\pi
      i)^{p+q+2}} \int_{\caX_{W,Z}}\eta_{W}\wedge
    \frac{dt_{1}}{t_{1}}\wedge \overline \eta_{Z}\wedge
    \frac{d\overline t_{2}}{\overline t_{2}}.
  \end{multline}
  Recall for the last equality that
  \begin{displaymath}
    \overline \nu _{Z} =\overline{-\eta _{Z}\wedge
      \frac{dt_{2}}{t_{2}}}=\eta_{Z}\wedge \frac{d\bar t_{2}}{\bar t_{2}} 
  \end{displaymath}

Using the fact that $g_Z|_{t_1=0}=g_Z|_{t_1=\infty}=0$ and that
$\eta_W|_{t_2=0}=\eta_W|_{t_2=\infty}=0$, the residue theorem, and the
relations 
\begin{displaymath}
d[g_Z]=\left[\frac{1}{2}(\eta_Z-\overline{\eta}_Z)-\theta_Z\right],
\quad d[\eta_{W}]= -\delta _{W},
\end{displaymath}
we have
\begin{multline*}
d\left[\eta_W\wedge \frac{dt_1}{t_1}\wedge g_Z\wedge
  \frac{d\overline{t}_2}{\overline{t}_2}\right]
= -\delta _W\wedge \frac{dt_1}{t_1}\wedge g_Z\wedge
  \frac{d\overline{t}_2}{\overline{t}_2}
+\frac{1}{2}\left[\eta_W\wedge \frac{dt_1}{t_1}\wedge \eta_Z\wedge
  \frac{d\overline{t}_2}{\overline{t}_2}\right]\\
-\frac{1}{2}\left[\eta_W\wedge \frac{dt_1}{t_1}\wedge \overline
  \eta_Z\wedge 
  \frac{d\overline{t}_2}{\overline{t}_2}\right]
-\left[\eta_W\wedge \frac{dt_1}{t_1}\wedge \theta _Z\wedge
  \frac{d\overline{t}_2}{\overline{t}_2}\right]. 
\end{multline*}
For type reasons the second term on the right hand side is zero
(as differential form, $\eta_{Z}$ is in $F^p$, $\eta_W$ is in $F^q$,
so the term is in $F^{p+q+1}$, but $p+q+1=d+3 > d+2$). Hence, by Stokes'
theorem, 
  \begin{multline}\label{formula-ht-B}
    \Ht(B)=\im\frac{-1}{(2\pi
      i)^{p+2}} \int_{\overline{W}}\frac{dt_1}{t_1}\wedge g_Z\wedge
    \frac{d\overline{t}_2}{\overline{t}_2}\\
    +\im \frac{-1}{(2\pi i)^{p+q+2}} \int_{\caX_{W,Z}}
    \eta_W\wedge \frac{dt_1}{t_1}\wedge \theta _Z\wedge
  \frac{d\overline{t}_2}{\overline{t}_2}.
  \end{multline}
The first term on the right hand side of the above equation resembles
the higher height pairing, and in fact, it agrees with the higher
height pairing, in case the real
regulators of the cycles are zero.

\begin{rmk}\label{rem:3}
  Although, to define the extension $B_{Z,W}$, we needed to go to the
  blow up $\caX_{Z,W}$ in order to be in local product situation and
  use duality, In the actual computation of $\Ht(B)$ we can
  remain in $X\times \P^{1}\times \P^{1}$. 
\end{rmk}

\subsection{Connection to the higher height pairing when the regulators are zero}\label{ss5ht}
In this subsection we want to compare $\Ht(B)$ to the higher
archimedean height pairing $\langle Z,W\rangle_{\Arch}$, when the real
regulator classes of $Z$ and 
$W$ are both zero and Assumption \ref{def:7} is satisfied. This can be
seen as a generalization of Hain's result \cite{Hain:Height} relating
the archimedean height pairing  for the usual cycles
homologous to zero with biextensions of mixed Hodge
structures.

Before doing the comparison, we need to put both invariants in the
same place. Recall that
\begin{displaymath}
  \langle Z,W\rangle_{\Arch}\in H^{1}_{\DB}(\Spec(\C);\R(2))=\Q(2)_{\C}/\Q(2)_{\R},
\end{displaymath}
while
\begin{displaymath}
  \Ht(B)\in \R.
\end{displaymath}
We denote by $\rho _{2}\colon \Q(2)_{\C}/\Q(2)_{\R}\to \R$ the
isomorphism given by
\begin{equation}
  \label{eq:89}
  \rho _{2}(v )= \im \left(\frac{v}{(2\pi i)^{2}}\right).
\end{equation}

\begin{thm}\label{compareheight}
  If the real regulators of $Z$ and $W$ are zero, then
  \begin{displaymath}
   \rho _{2}(\langle Z,W\rangle_{\Arch})=\Ht(B). 
  \end{displaymath}
\end{thm}
\begin{proof}
Since the real regulators of $Z$ and $W$ are zero, by Corollary
\ref{thetazero}, we can choose $g_{Z}$ and $\eta _{Z}$ with $\theta
_{Z}=0$ and the same for $W$. With this choice, after changing
the order of the terms,  equation \eqref{eq:87}
can be written as 
\begin{displaymath}
  \Ht(B)=\im\left(\frac{1}{(2\pi i)^{2}}(p)_{\ast}\left(\delta_W\wedge
    \frac{d\overline{t}_2}{\overline{t}_2}\wedge g_Z\wedge
    \frac{dt_1}{t_1}\right)\right).
\end{displaymath}
Since $n=m=1$ and the $\ast$-product is graded commutative, we have
that
\begin{displaymath}
  \langle Z,W\rangle_{\Arch}=-\langle W,Z\rangle_{\Arch}.
\end{displaymath}
By Corollary \ref{cor:refined-height}, for $n=m=1$ we have 
\begin{displaymath}
\langle Z,W\rangle_{\Arch}=-\langle W,Z\rangle_{\Arch}=
(p)_{\ast}\left(\delta_{W,\TW}\cdot
  W_1(t_{2})\cdot g_{Z,\TW}\cdot W_1(t_{1})\right)^\sim,
\end{displaymath}
as an element in $H^1_{\fD}(\Spec(\C),\R(2))$. Here
\begin{displaymath}
\begin{gathered}
W_1(t_{2})\coloneqq-\frac{1}{2}\left((\epsilon+1)\otimes \frac{dt_2}{t_2}+(\epsilon-1)\otimes \frac{d\overline{t}_2}{\overline{t}_2}+d\epsilon\otimes \log(t_2\overline{t}_2)\right),\\
W_1(t_{1})\coloneqq-\frac{1}{2}\left((\epsilon+1)\otimes \frac{dt_1}{t_1}+(\epsilon-1)\otimes \frac{d\overline{t}_1}{\overline{t}_1}+d\epsilon\otimes \log(t_1\overline{t}_1)\right),
\end{gathered}
\end{displaymath}
while 
\begin{displaymath}
  g_{Z,\TW}=\frac{\epsilon+1}{2}\otimes
  \eta_Z-\frac{\epsilon-1}{2}\otimes
  \overline{\eta}_Z+d\epsilon\otimes g_Z. 
\end{displaymath}

In order to prove the proposition, we need to unwrap the product in the $\TW$-complex and use Stokes' theorem. 
Since the pullback of $W$ in $X\times (\P^1)^2$ has dimension $p$, we get
\begin{displaymath}
\langle Z,W\rangle_{\Arch}=(p)_{\ast}\left(f(\epsilon)d\epsilon
  \otimes \delta_W\wedge
  \left(\Omega_1+\Omega_2+\Omega_3\right)\right), 
\end{displaymath}
where $f(\epsilon)=\frac{1}{4}(\epsilon^2-1)$ and
\begin{align*}
 \Omega_1&=-\frac{d\overline{t}_2}{\overline{t}_2}\wedge g_Z\wedge
          \frac{dt_1}{t_1} -\frac{dt_2}{t_2}\wedge g_Z\wedge \frac{d\overline{t}_1}{\overline{t}_1},\\
\Omega_2&=\frac{d\overline{t}_2}{\overline{t}_2}\wedge\frac{\eta_Z}{2}\log(t_1\overline{t}_1)
          -\frac{dt_2}{t_2}\wedge\frac{\overline{\eta}_Z}{2}\log(t_1\overline{t}_1),\\
\Omega_3&=\log(t_2\overline{t}_2)\frac{\eta_Z}{2}
          \wedge\frac{d\overline{t}_1}{\overline{t}_1}
          -\log(t_2\overline{t}_2)\frac{\overline{\eta}_Z}{2}\wedge
          \frac{dt_1}{t_1}. 
\end{align*}
In the computation above one has to take into account that $d\epsilon
$ anticommutes with forms of odd degree.
Now let
\begin{align*}
\Lambda_1&\coloneqq \delta _{W}\wedge d(\log(t_{2}\overline{t}_2))\wedge g_Z\log(t_1\overline{t}_1),\\
\Lambda_2&\coloneqq \delta _{W}\wedge \log(t_2\overline{t}_2)g_Z\wedge d(\log(t_{1}\overline{t}_1)).
\end{align*}
Then one can easily see that
\begin{displaymath}
  d\Lambda_1=\delta _{W}\wedge(\Omega_1-\Omega_2),\qquad
  d\Lambda_2=\delta _{W}\wedge(\Omega_3-\Omega_1). 
\end{displaymath}
Since our higher height pairing is an element of the Deligne
cohomology group, we conclude 
\begin{align*}
  \langle Z,W\rangle_{\Arch}
  &=(p)_{\ast}\left(f(\epsilon)d\epsilon
  \otimes \delta_W\wedge
    \left(\Omega_1+\Omega_1+\Omega_1\right)\right)\\ 
  &=3f(\epsilon)d\epsilon
    \otimes (p)_{\ast}\left(\delta_W\wedge\Omega_1\right)
\end{align*}
After integrating $f(\epsilon)$ from $0$ to $1$, we arrive at
\begin{displaymath}
\langle Z,W\rangle_{\Arch} = -\frac{1}{2}(p)_{\ast}\left(\delta_W\wedge\Omega_1\right)
\end{displaymath}
Finally, using (remember Notation \ref{def:1})
\begin{displaymath}
\overline{\delta_W\wedge\frac{d\overline{t}_2}{\overline{t}_2}\wedge
  g_Z\wedge \frac{dt_1}{t_1}}=-\delta_W\wedge\frac{dt_2}{t_2}\wedge
g_Z\wedge \frac{d\overline{t}_1}{\overline{t}_1}, 
\end{displaymath}
 we conclude
\begin{displaymath}
  \frac{1}{2}(p)_{\ast}\left(\delta_W\wedge\Omega_1\right)=
  -i\im(p)_{\ast}\left(\delta_W\wedge 
  \frac{d\overline{t}_2}{\overline{t}_2}\wedge g_Z\wedge
  \frac{dt_1}{t_1}\right).
\end{displaymath}
Hence we get
\begin{align*}
\rho _{2}(\langle Z,W\rangle_{\Arch}) &= \im\left(\frac{1}{(2\pi
    i)^{2}}(p)_{\ast}\left(\delta_W\wedge 
  \frac{d\overline{t}_2}{\overline{t}_2}\wedge g_Z\wedge
                                        \frac{dt_1}{t_1}\right)\right)\\
  &=\Ht(B).
\end{align*}
\end{proof}

\section{Examples of higher height pairing}
\label{sec:example}

\subsection{The case of dimension zero}
\label{sec:case-dimension-zero}

As a starter we discuss the case when $X=\Spec(\C)$, so $d=0$, and
$n=m=p=q=1$. Let $a,b\in \C\setminus \{0,1\}$ then $a$ and $b$ define
cycles in $Z^{1}(X,1)_{00}$ that we denote $Z$ and $W$. Moreover
these cycles always have proper 
intersection and satisfy Assumption \ref{def:7}. A choice of
differential forms satisfying the conditions of Proposition
\ref{prop:4} for the cycle $Z$ are
\begin{align*}
  \eta _{Z}&= \frac{d t}{t-1} - \frac{dt}{t-a}\in F^{0}\Sigma
             _{B}E^{1}_{\P^{1}}(\log A\cup |Z|;1)\\
  g_{Z}&=\log |t-1| - \log|t-a| +\log|a| \frac{1}{1+t\overline t}
         \in \Sigma
             _{B}E^{0}_{\P^{1}}(\log A\cup |Z|;1)
  \\
  \theta _{Z}& = -d\left( \log|a| \frac{1}{1+t\overline t}\right)
               =\log|a|\frac{\overline t dt + td\overline
               t}{(1+t\overline t)^{2}}\in F^{-1}\Sigma
             _{B}E^{1}_{\P^{1}}(\log A;1). 
\end{align*}
Note that the third term in the definition of $g_{Z}$ is added to
satisfy the condition  $g_{Z}(0)=0$ and is the responsible for the
presence of $\theta _{Z}$. Recall also Notation \ref{def:1}. With this
notation the complex conjugate of $\eta _{Z}$ is
\begin{displaymath}
  \overline \eta_{Z}=-\frac{d \overline t}{\overline t-1} -
  \frac{d\overline t}{\overline t-\overline a}.
\end{displaymath}
We denote by $\eta _{W}$, $g_{W}$ and $\theta _{W}$ the corresponding
differential forms for $W$ obtained by replacing $b$ for $a$.

Since $X=\Spec(\C)$, the relative products over $X$ are just absolute
products. Therefore there should not be any non trivial interaction
between $Z$ and $W$. As we will see, this is indeed the case.

We can choose $\caX_{Z,W}=\P^{1}\times \P^{1}$. The intersection
$\overline W \cap \overline Z$ is reduced to the point $(a,b)$. Since
$H^{0}(X;1)=\Q(1)$, the biextension $B_{Z,W}$ has the middle graded piece
\begin{displaymath}
  \Gr_{-2}^{W}B_{Z,W}=\Q(1)\oplus \Q(1).
\end{displaymath}
The first factor comes from the cohomology of $X$ and the second from
the intersection point.

The different invariants are easy to compute. We start with $\delta
_{E_{Z}}(e)$. This has to be a real element of $H^{0}(X;1)$. For
clarity, as in Definition \ref{def:1}, we will use explicitly the
generator $\bfone(1)_{\C}$ and write $\theta _{Z}=\theta'
_{Z}\otimes \bfone(1)_{\C}$ with
\begin{displaymath}
  \theta _{Z}'\in F^{0}\Sigma_{B}E^{1}_{\P^{1}}(\log A)
\end{displaymath}
given by the same formula as $\theta _{Z}$. Then, by equation \eqref{eq:90},
\begin{align*}
  \delta_{E_{Z}}(e)&=-i \Psi (\theta _{Z})=
  i\frac{1}{2\pi i}\int d\left( -\log|a| \frac{1}{1+t\overline
      t}\right)\wedge \frac{dt}{t}     \otimes   \bfone(1)_{\C}\\
  &=i\log|a|\otimes \bfone(1)_{\C}=\frac{1}{2\pi }\log|a|\otimes \bfone(1)_{\Q}. 
\end{align*}
This element is real as expected.

The invariant $\delta _{C}(e)$ is given by equation \eqref{eq:91}:
\begin{displaymath}
  \delta_{C}(e)=
  \frac{1}{2\pi}\log|b|\otimes \bfone(1)_{\Q}\in \Q(1)_{\C}.
\end{displaymath}

Finally we compute the height $\Ht(B)$. According to
\eqref{eq:87}, it is given by
\begin{multline*}
  \Ht(B)=\\\frac{1}{2}\frac{1}{(2\pi i)^4}\im
  \int_{(\P^{1})^{2}}
  \left(\frac{d t_{2}}{t_{2}-1} - \frac{dt_{2}}{t_{2}-b}\right)
  \wedge \frac{dt_{1}}{t_{1}}\wedge\left(
  \frac{d \overline t_{1}}{\overline t_{1}-1} - \frac{d\overline t_{1}}{\overline
    t_{1}-\overline a}\right)\wedge \frac{d\overline t_{2}}{\overline
  t_{2}}.
\end{multline*}
This integral can be computed separately in each variable. Since
\begin{displaymath}
  \frac{1}{2\pi i}\int_{\P^{1}}\left(\frac{d t}{t-1} -
    \frac{dt}{t-b}\right)\wedge \frac{d\overline t}{\overline t} =
  -\log|b|
\end{displaymath}
and 
\begin{displaymath}
  \frac{1}{2\pi i}\int_{\P^{1}}\frac{d t}{t} \wedge
  \left(\frac{d \overline t}{\overline t-1} - \frac{d\overline
      t}{\overline t-\overline a}\right) = 
  -\log|a|
\end{displaymath}
we obtain
\begin{displaymath}
\Ht(B)=\frac{1}{2(2\pi i)^{2}}\im(\log|a|\log|b|)=0  
\end{displaymath}
as we were expecting.

\subsection{An example in dimension 2}
\label{sec:an-example-dimension-2}

We next compute an example in $\P^2$. In this example $d=2$, $p=q=2$
and $n=m=1$. So condition \eqref{eq:57} is satisfied.

In this subsection we will present
the setting, in the next one we will develop the tools needed to
perform the computation using currents and in the last one we will
compute the main invariant associated with the biextension.

Let $X=\P^2$ and let $[x_0 :  x_1 :x_2]$ be homogeneous coordinates of
$\P^{2}$ and let 
\begin{align*}
s_0&=a_0x_0+a_1x_1+a_2x_2,\\
s_1&=b_0x_0+b_1x_1+b_2x_2,\\
s_2&=c_0x_0+c_1x_1+c_2x_2.\\
\end{align*}
be three linear global sections of $\mathcal{O}_{\P^2}(1)$ in general
position.
Let 
$\ell_i=\Div(s_i)$, $i=0,1,2$ be
the corresponding reduced divisors that we identify with their
support.
By general
position we mean that the lines $\ell_{1}$, $\ell_{2}$ and $\ell_{3}$
form a non-degenerate triangle.

For $i=0,1, 2 (\text{mod}\hspace{0.1cm}3)$ we write
\begin{displaymath}
  f_i=\frac{s_{i+1}}{s_{i+2}}
\end{displaymath}
for the rational function and 
$p_i=\ell_{i+1}\cap \ell_{i+2}$ for the intersection point. Note the equation $f_0\cdot f_1\cdot f_2=1$, which will be used later.

\begin{df} \label{def:12}
  Given a line $\ell$ and a rational function $f$ whose divisor does
  not contain $\ell$, we denote by $(\ell,f)\in Z^{2}(X,1)$ the
  pre-cycle 
  given as the graph of $f|_{\ell}$. Let $s_{0}$, $s_{1}$ and $s_{2}$
  be sections as before. We denote by
  \begin{displaymath}
    Z(s_{0},s_{1},s_{2})= \sum^2_{i=0} (\ell_i, f_i) - \sum ^2_{i=0} \pi
    _{X}^\ast (p_i).
  \end{displaymath}
  Moreover, if $\alpha \in \C^{\times}$, we write
  \begin{displaymath}
    Z(s_{0},s_{1},s_{2};\alpha )= (\ell_0, \alpha f_0)+(\ell_{1},f_{1})+(\ell_{2},f_{2}) -
    \sum ^2_{i=0} \pi _{X}^\ast (p_i). 
  \end{displaymath}
  In particular $Z(s_{0},s_{1},s_{2})=Z(s_{0},s_{1},s_{2};1)$.
\end{df}
The following lemma is an easy verification. 
\begin{lem}
  For $s_{0}$, $s_{1}$ and $s_{2}$ in general position and $\alpha \in
  \C^{\times}$, the pre-cycle $Z(s_{0},s_{1},s_{2};\alpha )$ is a cycle and
  belongs to $Z^{2}(X,1)_{00}$. 
\end{lem}
\begin{proof}
  The fact that $Z(s_{0},s_{1},s_{2};\alpha )$ is a cycle follows directly
  from the condition $\sum_{i}\Div(f_{i})=0$. The degenerate
  components $\sum ^2_{i=0} \pi _{X}^\ast (p_i)$ are subtracted
  precisely in order to fulfill the condition that
  $Z(s_{1},s_{2},s_{3};\alpha )$ belongs to the refined normalized complex.
\end{proof}

We define
\begin{align*}
  W_{\beta }\coloneqq Z(x_{0},x_{1},x_{2};\beta )
\end{align*}
and choose section $s_{0}$, $s_{1}$ and $s_{2}$ that are in general
position with respect to $\{x_{0},x_{1},x_{2}\}$ so that, for any
complex number 
$\alpha \in \C^{\times}$, if we write
$Z_{\alpha }=Z(s_{0},s_{1},s_{2};\alpha )$ then $Z_{\alpha }$ and $W_{\beta }$ satisfy
Assumption \ref{def:7}. We will also write $Z=Z_{1}$ and $W=W_{1}$.

The real regulator of $Z_{\alpha }$ is easy to compute.

\begin{prop}\label{reg-zero-triangle}
The real regulator class of $Z_{\alpha }$ in
\begin{displaymath}
H^{3}_{\DB}(\P^{2},\R(2))=H^{2}(\P^{2};2)_{\C}/H^{2}(\P^{2};2)_{\R}
\end{displaymath}
is represented by
the closed current $-\log|\alpha |\delta _{\ell}$, for any line $\ell$ in
$\P^{2}$. In particular, if $|\alpha |=1$ then the regulator class is zero.    
\end{prop}
\begin{proof}
  In the Thom-Whitney complex, the regulator of the cycle 
      $Z_{\alpha }$ is represented by $(\pi
      _{X})_{\ast}(\delta _{Z}\cdot W_{1})$. After taking the direct
      image and integrate with respect to $\epsilon $   we obtain  
  \begin{displaymath}
\caP(Z_{\alpha })=-\frac{1}{2}\left(\left(\log|\alpha |^2+\log|f_0|^2\right)
  \delta_{\ell_0}+\log|f_1|^2\delta_{\ell_1}
  +\log|f_2|^2\delta_{\ell_2}\right). 
\end{displaymath}
Since each $\delta _{\ell_{i}}$ is cohomologous to $\delta _{\ell}$
and, by construction $f_{0}f_{1}f_{2}=1$ we deduce the result.
\end{proof}

Let $\ell_{i}$, $f_{i}$ and $p_{i}$, $i=0,1,2$ be the lines, rational
functions and
intersection points constructed as before for the sections
$s_{0},s_{1},s_{2} $ and let $\ell_{i}'$, $f_{i}'$ and $p'_{i}$,
$i=0,1,2$ be the ones corresponding to the sections $x_{0},x_{i}$ and
$x_{2}$. For instance $\ell'_{0}=\{x_{0}=0\}$, $p'_{0}=[1:0:0]$ and
$f'_{0}=x_{1}/x_{2}$.    

For $i=0,1,2$ and $j=0,1,2$ we write $p_{i,j}=\ell_{i}\cap \ell'_{j}$,
\begin{displaymath}
  \alpha _{i}=
  \begin{cases}
    \alpha, &\text{ if }i=0,\\
    1, &\text{ otherwise,}
  \end{cases}\qquad
  \beta _{j}=
  \begin{cases}
    \beta,  &\text{ if }j=0,\\
    1, &\text{ otherwise,}
  \end{cases}
\end{displaymath}
and 
\begin{displaymath}
q_{i,j}=(p_{i,j}.\alpha _{i}f_{i}(p_{i,j}),\beta _{j}f'_{j}(p_{i,j}))\in X\times
\P^{1}\times \P^{1}.
\end{displaymath}
By the generality assumption, the set $S$ consist
of the nine points $q_{i,j}$. Moreover,
$H^{2p-2}(X;p)=H^{2}(\P^{2};2)=\Q(1)$.  
Therefore, the biextension $B=B_{Z_{\alpha },W_{\beta }}$ has the shape
\begin{align*}
\Gr_0B&= \Q(0),\\
\Gr_{-2}B&=\Q(1)\oplus \Q(1)^{\oplus_9},\\
\Gr_{-4}B&= \Q(2).
\end{align*}
From the description of the real regulator of $Z_{\alpha }$ above, the invariant
$\delta_{E_{Z_{\alpha }}}$ is given by
\begin{displaymath}
\delta_{E_{Z_{\alpha }}}(e)=i\left((\log|\alpha
  |+\log|f_0|)\delta_{l_0}+\log|f_1|\delta_{l_1}+\log|f_2|\delta_{l_2}\right)^\sim,
\end{displaymath}
while from equation \eqref{eq:91} the invariant $\delta _{C}(e)$ is given by 
\begin{displaymath}
\delta_C(e)=\frac{1}{2\pi}\sum_{i,j}\log|\beta _{j}f'_j(p_{ij})|e_{i,j},
\end{displaymath}
where $e_{i,j}$ is the generator of the cohomology with support on the
point $q_{i,j}$. The remaining invariant $\Ht(B)$ will be computed in section
\ref{subsec-triangle-height} after we discuss how to use currents to
ease the computation.

\subsection{Computation using currents}
In classical Arakelov geometry, it is usually simpler to write down 
explicitly a Green current for a cycle than to write a Green form with
logarithmic singularities for the cycle. Although in general, inverse images and
products of currents are not defined, the theory of wave front sets
sketched in Subsection \ref{sec:wave-front-sets} allow us, in some
situations, to work with currents with the same ease as with
differential forms. We will use the notations and results of section
\ref{sec:wave-front-sets}.

For simplicity we make the following enhancement of Assumption
\ref{def:7}.

\begin{assumption}\label{def:10}
  To Assumption \ref{def:7} we add the condition that $|Z|$ and $|W|$
  are both union of smooth subvarieties that intersect $A_{X}$ and $B_{X}$
  transversely. 
\end{assumption}
Note that Assumption \ref{def:10} is satisfied in the example
presented in Subsection \ref{sec:an-example-dimension-2}.

Hence we assume \ref{def:10} and we consider first the situation of
$Z$ in $X\times \P^{1}$. We 
denote by $t$ the absolute coordinate of $\P^{1}$, and, for 
shorthand,
$A=A_{X}$ and $B=B_{X}$. Since $|Z|=\bigcup Z_{i}$ is a union of
smooth components, we write $N_{0}^{\vee}|Z|=\bigcup N_{0}^{\vee}Z_{i}$.
Let $\iota \colon A \hookrightarrow X\times
\P^{1}$ be the inclusion and $\caS =\iota_{\ast}\iota^{\ast} N_{0}^{\vee}|Z|$.
Then $\caS $ is saturated with respect to $\iota$ by construction. The
fact that the $Z_{i}$ intersect $B$ transversely readily implies that
$\caS $ and $B_{X}$ are in good position.  So the hypothesis of Theorem
\ref{thm:2} are satisfied.

Let $g_Z$, $\eta_Z$ and
$\theta_Z$ be the differential forms obtained in Proposition \ref{prop:4}, 
They define
currents
\begin{align*}
  [\eta_{Z}]&\in F^{0}\Sigma _{B}D^{2p-1}_{X\times \P^1/A;\caS }(p)\\
  [\theta_{Z}]&\in F^{-1}\Sigma
                _{B}D^{2p-1}_{X\times \P^1/A;\caS }(p)\\
  [g_{Z}]&\in F^{-1}\cap\overline F^{-1} \Sigma _{B}D^{2p-2}_{X\times \P^1/A;\caS }(p)
\end{align*}
satisfying the differential equations
\begin{align*}
  d[\eta_{Z}]&=-\delta _{Z}\\
d[g_Z]&=\frac{1}{2}([\eta_Z] - [\overline{\eta}_Z])-[\theta_Z].
\end{align*}
In fact, in our situation, as the following result implies, any choice
of currents satisfying the above properties
is enough to compute the regulator of $Z$ and the invariant
$\Ht(B)$.

\begin{lem}\label{lemm:6}
  Let $\caS \subset T_{0}^{\vee}X$ be a closed conical subset that is
  saturated with respect to $\iota$ and is in good position with
  respect to $B$. Let
\begin{align*}
  \eta_{Z}'&\in F^{0}\Sigma _{B}D^{2p-1}_{X\times \P^1/A;\caS }(p)\\
  \theta_{Z}'&\in F^{-1}\Sigma
                _{B}D^{2p-1}_{X\times \P^1/A;\caS }(p)\\
  g_{Z}'&\in F^{-1}\cap\overline F^{-1} \Sigma _{B}D^{2p-2}_{X\times \P^1/A;\caS }(p)
\end{align*}
be currents 
satisfying the differential equations
\begin{align}
  d\eta_{Z}'&=-\delta _{Z}\label{eq:111}\\
dg_{Z}'&=\frac{1}{2}(\eta_Z' - \overline{\eta'}_Z)-\theta'_Z.\label{eq:103}
\end{align}
  then $\theta '_{Z}$ is closed and there are currents
  \begin{align*}
    v_{1}&\in \Sigma _{B}D^{2p-2}_{X\times \P^1/A;\caS }(p)\\
    v_{2}&\in F^{0}\Sigma _{B}D^{2p-2}_{X\times \P^1/A;\caS }(p)
  \end{align*}
satisfying
\begin{displaymath}
  dv_{1} = [\theta _{Z}]- \theta' _{Z},\qquad
  dv_{2}= [\eta_{Z}] - \eta '_{Z}.
\end{displaymath}
In particular $\theta '_{Z}$ represents the class of the regulator of
$Z$. 
\end{lem}
\begin{proof}
  By the properties of the involved forms and currents is easy to see
  that $[\eta _{Z}]-\eta _{Z}$ and $\theta _{Z}'$ are both
  closed. Moreover the current
  \begin{displaymath}
    ([\eta _{Z}]-\eta' _{Z})/2 - ([\theta _{Z}]-\theta '_{Z})
    - ([\overline \eta _{Z}]-\overline {\eta' _{Z}})/2
  \end{displaymath}
  is exact. By Theorem \ref{thm:2}, the cohomology group
  \begin{displaymath}
    H^{2p-1}(\Sigma_{B}D^{*}_{X\times \P^1/A;\caS}(p))
  \end{displaymath}
  is the de Rham part of a pure Hodge
  structure $H$ of weight $-2$. Moreover,
  \begin{displaymath}
    ([\eta _{Z}]-\eta' _{Z})/2\in F^{0}, \quad ([\theta _{Z}]-\theta
    '_{Z})\in F^{-1}\cap \overline F^{-1},\quad\text{and}\quad
    ([\overline \eta _{Z}]-\overline {\eta' _{Z}})/2\in \overline F^{0}.
  \end{displaymath}
  Since $H$ is pure of weight $-2$, there is a direct sum decomposition
  \begin{displaymath}
    H = F^{0}H\oplus F^{-1}\cap \overline F^{-1} H \oplus \overline F^{0}H.
  \end{displaymath}
  Therefore the three terms $[\eta _{Z}]-\eta' _{Z}$, $[\theta
  _{Z}]-\theta '_{Z}$ and $[\overline \eta _{Z}]-\overline {\eta'
    _{Z}}$ are exact. In particular we obtain the
  current $v_{1}$ in the statement. By theorems \ref{thm:1} and
  \ref{thm:2}, the differential of the above complex is strict with
  respect to the Hodge filtration. Therefore we can find a primitive
  $v_{2}$ of $[\eta _{Z}]-\eta _{Z}$ belonging to $F^{0}$, completing
  the proof of the result.
\end{proof}

\begin{rmk}
  Since $\WF(\delta _{Z})=N^{\vee}_{0}|Z|$ and the differential does
  not increase the wave front set, equation \eqref{eq:111} implies
  that, for the currents in the lemma to
  exist, a necessary condition is that $ N^{\vee}_{0}|Z|\subset
  \caS $. Clearly $\iota _{\ast}\iota ^{\ast}N^{\vee}_{0}|Z|\subset \caS $ is a
  sufficient condition for the currents to exist. In the explicit
  computation of next section it will be handy to have the freedom to
  enlarge $\caS $.  
\end{rmk}

We now put together $Z$ and $W$ to obtain the next result. Recall that
we are implicitly taking the pullbacks to $X \times (\P^{1})^{2}$. Let
$\caS _{Z},\caS _{W}\subset T^{\vee}_{\ast}X\times \P^{1}$ be closed conical
subsets that are saturated with respect to $A$, in good position
with respect to $B$ and such that $\pi _{1}^{\ast}\caS _{Z}\cap \pi
_{2}^{\ast}\caS _{W}=\emptyset$.

\begin{cor} \label{cor:5} Assuming \ref{def:10}, let $\eta _{Z}'$,
  $\theta '_{Z}$ and $g_{Z}'$ (respectively $\eta _{W}'$,
  $\theta '_{W}$ and $g_{W}'$) be currents satisfying the hypothesis of
  Lemma \ref{lemm:6} for the cycle $Z$ and the set $\caS _{Z}$
  (respectively $W$ and the set $\caS _{W}$). Then
  \begin{multline*}
    \Ht(B) = \frac{1}{2}\im\frac{1}{(2\pi
      i)^{2}} p_{\ast}\left(\eta'_{W}\wedge
    \frac{dt_{1}}{t_{1}}\wedge \overline \eta'_{Z}\wedge
    \frac{d\overline t_{2}}{\overline t_{2}}\right)\\
  =-\im \frac{1}{(2\pi
      i)^{2}} p_{\ast}\left(\delta _{W}\wedge \frac{dt_{1}}{t_{1}}\wedge g'_{Z}\wedge
    \frac{d\overline t_{2}}{\overline t_{2}} 
    +
    \eta'_{W}\wedge \frac{dt_{1}}{t_{1}}\wedge \theta '_{Z}\wedge
    \frac{d\overline t_{2}}{\overline t_{2}}\right),
  \end{multline*}
  where $p\colon X\times (\P^{1})^{2}\to \Spec(\C)$ is the structural
  map. 
\end{cor}
\begin{proof}
  That the product current is well defined follows from the fact that the
  wave front sets of the involved currents are disjoint. 
  By \eqref{eq:87}, we have
  \begin{equation}\label{eq:98}
    \Ht(B) = \frac{1}{2}\im\frac{1}{(2\pi
      i)^{2}} p_{\ast}\left([\eta_{W}]\wedge
    \frac{dt_{1}}{t_{1}}\wedge \overline{[\eta_{Z}]}\wedge
    \frac{d\overline t_{2}}{\overline t_{2}}\right).
  \end{equation}
  By Lemma \ref{lemm:6} there are currents
  \begin{displaymath}
    v_{Z}\in F^{0}\Sigma
    _{B}D^{2p-2}_{X/A;\caS }(p),\qquad
    v_{W}\in F^{0}\Sigma _{B}D^{2q-2}_{X/A;\caS}(q),
  \end{displaymath}
  such that
  \begin{equation}
    \label{eq:99}
     dv_{Z}=[\eta_{Z}] - \eta '_{Z},\qquad
  dv_{W}=[\eta_{W}] - \eta '_{W}.
  \end{equation}
  Since $v_{Z}$ belongs to $F^{0}$
  it has at least $p$ holomorphic differentials. As $W\times
  \P^{1}\subset X\times (\P^{1})^{2}$ has dimension $p$, we obtain 
  \begin{equation}\label{eq:100}
    \delta _{W}\wedge
    \frac{dt_{1}}{t_{1}}\wedge \overline v_{Z}\wedge
    \frac{d\overline t_{2}}{\overline t_{2}}=0
  \end{equation}
  Similarly
  \begin{equation}\label{eq:101}
    v _{W}\wedge
    \frac{dt_{1}}{t_{1}}\wedge \delta_{Z}\wedge
    \frac{d\overline t_{2}}{\overline t_{2}}=0.
  \end{equation}
  Then the result follows from Stokes' theorem by using equations
  \eqref{eq:98}, \eqref{eq:99}, \eqref{eq:100} and \eqref{eq:101} and
  the fact that the forms $v_{Z}$ and $\eta_{Z}$ vanish for $t_{1}=0$
  and $t_{1}=\infty$, while the forms $v_{W}$ and $\eta_{W}$ vanish for $t_{2}=0$
  and $t_{2}=\infty$.  
\end{proof}

\subsection{The invariant \texorpdfstring{$\Ht(B)$}{htB} of the example in dimension two}
\label{subsec-triangle-height}

Now that we have set up the theory, we are ready to compute the
remaining invariant $\Ht(B)$ for the pair of higher cycles $Z_{\alpha }$ and
$W_{\beta }$ described in \ref{sec:an-example-dimension-2}.

The first task is to compute a set of currents satisfying the
conditions of Lemma \ref{lemm:6} for the cycle $Z_{\alpha }$. The
currents for the cycle $W_{\beta }$ will be constructed in a similar
way. Since, for the moment we work with a single cycle we denote by
$t$ the absolute coordinate of $\P^{1}$ and we omit any needed
pullback to $X\times \P^{1}$ from the formulas.

We start with a classical Green
current for the cycle  $Z_{\alpha }$ in $X\times \square$:
\begin{displaymath}
  g_{Z_{\alpha },0}\coloneqq
  -\sum^2_{i=0}\left(\log\frac{|t-\alpha
      _{i}f_i|}{|t-1|}\right)\delta_{\ell_i} \in F^{-1}D^{2}_{X\times \P^{1}}(2).
\end{displaymath}
Then one can check that
\begin{displaymath}
  2\partial\bar \partial g_{Z_{\alpha },0}=
  \sum_{i=0}^{2}\delta _{(\ell_{i},\alpha _{i}f_{i})}-\delta
  _{p_{i}}-\delta _{\ell_{i}\times \{1\}}.
\end{displaymath}
Hence
\begin{equation}\label{eq:104}
  2\partial\bar \partial g_{Z_{\alpha },0}|_{X\times \square}= \delta _{Z_{\alpha }}.
\end{equation}
Moreover $g_{Z_{\alpha },0}|_{t=\infty}=0$. But in general $g_{Z,0}|_{t=0}\not =
0$. In fact,
\begin{equation}\label{eq:105}
  g_{Z_{\alpha },0}|_{t=0} = -\sum \log|f_{i}| \delta _{\ell_{i}} - \log|\alpha
  | \delta _{\ell _{0}}. 
\end{equation}
The two terms appearing in this decomposition have a different
nature. The first one, the sum, is a boundary, hence we will be able
to get rid of it without altering equation \eqref{eq:104}, while the second one is
the responsible for the real regulator of $Z_{\alpha }$ therefore will
force us a non zero current $\theta _{Z_{\alpha }}$. 

To see that the first term is a boundary 
We introduce the current 
\begin{equation}\label{eq:109}
u_Z=[\log|f_0|\partial\log|f_1|-\log|f_1|\partial\log|f_0|]\in F^{-1}D^{1}_{X}(2).
\end{equation}
This current does not depend on the choice of $\alpha $.
Using the fact that $2\partial
\overline{\partial}[\log|f_i|]=\delta_{\ell_{i+2}}-\delta_{\ell_{i+1}}$,
we get 
\begin{displaymath}
\overline{\partial}u_Z-\partial
  \overline{u}_Z=-\log|f_0|\delta_{\ell_0}
-\log|f_1|\delta_{\ell_1}+(\log|f_0|+\log|f_1|)\delta_{\ell_2}. 
\end{displaymath}
Finally, using the relation $f_0\cdot f_1\cdot f_2=1$, we get
\begin{equation}\label{eq:106}
\overline{\partial}u_Z-\partial \overline{u}_Z=-\sum^2_{i=0}\log|f_i|\delta_{\ell_i}.
\end{equation}
Let $h$ be the function
\begin{displaymath}
  h(t)=\frac{1}{1+|t|^2}.
\end{displaymath}
It is smooth in the whole $\P^{1}$ and satisfies
\begin{equation}
h(0)=1,\qquad h(\infty )=1.\label{eq:107}
\end{equation}

We define the currents
\begin{align*}
  g_{Z_{\alpha },1}
  &= -\sum^2_{i=0}\left(\log\frac{|t-\alpha
    _{i}f_i|}{|t-1|}\right) \delta_{
    \ell_i}-\left(\overline{\partial}(h(t)
    u_Z)- \partial (h(t)\overline{u}_Z)\right),\\
  g_{Z_{\alpha },2} &= h(t)\log|\alpha |\delta _{\ell_{0}}\\
  g_{Z_{\alpha }} &= g_{Z_{\alpha },1}+ g_{Z_{\alpha },2}.                   
\end{align*}
By equations \eqref{eq:105}, \eqref{eq:106} and \eqref{eq:107},
\begin{displaymath}
g_{Z_{\alpha }}|_{t=0}=g_{Z_{\alpha }}|_{t=\infty}=0.
\end{displaymath}
We also write
\begin{align*}
  \eta _{Z_{\alpha }}&= 2\partial g_{Z_{\alpha },1},\\
  \theta _{Z_{\alpha }}&= -d g_{Z_{\alpha },2}.
\end{align*}
Let $\iota \colon A_{X}\to X\times \P^{1}$ denote the inclusion and
$\caS=\iota _{\ast}\iota ^{\ast}\WF(g_{Z_{\alpha }})$.

\begin{prop} The set $\caS$ and 
  the currents $g_{Z _{\alpha }}$, $\theta _{Z_{\alpha }}$ and $\eta
  _{Z_{\alpha }}$ satisfy the hypothesis of Lemma \ref{lemm:6}.
\end{prop}
\begin{proof}
  By construction the set $\caS$ is saturated by respect to
  $A_{X}$. By examining the singularities of the different functions,
  the wave front set of $g_{Z_{\alpha }}$ is given by
  \begin{displaymath}
    \WF(g_{Z_{\alpha }}) = \bigcup_{i=0}^{2}\left(
    N^{\vee}_{0}(\ell_{i}\times \P^{1})\cup N^{\vee}_{0}(\ell_{i}\times
    \{1\}) \right) \cup N^{\vee}_{0}|Z_{\alpha }|
\end{displaymath}
Therefore, if $\iota'\colon B_{X}\to X\times \P^{1}$ is the inclusion,
then
\begin{displaymath}
  (\iota')^{\ast} \caS = \bigcup _{i=0}^{2}N^{\vee}_{0}(\ell_{i}\times
  \{0\})\cup N^{\vee}_{0}(\ell_{i}\times
  \{\infty\}). 
\end{displaymath}
Here the conormal bundle is computed in $B_{X}$. Let $r_{0}$ be the
retraction to $X\times \{0\}$ and $r_{\infty}$ the retraction to
$X\times\{\infty\}$. Since, for $i=0,1,2$ and $j=0,\infty$,
\begin{displaymath}
  s_{j}^{\ast}N^{\vee}_{0} (\ell_{i}\times \{j\}) N^{\vee}_{0}
  (\ell_{i}\times \P^{1})
\end{displaymath}
we deduce that $\caS$ and $B_{X}$ are in good position.

By construction, for $j=\emptyset, 0,1,2$,   $\overline g_{Z_{\alpha
  },j}=-g_{Z_{\alpha },j}$. Therefore
\begin{displaymath}
  2\bar \partial g_{Z_{\alpha },1}=-\overline{\partial g_{Z_{\alpha
      },1}} = -\bar \eta_{Z_{\alpha }}.  
\end{displaymath}
Therefore
\begin{displaymath}
  dg_{Z_{\alpha }}=\partial g_{Z_{\alpha },1}+\bar \partial
  g_{Z_{\alpha },1} + dg_{Z_{\alpha },2}=\frac{1}{2}\left(\eta
    _{Z_{\alpha }}-\bar \eta _{Z_{\alpha }}\right)-\theta _{Z_{\alpha
    }}.  
\end{displaymath}
The remaining hypothesis follow directly from the construction of the
different currents. 
\end{proof}

By Corollary \ref{cor:5}, the height of $B_{W_{\beta },Z_{\alpha }}$ is given by
\begin{displaymath}
  \Ht(B)=
-\frac{1}{(2\pi i)^2}
  \im  p_{\ast}\left(\delta _{W_{\beta }}\wedge \frac{dt_{1}}{t_{1}}\wedge g_{Z_{\alpha }}\wedge
    \frac{d\overline t_{2}}{\overline t_{2}} 
    +
    \eta_{W_{\beta }}\wedge \frac{dt_{1}}{t_{1}}\wedge \theta _{Z_{\alpha }}\wedge
    \frac{d\overline t_{2}}{\overline t_{2}}\right),
\end{displaymath}
The support of the current $g_{Z_{\alpha },0}$ is the union of the
threefolds $\ell_{i}\times (\P^{1})^{2}$.
Since we are assuming that the intersection of $Z_{\alpha }$ and $W_{\beta }$ is proper,
the intersection of $\overline W_{\beta }$ with this support is the union of
the lines $p_{ij}\times \P^1\times \{\beta _{j}f'_j(p_{ij})\}$ (see
Section \ref{sec:an-example-dimension-2} for the notation). 
Since the pullback of $\frac{d\overline{t}_2}{\overline{t}_2}$ to
these lines is zero, we obtain
\begin{displaymath}
  p_{\ast}\left(\delta _{W_{\beta }}\wedge \frac{dt_{1}}{t_{1}}\wedge g_{Z_{\alpha },0}\wedge
    \frac{d\overline t_{2}}{\overline t_{2}}\right)=0.
\end{displaymath}
We next compute
\begin{displaymath}
  I_{1}=p_{\ast}\left(\delta _{W_{\beta }}\wedge \frac{dt_{1}}{t_{1}}\wedge 
  \partial \bar u_{Z}\wedge
    \frac{d\overline t_{2}}{\overline t_{2}}\right).
\end{displaymath}
Using that $W\times \P^{1}$ has dimension 2, that $\delta
_{W_{\beta }}$ vanishes when restricted to $t_{2}=0$ and
$t_{2}=\infty$, and that $u_{Z}$ vanishes when restricted to
$t_{1}=\infty$, we obtain
\begin{displaymath}
  I_{1}=p_{\ast}\left(\delta _{W_{\beta }}\wedge \frac{dt_{1}}{t_{1}}\wedge 
  d \bar u_{Z}\wedge
  \frac{d\overline t_{2}}{\overline t_{2}}\right) = 
(p_{2})_{\ast} \left(\delta _{W_{\beta }}\wedge 
  \bar u_{Z}\wedge
  \frac{d\overline t_{2}}{\overline t_{2}}\right),
\end{displaymath}
where now $p_{2}$ is the structural morphism of the product $X\times \P^{1}$
where $W_{\beta }$ lives. Since the support of $W_{\beta }$ consist of
lines and $\bar u_{Z}$ contains one anti-holomorphic differential we
deduce $I_{1}=0$.

Next we consider
\begin{displaymath}
  I_{2}=p_{\ast}\left(\delta _{W_{\beta }}\wedge \frac{dt_{1}}{t_{1}}\wedge 
  \bar \partial u_{Z}\wedge
    \frac{d\overline t_{2}}{\overline t_{2}}\right).
\end{displaymath}
By the same argument as before
\begin{displaymath}
  I_{2}= (p_{2})_{\ast} \left(\delta _{W_{\beta }}\wedge 
  u_{Z}\wedge
  \frac{d\overline t_{2}}{\overline t_{2}}\right).
\end{displaymath}
This time the integral may be non-zero and we will compute it later. 
The last piece to consider is
\begin{displaymath}
  I_{3}=p_{\ast}\left(\delta _{W_{\beta }}\wedge \frac{dt_{1}}{t_{1}}\wedge 
  g_{Z_{\alpha },2}\wedge
  \frac{d\overline t_{2}}{\overline t_{2}}+
\eta _{W_{\beta }}\wedge \frac{dt_{1}}{t_{1}}\wedge 
  \theta _{Z_{\alpha }}\wedge
    \frac{d\overline t_{2}}{\overline t_{2}}
\right)
\end{displaymath}
Using that $\delta _{W_{\beta }}=-d\eta _{W_{\beta }}$, that $d
g_{Z_{\alpha },2}=-\theta _{Z_{\alpha }}$, and that
\begin{alignat*}{2}
  \delta _{W_{\beta }}|_{t_{2}=0}&= \delta _{W_{\beta
    }}|_{t_{2}=\infty}=0,&\qquad g_{Z_{\alpha
    }}|_{t_{1}=0}&=\log|\alpha |\delta _{\ell_{0}},\\
  \eta _{W_{\beta }}|_{t_{2}=0}&= \eta  _{W_{\beta
    }}|_{t_{2}=\infty}=0&\qquad
  g_{Z_{\alpha }}|_{t_{1}=\infty}&=0,\\
  \theta  _{Z_{\alpha }}|_{t_{1}=0}&= \theta  _{Z_{\alpha 
    }}|_{t_{1}=\infty}=0,
\end{alignat*}
we obtain
\begin{displaymath}
  I_{3}=(p_2)_{\ast}\left(\eta _{W_{\beta }}\wedge 
  \log|\alpha |\delta _{\ell_{0}}\wedge
  \frac{d\overline t_{2}}{\overline t_{2}}\right).
\end{displaymath}
Using $\eta _{W_{\beta }}=2\partial g_{W_{\beta },1}$, $\delta_{\ell_0}$
is closed, Stokes' theorem, and the fact that $\partial
g_{W_{\beta},1}\wedge \delta_{l_0}\wedge
\frac{d\overline{t}_2}{\overline{t}_2}$ is of type $(3,3)$, we deduce 
\begin{align*}
I_{3}&=-\log|\alpha|(p_2)_{\ast}\left(g_{W_{\beta},1}\wedge
       \delta_{\ell_0}\wedge
       d\left[\frac{d\overline{t}_2}{\overline{t}_2}\right]\right)\\
&=\log|\alpha|\left(-\sum^2_{j=0}\log|\beta_jf'_j(p_{0j})|-\sum^2_{j=0}\log|f'_{j}(p_{0j})|\right)\\
&=-\log|\alpha|\log|\beta|,
\end{align*}
using $f'_0f'_1f'_2=1$. So $\I(I_3)=0$, and we are reduced to the expression
\begin{displaymath}
\Ht(B)=\frac{1}{(2\pi i)^2} \I\left((p_{2})_{\ast} \left(\delta _{W_{\beta }}\wedge 
  u_{Z}\wedge
  \frac{d\overline t_{2}}{\overline t_{2}}\right)\right).
\end{displaymath}

Recall that the cycle $W$ has six components. The three degenerate
vertical components $V\coloneqq \sum ^2_{j=0} p^\ast (q'_j)$ and the
three lines $\sum^2_{j=0}(\ell'_j, \beta _{j}f'_j)$.
Since $u_{Z}|_{q'_{j}}=0$, we obtain $\delta_V\wedge u_Z\wedge
\frac{d\overline{t}}{\overline{t}}=0$.
Hence, we arrive at
\begin{displaymath}
\Ht(B)
=\frac{1}{(2\pi i)^2}\sum^2_{j=0}\I\left(p_{\ell'_j,\ast}\left[u_Z\wedge
    \frac{d\overline{t}_{2}}{\overline{t}_{2}}\right]\right),
\end{displaymath}
where $p_{\ell'_j}\colon \ell'_j\rightarrow \Spec(\C)$ is the structural morphism,
To compute the contribution of each line we use that 
$\ell'_j=V(x_j)$ and $f'_j=\frac{x_{j+1}}{x_{j+2}}$ to obtain the parametrizations 
\begin{alignat*}{2}
(\ell'_0, \beta _{0}f'_0) &= \{(0:1:t),(1:\beta_{0} t)\}\cong
\P^1;&\qquad t_{2}&=\beta_{0} t,\\
(\ell'_1,\beta _{1}f'_1) &= \{(t:0:1),(1:\beta_{1}t)\}\cong
\P^1;&\qquad t_{2}&=\beta_{1} t,\\ 
(\ell'_2, \beta _{2}f'_2) &= \{(1:t:0),(1:\beta_{2}t)\}\cong \P^1
&\qquad t_{2}&=\beta_{2} t.
\end{alignat*}
By symmetry we only need to compute the contribution of $(\ell'_0,
\beta _{0}f'_0)$ as the other two terms will be obtained by a cyclic
permutation of $\{0,1,2\}$.
Restricting to this line we obtain 
\begin{align*}
  f_0|_{(\ell'_0, \beta _{0}f'_0)}(t)
  &=\frac{b_1+b_2t}{c_1+c_2t}
  =\left(\frac{b_2}{c_2}\right)\frac{t-(-\frac{b_1}{b_2})}{t-(-\frac{c_1}{c_2})},\\
  f_1|_{(\ell'_0, \beta _{0}f'_0)}(t)&=\frac{c_1+c_2t}{a_1+a_2t} =
          \left(\frac{c_2}{a_2}\right)\frac{t-(-\frac{c_1}{c_2})}{t-(-\frac{a_1}{a_2})},\\
  \frac{d\overline t_{2}}{\overline t_{2}}|_{(\ell'_0, \beta _{0}f'_0)}&= \frac{d\overline t}{\overline t}.
\end{align*}
For shorthand we write
\begin{displaymath}
\gamma \coloneqq \frac{b_2}{c_2},\quad \rho
\coloneqq\frac{c_2}{a_2},\quad\theta_1\coloneqq-\frac{b_1}{b_2},\quad
\theta_2\coloneqq-\frac{c_1}{c_2},\quad \theta_3\coloneqq
-\frac{a_1}{a_2},
\end{displaymath}
and
\begin{displaymath}
\tilde{f}_0(t)=\frac{t-\theta_1}{t-\theta_2}, \qquad \tilde{f}_1(t)\coloneqq \frac{t-\theta_2}{t-\theta_3}.
\end{displaymath}
The differential form $u_Z|_{(\ell'_0, \beta _{0}f'_0)}$ splits up into
\begin{displaymath}
u_Z|_{(\ell'_0, \beta _{0}f'_0)}=u_{1,Z}+u_{2,Z},
\end{displaymath}
where
\begin{align*}
u_{1,Z}&=\log|\gamma |\partial\log|\tilde{f}_1(t)| -\log|\rho
         |\partial\log|\tilde{f}_0(t)|,\\
u_{2,Z}&=\log|\tilde{f}_0(t)|\partial(\log|\tilde{f}_1(t)|)
         -\log|\tilde{f}_1(t)|\partial(\log|\tilde{f}_0(t)|). 
\end{align*}
The current $p_{\ell'_0,\ast}\left[u_{1,Z}\wedge
  \frac{d\overline{t}}{\overline{t}}\right]$ is simple to compute:
\begin{align*}
  p_{\ell'_0,\ast}&\left[u_{1,Z}\wedge
    \frac{d\overline{t}}{\overline{t}}\right]\\
  &=\log|\gamma |p_{\ell'_0,\ast}\left(d\left[\log|\tilde{f}_1|
      \frac{d\overline{t}}{\overline{t}}\right]\right) -\log|\rho|
  p_{\ell'_0,\ast}\left(d\left[\log|\tilde{f}_0|
      \frac{d\overline{t}}{\overline{t}}\right]\right)\\ 
  &=\log|\gamma |\log\frac{|\theta_2|}{|\theta_3|}
  -\log|\rho|\log\frac{|\theta_1|}{|\theta_2|}. 
\end{align*}
Since this expression is purely real it does not contribute to the
height of $B_{W_{\beta },Z_{\alpha }}$. 
To compute $p_{\ell'_0,\ast}\left(u_{2,Z}(t)\wedge
  \frac{d\overline{t}}{\overline{t}}\right)$, we have to make a slight
digression to the theory of Bloch-Wigner dilogarithm function. For 
details the reader is referred to \cite{Zagier:dilog}. The
dilogarithm function is the holomorphic function defined, over the
disk $\D:=\{t\in \C\colon |t|<1\}$  as  
\begin{displaymath}
\Li_2(t)=\sum_{n\geq 1}\frac{t^n}{n^2}.
\end{displaymath}
This function can be extended as a holomorphic function to
$\C\setminus [1,\infty)$ with jumps $2\pi i\log|t|$. Thus the function
$\Li_{2,\text{arg}}(t)\coloneqq \Li_2(t)+i\text{arg}(1-t)\log|t|$
is continuous. The Bloch-Wigner dilogarithm is defined by taking the
imaginary part of $\Li_{2,\text{arg}}$:
\begin{multline*}
D_2(t)=\im(\text{Li}_2(t))+\arg(1-t)\log|t|\\
=\frac{1}{2i}(\Li_2(t)-\Li_2(\overline{t}))+
\frac{1}{4i}(\log(1-t)-\log(1-\overline{t}))
(\log(t)+\log(\overline{t})). 
\end{multline*}
We take the branch $-\pi\leq\arg(t)<\pi$. The Bloch-Wigner dilogarithm
satisfies the following partial differential equation.
\begin{displaymath}
\partial iD_2(t)=\log|t|\partial \log|1-t|-\log|1-t|\partial \log |t|.
\end{displaymath}
For any two linear rational functions $f,g$ in $\C(\P^1)$, we define
\begin{displaymath}
S(f,g)\coloneqq \log|f|\partial(\log|g|)-\log|g|\partial(\log|f|).
\end{displaymath}
We make the following observations: Let $f,g,h$ be three linear rational functions. Then
\begin{enumerate}
\item $S(f,g)=-S(g,f)$
\item $S(f,gh)=S(f,g)+S(f,h)$
\item $S(f, 1-f)=S(f,f-1)=\partial i D_2(f).$
\end{enumerate}
Using the above observations, we can find a boundary formula for
$S(\tilde{f}_0,\tilde{f}_1)$. First for rational
functions of the form $\frac{t-a}{b-a}$ and $\frac{t-b}{b-a}$ we get 
\begin{displaymath}
S\left(\frac{t-a}{b-a}, \frac{t-b}{b-a}\right)=S\left(\frac{t-a}{b-a},
  \frac{t-a}{b-a}-1\right)=\partial i
D_2\left(\frac{t-a}{b-a}\right).
\end{displaymath}
Hence
\begin{align*}
S(t-a,t-b)&=\partial i D_2\left(\frac{t-a}{b-a}\right)+S\left(b-a, \frac{t-b}{t-a}\right)\\
&=\partial\left(i D_2\left(\frac{t-a}{b-a}\right)+\log|b-a|\log\frac{|t-b|}{|t-a|}\right).
\end{align*}
Since $u_{2,Z}=S(\tilde{f}_0, \tilde{f}_1)$, we obtain
\begin{align*}
  u_{2,Z}
  &=S(t-\theta_1, t-\theta_2)-S(t-\theta_1,t-\theta_3)\\
   &\quad -S(t-\theta_2,t-\theta_2)+S(t-\theta_2,t-\theta_3)\\
  &=S(t-\theta_1,
    t-\theta_2)+S(t-\theta_2,t-\theta_3)+S(t-\theta_3,t-\theta_1)\\
  &=\partial(G(t)),
\end{align*}
where $G(t)$ is given by
\begin{multline*}
G(t)=i\left(D_2\left(\frac{t-\theta_1}{\theta_2-\theta_1}\right)+
  D_2\left(\frac{t-\theta_2}{\theta_3-\theta_2}\right)+
  D_2\left(\frac{t-\theta_3}{\theta_1-\theta_3}\right)\right)\\  
+\log|\theta_2-\theta_1|\log\frac{|t-\theta_2|}{|t -\theta_1|}
  +\log|\theta_3-\theta_2|\log\frac{|t-\theta_3|}{|t-\theta_2|}
  +\log|\theta_1-\theta_3|\log\frac{|t-\theta_1|}{|t-\theta_3|}.
\end{multline*}
Putting everything in place, we obtain
\begin{displaymath}
  p_{\ell'_{0,\ast}}\left[u_{2,Z}\wedge \frac{d\overline{t}}{\overline{t}}\right]
=p_{\ell'_{0,\ast}}\left(\left[dG(t)\frac{d\overline{t}}{\overline{t}}\right]\right)
=G(0)-G(\infty). 
\end{displaymath}
Noting that $G(\infty)=0$ since $D_2(\infty)=\log 1=0$, and using the
six-fold symmetry of Bloch-Wigner dilogarithm functions, we deduce
\begin{multline*}
  p_{\ell'_{0,\ast}}\left[u_{2,Z}(t)\wedge
    \frac{d\overline{t}}{\overline{t}}\right]
  =i\left(D_2\left(\frac{\theta_2}{\theta_1}\right) +
    D_2\left(\frac{\theta_3}{\theta_2}\right) +
    D_2\left(\frac{\theta_1}{\theta_3}\right) \right)\\
  +\left(\log|\theta_1|\log\frac{|\theta_1-\theta_3|}{|\theta_1-\theta_2|}
    +
    \log|\theta_2|\log\frac{|\theta_2-\theta_1|}{|\theta_2-\theta_3|}
    + \log|\theta_3|\log\frac{|\theta_3-\theta_2|}
    {|\theta_3-\theta_1|}\right).  
\end{multline*}
After plugging in the values of $\theta_1$, $\theta_2$ and $\theta_3$
and taking the imaginary part,

\begin{displaymath}
\im p_{\ell'_{0,\ast}}\left[u_{Z}(t)\wedge
  \frac{d\overline{t}}{\overline{t}}\right] =
D_2\left(\frac{b_2c_1}{b_1c_2}\right) +
  D_2\left(\frac{c_2a_1}{c_1a_2}\right) 
  + D_2\left(\frac{a_2b_1}{a_1b_2}\right).  
\end{displaymath}
Similarly, for $\ell'_1$ and $\ell'_2$, we have
\begin{align*}
\im p_{\ell'_{1,\ast}}\left[u_{Z}(t)\wedge
  \frac{d\overline{t}}{\overline{t}}\right]
  &=D_2\left(\frac{b_0c_2}{b_2c_0}\right)
  +D_2\left(\frac{c_0a_2}{c_2a_0}\right) +
  D_2\left(\frac{a_0b_2}{a_2b_0}\right),\\  
\im p_{l'_{2,\ast}}\left[u_{Z}(t)\wedge
  \frac{d\overline{t}}{\overline{t}}\right]
  &=D_2\left(\frac{b_1c_0}{b_0c_1}\right) +
    D_2\left(\frac{c_1a_0}{c_0a_1}\right) +
    D_2\left(\frac{a_1b_0}{a_0b_1}\right). 
\end{align*}
Summing up, the height of $B_{Z_{\alpha },W_{\beta }}$ is given by
\begin{displaymath}\label{htdilog}
\Ht(B)=\frac{1}{(2\pi i)^2}
\sum_{\substack{(0,1,2)\\(a,b,c)}}D_2\left(\frac{a_{2}b_{1}}{a_{1}b_{2}} \right),  
\end{displaymath}
where the sum is over all cyclic permutations of $(0,1,2)$ and
$(a,b,c)$ for a total of nine terms.

The expression above can be reduced to a six dilogarithms one by using
the five-term relation
for Bloch-Wigner dilogarithm. As a prototype, we show the
simplification for the first component of the above sum. Taking 
\begin{displaymath}
u\coloneqq \frac{b_2c_1}{b_1c_2},\qquad
v\coloneqq\frac{c_2a_1}{c_1a_2},\qquad w\coloneqq\frac{a_2b_1}{a_1b_2},
\end{displaymath}
we observe that $uvw=1$. Hence
$D_2(w)=D_2(1-\frac{1}{w})=D_2(1-uv)$. Now using the five-term
relation, we conclude 
\begin{displaymath}
D_2(u)+D_2(v)+D_2(w)=D_2\left(\frac{1-uv}{1-u}\right)+D_2\left(\frac{1-uv}{1-v}\right).
\end{displaymath}
Plugging back the values of $u,v$ and $w$, we get 
\begin{multline*}
D_2\left(\frac{b_2c_1}{b_1c_2}\right)+D_2\left(\frac{c_2a_1}{c_1a_2}\right)
+ D_2\left(\frac{a_2b_1}{a_1b_2}\right)=\\ 
D_2\left(\frac{a_2b_1-a_1b_2}{b_1c_2-b_2c_1} \left(\frac{c_2}{a_2}\right)\right)
+ D_2\left(\frac{a_2b_1-a_1b_2}{a_2c_1-a_1c_2} \left(\frac{c_1}{b_1}\right)\right).
\end{multline*}
Finally, putting everything together, we get a reduced expression

\begin{displaymath}
 \Ht(B) = \frac{1}{(2\pi i)^2} \sum_{(0,1,2)} D_2\left(\frac{a_2b_1-a_1b_2}{b_1c_2-b_2c_1} \left(\frac{c_2}{a_2}\right)\right)
+ D_2\left(\frac{a_2b_1-a_1b_2}{a_2c_1-a_1c_2} \left(\frac{c_1}{b_1}\right)\right),
\end{displaymath}
where the sum is over the cyclic permutations of $(0,1,2)$ only, giving us six terms.

\begin{rmk}
  From the formula for $\Ht(B)$ we can derive two conclusions.
  \begin{enumerate}
  \item Since $D_2(r)=0,\forall r\in \R$, we deduce that if the
    triangles are defined over $\R$ the height pairing is zero.
    In fact this is a general phenomenon as the Proposition
    \ref{prop:7} shows.

  \item Since the function $D_2$ can be extended to a continuous
    function on $\P^{1}(\C)$, the above height can be extended by
    continuity to any degenerate situation. We see in the next section
    that this is a very general phenomenon.
  \end{enumerate}
\end{rmk}

\begin{prop}\label{prop:7}
  Let $X$ be a smooth projective variety defined over $\R$ and
  $X_{\C}$ the corresponding complex variety. Let $Z\in
  Z^{p}(X_{\C},1)$ and $W\in Z^{q}(X_{\C},1)$ be two higher cycles
  defined also over $\R$ satisfying Assumption \ref{def:7}. Then
  \begin{displaymath}
    \Ht(B_{Z,W})=0.
  \end{displaymath}
\end{prop}
\begin{proof}
  The short proof is that, under the hypothesis of the proposition
  \begin{displaymath}
    \Ht(B_{Z,W})\in \rho _{2}(H^{1}_{\DB}(\Spec(\R);\R(2)))
  \end{displaymath}
  and $H^{1}_{\DB}(\Spec(\R);\R(2))=0$.

  In more down to earth
  terms. Let $\sigma \colon X_{\C}\to X_{\C}$ be the antilinear
  involution defined by the real structure of $X$. Assume for the
  moment that $Z$ and $W$ are not necessarily defined over $\R$. By the
  functoriality of the construction of mixed Hodge structures, we
  deduce that $B_{\sigma ^{\ast}Z,\sigma ^{\ast}W}=\overline {B_{Z,W}}$,
  where, for a mixed Hodge structure  $H$, we denote by $\overline H$
  the mixed Hodge structure obtained by sending $i$ to $-i$.

  Let now $B$ be any generalized biextension. Let
  $r=\ell(B)/2$. Then the operation $B\mapsto \overline B$ sends a
  generator $e$ of $\Q(a)$ to $(-1)^{a}e$ (see Remark \ref{rem:4}) and
  the map $\im$ is sent to $-\im$. Therefore
  \begin{displaymath}
    \Ht(\overline{B})= (-1)^{r+1}\Ht(B).
  \end{displaymath}
  In our case
  \begin{displaymath}
    \ell(B_{Z,W})/2=\frac{n+m}{2}+1=2.
  \end{displaymath}
  Therefore,
  \begin{displaymath}
    \Ht(B_{\sigma ^{\ast}Z},\sigma ^{\ast}W)=-\Ht(B_{Z,W}).
  \end{displaymath}
  But if $Z$ and $W$ are defined over $\R$ we also have  
  \begin{displaymath}
    \Ht(B_{\sigma ^{\ast}Z,\sigma ^{\ast}W})=\Ht(B_{Z,W}),
  \end{displaymath}
from which the proposition follows.
\end{proof}

\section{Asymptotic Behavior}
\label{sec:assymptotic mhs}

In this section, we begin the study of the asymptotic behavior of the
height of families of higher cycles.  In subsection \ref{subsec:ht-lim} we
prove the height extends continuously whenever the associated variation of
mixed Hodge structure is of Hodge--Tate type.  In subsection
\ref{subsec:ht-nilp} we give a definition of limit height for arbitrary
admissible variations of mixed Hodge structures over the punctured disk
with unipotent monodromy.  In subsection \ref{subsec:examples} we give three
examples of heights coming from (i) the dilogarithm variation, (ii) a particular
family of triangles in $\P ^2$ and (iii) a nilpotent orbit.  The first
two examples in subsection \ref{subsec:examples} can be read independently
of the rest of this section.  By definition, an oriented variation of
mixed Hodge structure is a variation equipped with a choice of flat, global
sections which induce an orientation on each fiber.

\subsection{Hodge--Tate Limits}
\label{subsec:ht-lim}

\begin{thm}\label{asymptotic-thm}  Let $S$ be a Zariski open subset of
a complex manifold $\bar S$ such that $D=\bar S-S$ is a normal crossing
divisor.  Let $\mathcal V\to S$ be an oriented Hodge--Tate variation
(graded-polarized) such that the length $\ell(\mathcal V) \ge
4$.
Assume $\mathcal V$ is admissible with respect to $\bar S$ and has unipotent
local monodromy about $D$. Let $p\in D$.  Then, the limit mixed
Hodge structure $\mathcal V_p$ of $\mathcal V$ at $p$ is an oriented
Hodge--Tate structure with the same weight filtration as $\mathcal V$.
Moreover,
\begin{equation}\label{eq:ht_continuity}
  \lim_{s\to p} \Ht(\mathcal V_s)= \Ht(\mathcal V_p).
\end{equation}
\end{thm}

To set up the machinery to prove Theorem~\ref{asymptotic-thm}, let
$p\in\bar S-S$.  Then we can find a polydisk $\Delta^r\subset\bar S$
containing $p$ and local holomorphic coordinates $(s_1,\dots,s_r)$ vanishing
at $p$ such that 
\begin{enumerate}
\item The image of $\Delta^r$ under $(s_1,\dots,s_r)$ is the unit
polydisk (coordinate norm $<1$) in $\C^r$;
\item $D\cap\Delta^r$ is given by the local equation $s_1\cdots s_k=0$.
\end{enumerate}
Therefore,
\begin{displaymath}
  \Delta^r - D\cap\Delta^r
  = \Delta^{*k}\times\Delta^{r-k}
  = \{s \mid s_1\cdots s_k\neq 0\}.
\end{displaymath}
      
As Theorem~\ref{asymptotic-thm} concerns the asymptotic behavior of the
variation, it is sufficient work on $\Delta^{*k}\times\Delta^{r-k}$.  We
therefore recall the theory of period maps of admissible variations of
graded-polarized mixed Hodge structures in this setting following the
conventions of~\cite{Pearlstein:vmhshf}.

Pick $b\in\Delta^{*k}\times\Delta^{r-k}$ and let $V=\mathcal V_b$ be the
fiber of $\mathcal V$ at $b$.  Let $T_j$ denote the local monodromy of
$\mathcal V$ about $s_j=0$.  We assume $T_j$ to be unipotent and write
$T_j = e^{N_j}$.  Note the $[N_a,N_b]=0$ since the fundamental group of
$\Delta^{*k}\times\Delta^{r-k}$ is abelian.

In analogy with the pure case, we can represent $\mathcal V$ by a period
map
\begin{displaymath}
  \varphi\colon \Delta^r-D\to\Gamma\backslash\mathcal M,
\end{displaymath}
where $\mathcal M$ the classifying space of mixed Hodge structure attached
to $\mathcal V$ with reference fiber $V$ and monodromy group $\Gamma$
generated by $T_1,\dots,T_k$.  As with variations of pure Hodge structures,
the classifying space $\mathcal M$ is a complex manifold and the period map
$\varphi$ is holomorphic, horizontal and locally liftable.  

Let $W$ denote the weight filtration of $\mathcal V$ and define
\begin{displaymath}
  \GL(V_{\C})^W
  =\{g\in \GL(V_{\C}) \mid g(W_k)\subseteq W_k,\quad\forall k\}.
\end{displaymath}
Let $\mathcal S_j$ denote the graded-polarization of $\Gr^W_j$ and define
\begin{displaymath}
  G = \{ g\in \GL(V_{\C}^W) \mid
  \Gr^{W}(g)\in\Aut_{\R}(\mathcal S_{\bullet})\,\}.
\end{displaymath}
Then (see \S 3, \cite{Pearlstein:vmhshf}) $G$ acts transitively on
$\mathcal M$ by biholomorphic transformations.

Let $G_{\R} = G\cap \GL(V_{\R})$ and $G_{\C}$ be the
complexification of $G_{\R}$.  The classifying space $\mathcal M$ is
a complex analytic open subset of a complex manifold $\check{\mathcal M}$
upon which $G_{\C}$ acts transitively by biholomorphisms.  Let
$\mathfrak g_{\C}$ be the Lie algebra of $G_{\C}$ and
$\mathfrak g_{\C}^F$ denote the isotopy subalgebra of elements
which preserve $F\in\check{\mathcal M}$.  Let $\mathfrak q$ be a vector space
complement to $\mathfrak g_{\C}^F$ in $\mathfrak g_{\C}$.  Then,
by the implicit function theorem, there exists a neighborhood
$\mathcal U$ of
$0\in \mathfrak q$ such that the map
\begin{displaymath}
     u\in\mathcal U\longmapsto e^u\cdot F\in\check{\mathcal M}
   \end{displaymath}
is a biholomorphism onto its image.     

Let $(z_1,\dots,z_k)$ denote the standard Euclidean coordinates on
$\C^k$ and $U^k\subset\C$ denote the product of upper
half-planes where $\I(z_1),\dots,\I(z_k)>0$.
Let $\Delta^{r-k}\subset\Delta^r$ be the locus where $s_1,\dots,s_k=0$
and
\begin{displaymath}
  (z,s) = (z_1,\dots,z_k,s_{k+1},\dots,s_r)
\end{displaymath}
be the corresponding coordinate system of $U^k\times\Delta^{r-k}$.
  
Let $U^k\times\Delta^{r-k}\to\Delta^{*k}\times\Delta^{r-k}$ be the
covering map
\begin{displaymath}
  (z_1,\dots,z_k,s_{k+1},\dots,s_r)\longrightarrow
  (e^{2\pi i z_1},\dots,e^{2\pi iz_k},s_{k+1},\dots,s_r),
\end{displaymath}
i.e. $s_j =e^{2\pi iz_j}$ for $j=1,\dots,k$.  Let $\eta_j$ be the covering
transformation $\eta_{j}(z,s) = (z+e_j,s)$ where $e_j$ is the $j$'th unit
coordinate vector in $\C^k$.  Set
\begin{displaymath}
  N(z) = z_1N_1 + \cdots +z_kN_k.
\end{displaymath}
By the local liftability of $\varphi$ there exists a holomorphic
map $F\colon U^k\times\Delta^{r-k}\to\mathcal M$ such that
$F(\eta_j(z,s)) = T_j\cdot F(z,s)$ which makes the following diagram commute
\begin{displaymath}
\begin{CD}
                 U @> F >> \mathcal M               \\
               @VVV         @VVV              \\
                 \Delta^*   @> \varphi >> \Gamma\backslash\mathcal M.
\end{CD}
\end{displaymath}
Accordingly, the formula $\tilde\psi(z,s) = e^{-N(z)}\cdot F(z,s)$ defines a map
$\tilde\psi\colon U^k\times\Delta^{r-k}\to\check{\mathcal M}$ such that
$\tilde\psi\circ\eta_j(z,s)=\tilde\psi(z,s)$. Therefore, $\tilde\psi$
descends to a holomorphic map
$\psi\colon \Delta^{*k}\times\Delta^{r-k}\to\check{\mathcal M}$.  By
admissibility~\cite{SZ:vmHsI,Kashiwara:vmHs}, $\psi$ extends to a
holomorphic map
$\Delta^r\to\check{\mathcal M}$ with limit Hodge filtration
\begin{equation}
  F_{\infty} = \lim_{s\to 0}\psi(s)\in\check{\mathcal M}.
  \label{limit-mhs}
\end{equation}

Let $N$ be an element of the monodromy cone
\begin{displaymath}
  \mathcal C=\Big\{\sum_j\, a_j N_j \,\big|\, a_1,\dots,a_k>0\Big\}.
\end{displaymath}
By admissibility, it follows that the relative weight filtration $M=M(N,W)$ 
of $N$ and $W$ exists, and together with  $F_{\infty}$ define a
graded-polarizable mixed Hodge structure $(F_{\infty},M)$.

The mixed Hodge structure $(F_{\infty},M)$ induces a mixed Hodge structure
on $\mathfrak g_{\C}$ with associated Deligne bigrading
\begin{displaymath}
  \mathfrak g_{\C} = \bigoplus_{a+b\leq 0}\mathfrak g_{\C}^{a,b}.
\end{displaymath}
In particular,
\begin{displaymath}
  \mathfrak g_{\C}^{F_{\infty}}
  = \bigoplus_{\substack{a\geq 0\\a+b\leq 0}}\mathfrak g_{\C}^{a,b}
\end{displaymath}
and hence
\begin{equation}
      \mathfrak q_{\infty}
      \coloneqq \bigoplus_{\substack{a<0\\a+b\leq 0}}\mathfrak g_{\C}^{a,b}
      \label{q-lim}
\end{equation}
is a vector space complement to $\mathfrak g_{\C}^{F_{\infty}}$ in
$\mathfrak g_{\C}$.  Therefore, it follows from equation
\eqref{limit-mhs} that for $s\sim 0$ we can write
\begin{displaymath}
  \psi(s) = e^{\Gamma(s)}\cdot F_{\infty},
\end{displaymath}
where $\Gamma(s)$ is a holomorphic function with values in
$\mathfrak q_{\infty}$ which vanishes at $s=0$.  Thus,
\begin{equation}
   F(z,s) = e^{N(z)}e^{\Gamma(s)}\cdot F_{\infty}.    \label{local-nf}
\end{equation}
See~\cite[\S 6]{Pearlstein:vmhshf} for a complete account of the
 constructions outlined in the previous paragraphs.

The final preliminary result we need is the following
\cite[Lemma 5.7]{HP:asymptotics} 
\begin{equation}
       [N_j,\left.\Gamma(s)\right|_{s_j=0}] = 0,  \label{hp-1}
\end{equation}
which follows from a straightforward consequence of horizontality and the
results established above.  Accordingly,
\begin{equation}
       [N_j,\Gamma(s)] = [N_j,\Gamma(s)-\left.\Gamma(s)\right|_{s_j=0}].
       \label{hp-2}
\end{equation}
Considering the power series expansion of $\Gamma(s)$ about $s=0$ we
see that $\Gamma(s)-\left.\Gamma(s)\right|_{s_j=0}$ is divisible by $s_j$.
Thus,
\begin{equation}
     s_j | [N_j,\Gamma(s)]    \label{hp-3}
\end{equation}  
in $\mathcal O(\Delta^r)$.

By induction one has the following result~\cite[8.11]{HP:asymptotics}):  Given a
multi-index $J=(a_1,\dots,a_k)$ with non-negative entries define 
\begin{displaymath}
  A_J = \prod_j\Ad(N_j)^{a_j}.
\end{displaymath}
and 
\begin{displaymath}
  s^{|J|} = \prod_{\{j\mid a_j\neq 0\}} s_j.
\end{displaymath}
Then
\begin{equation}
  s^{|J|}|A_J\Gamma.  \label{hp-4}
\end{equation}

Let $M(z,\bar z)$ be a monomial in $z_1,\dots,z_k$ and
$\bar z_1,\dots,\bar z_k$.  Let $\alpha(s,\bar s)$ be a real analytic
$\mathfrak{g}_{\C}$-valued  function
on $\Delta^r$ in the variables $s_1,\dots,s_r$ and $\bar s_1,\dots,\bar s_r$ which vanishes
at $s=0$.  Motivated by \eqref{hp-4} we say that the product
$M(z,\bar z)\alpha(s,\bar s)$ is a \emph{tame monomial} if, whenever $z_j$ or
$\bar z_j$ divide $M$, then either $s_j$ or $\bar s_j$ divides $\alpha$
(note: if $f$ is any $\mathfrak g_{\C}$ valued real analytic function,
then $z_j s_jf$, $z_j\bar s_jf$, $\bar z_j s_jf$, $\bar z_j\bar s_jf$ are all
tame monomials).  A \emph{tame polynomial} is a finite sum of tame monomials.  Let
$\mathcal T$ denote the set of all tame polynomials.

$\mathcal T$ is a complex vector space which is closed under complex
conjugation and taking Hodge components with respect to a fixed mixed Hodge
structure.  If $\beta\in\mathfrak g_{\C}$ and $\tau\in\mathcal T$ then
$[\beta,\tau]$ clearly belongs to $\mathcal T$.  By equation \eqref{hp-4},
the application of any polynomial in $\Ad(N(z))$ and
$\Ad(N(\bar z))$ to $\Gamma(s)$ is tame.

To see that $\mathcal T$ is closed under Lie bracket, note that if
$m_1\alpha_1$ and $m_2\alpha_2$ are tame monomials then
\begin{displaymath}
  [m_1\alpha_1,m_2\alpha_2] = m_1m_2[\alpha_1,\alpha_2].
\end{displaymath}
If $z_j$ or $\bar z_j$ divides $m_1m_2$ then $z_j$ or $\bar z_j$ must divide
either $m_1$ or $m_2$.  If $z_j$ or $\bar z_j$ divides $m_1$ then either
$s_j$ or $\bar s_j$ divides $\alpha_1$.  As such $s_j$ or $\bar s_j$
divides $[\alpha_1,\alpha_2]$.  The same argument applies to the case
where $z_j$ or $\bar z_j$ divides $m_2$.

Finally, if $\tau\in\mathcal T$ then
\begin{equation}
  \lim_{\substack{\I(z)\to\infty\\s\to 0}}\tau(z,s) = 0, \label{hp-5b}
\end{equation}  
where the limit is taken along sequences $(z(m),s(m))\in U^k\times\Delta^{k-r}$
such that $s(m)\to 0$,
$\I(z_1(m)),\dots,\I(z_k(m))\to\infty$ and
$\RE(z_1(m)),\dots,\RE(z_k(m))$ is constrained to a finite
interval.

We now specialize to the case where $\mathcal V$ is Hodge--Tate.
By the monodromy theorem~\cite[6.1]{Schmid:vhsspm} it follows that $N\in C$
acts trivially on each $\Gr^W_{2\ell}$ as $\Gr^W_{2\ell}$ is pure of type
$(\ell,\ell)$.  Therefore, by admissibility and Proposition $(2.14)$
of~\cite{SZ:vmHsI} it follows that the relative weight filtration $M=M(N,W)$
exists and equals $W$.  Accordingly, the limit Hodge filtration $F_{\infty}$
of $\mathcal V$ belongs to $\mathcal M$.  Therefore, the image of
$\psi $ is contained in $\mathcal M$.

\begin{rmk}\label{extended-limit} Since every element $N\in C$ acts trivially
on $\Gr^W$, the same holds for every 
\begin{displaymath}
  N\in\bar C = \{\,\sum_j\, a_j N_j \mid a_1,\dots,a_k\geq 0\}
\end{displaymath}
and hence $N\in\bar C$ implies that $M(N,W)=W$.  Therefore, $(\psi(s),W)$
is the limit Hodge structure at $s\in D\cap\Delta^r$.
\end{rmk}

Before continuing, we note that $F_{\infty}$ depends upon the choice of
local coordinates $(s_1,\dots,s_r)$.  The permissible changes of coordinates
which are compatible with the divisor structure result in the limit Hodge
filtration $F_{\infty}$ only being well defined up to transformation of the
form
\begin{equation}
    F_{\infty}\mapsto e^{N(\lambda)}\cdot F_{\infty},\qquad
    N(\lambda) = \sum_j\, \lambda_j N_j,    \label{limit-mhs-rescale}
\end{equation}
for some complex numbers $\lambda_1,\dots,\lambda_k$.
Since $\mathcal V$ is Hodge--Tate, $\Gr^W_k=0$
for odd $k$. Since
$\ell(\mathcal V)\ge 4$, by Corollary
\ref{ht-reparam} we have
\begin{displaymath}
  \Ht(e^{\lambda N}\cdot F_{\infty},W) = \Ht(F_{\infty},W).
\end{displaymath}

We conclude this section with the proof of Theorem~\ref{asymptotic-thm}.
\begin{proof}[Proof of Theorem~\ref{asymptotic-thm}]
  By Remark \ref{extended-limit} and the fact that $F_{\infty}\in
  \mathcal M$ we deduce that the limit mixed Hodge structure
  $(F_{\infty},M(N,W))=(F_{\infty},W)$ is Hodge--Tate and has the same
  weight filtration. So it only remains to be shown the continuity
  condition \eqref{eq:ht_continuity}.

  Returning to the subspace 
  \eqref{q-lim}, we see that since $\mathcal V$ is
Hodge--Tate and $F_{\infty}\in\mathcal M$, it follows that
\begin{displaymath}
  \mathfrak q_{\infty} = \bigoplus_{a<0}\,\mathfrak g_{\C}^{a,a}
  = \Lambda^{-1,-1}_{(F_{\infty},W)}
\end{displaymath}
in this case.  Accordingly, by \eqref{local-nf} and Lemma \ref{lambda-equiv}
we have
\begin{displaymath}
  Y_{(F(z,s),W)} = Y_{(e^{N(z)}e^{\Gamma(s)}\cdot F_{\infty},W)}
  = e^{N(z)}e^{\Gamma(s)}\cdot Y_{(F_{\infty},W)}
\end{displaymath}
and hence
\begin{displaymath}
  \overline{Y_{(F(z,s),W)}}
  = e^{N(\bar z)}e^{\bar\Gamma(s)}\cdot \overline{Y_{(F_{\infty},W)}}.
\end{displaymath}
Let $\delta = \delta_{(F_{\infty},W)}$ and $\delta(z,s) =
\delta_{(F(z,s),W)}$ as in \eqref{delta-def}. 
Then,
\begin{displaymath}
  \overline{Y_{(F(z,s),W)}}
  = e^{N(\bar z)}e^{\bar\Gamma(s)}e^{-2i\delta}\cdot Y_{(F_{\infty},W)}.
\end{displaymath}
On the other hand, by definition
\begin{displaymath}
  \overline{Y_{(F(z,s),W)}} = e^{-2i\delta(z,s)}\cdot Y_{(F(z,s),W)} 
  = e^{-2i\delta(z,s)}e^{N(z)}e^{\Gamma(s)}\cdot Y_{(F_{\infty},W)}.
\end{displaymath}
Comparing these two equations, it follows that
\begin{equation}                          
  e^{N(\bar z)}e^{\bar\Gamma(s)}e^{-2i\delta}\cdot Y_{(F_{\infty},W)}
  = e^{-2i\delta(z,s)}e^{N(z)}e^{\Gamma(s)}\cdot Y_{(F_{\infty},W)}.
  \label{two-sides-1} 
\end{equation}
By Proposition 2.2 of~\cite{CKS:dhs}, the group $\exp(W_{-1}\gl(V))$ acts simply
transitively on the set of gradings of $W$.  Therefore, equation 
\eqref{two-sides-1} implies that
\begin{equation}
     e^{N(\bar z)}e^{\bar\Gamma(s)}e^{-2i\delta}
    =  e^{-2i\delta(z,s)}e^{N(z)}e^{\Gamma(s)}.      \label{two-sides-2}
\end{equation}

The Hodge components of $\alpha\in\mathfrak g_{\C}$ relative to
$(F_{\infty},W)$ will be denoted $\alpha^{-b,-b}$.  For the remainder of this
proof, we constrain $\RE(z_1),\dots,\RE(z_k)$ to a finite interval.  

By the Campbell--Baker--Hausdorff formula (CBH),
\begin{equation}
  e^{N(\bar z)}e^{\bar\Gamma(s)} = e^{N(\bar z) + \bar\Gamma(s) + A(z,s)},
  \label{CBH-1}
\end{equation}
where $A(z,s)$ is a Lie polynomial with terms
$X=\Ad(X_1)\circ\Ad(X_{m-1})X_m$ where at least one
$X_j=\bar\Gamma(s)$ and the other $X_i$ are either $N(\bar z)$ or
$\bar\Gamma(s)$. Therefore, by the discussion following \eqref{hp-4}, $A(z,s)$
belongs to $\mathcal T$.  For future use, we observe that $A^{-1,-1}(z,s)=0$
since $A(z,s)$ is a sum of terms containing at least two elements from
$\mathfrak q_{\infty}=\oplus_{k>1}\,\mathfrak g_{\C}^{-k,-k}$.

Before continuing, observe that because each $N_j=N^{-1,-1}_j$ and
$\delta=\sum_{k>0}\,\delta^{-k,-k}$ the equation $[N_j,\delta]=0$ implies
$[N(\bar z),\delta^{-k,-k}]=0$ for all $k>0$. In particular,
\begin{equation}
  \Ad(L_1)\circ\dots\circ \Ad(L_{m-1})A(z,s)\in\mathcal T
  \label{hp-6}
\end{equation}
if each $L_j$ is either $-2i\delta$ or $N(\bar z)$ since $[N(\bar z),\delta]=0$
and $A(z,s)$ is itself constructed from Lie polynomials in
$\Ad(N(\bar z))$ and $\bar\Gamma(s)$.

More generally, any Lie polynomial
$U = \Ad(U_1)\circ\cdots\circ\Ad(U_{m-1})U_m$ where
each $U_j$ is either $N(\bar z)$, $\bar\Gamma(s)$, $A(z,s)$ and
$-2i\delta$ again belongs to $\mathcal T$.  Indeed, bracketing $\bar\Gamma(s)$
or $A(z,s)$ with $-2i\delta$ produces another element of $\mathcal T$.  By
the remarks of the previous paragraph, if $\bar\Gamma(s)$ does not appear
the result belongs to $\mathcal T$.  Finally, $\bar\Gamma(s)$ belongs to
$\mathcal T$, and $\mathcal T$ is closed under Lie brackets.  Application
of the Jacobi identity now shows that $U$ belongs to $\mathcal T$.

Continuing, by the CBH,
\begin{equation}      
     e^{N(\bar z) + \bar\Gamma(s) + A(z,s)}e^{-2i\delta}
     = e^{N(\bar z) + \bar\Gamma(s) + A(z,s)-2i\delta + B(z,s)},  \label{delta-B2}
\end{equation}
where $B(z,s)$ is a Lie polynomial with terms
$X=\Ad(X_1)\circ\Ad(X_{m-1})X_m$ where at least one
$X_j=-2i\delta$ and the other $X_i$ are either
$N(\bar z) + \bar\Gamma(s) + A(z,s)$ or $-2i\delta$.
Expanding out $X$ as a sum of terms 
$U = \Ad(U_1)\circ\cdots\circ\Ad(U_{r-1})U_r$ where
each $U_j$ is either $N(\bar z)$, $\bar\Gamma(s)$, $A(z,s)$ and
$-2i\delta$ it follows that $B(z,s)$ belongs to $\mathcal T$ by the
previous paragraph.  As was the case for $A$, $B^{-1,-1}(z,s)=0$ since
$B(z,s)$ is a sum of terms 
involving the Lie bracket of at least 2 elements of $\mathfrak q_{\infty}$.

Turning now to the right hand side of \eqref{two-sides-2}, by
\eqref{CBH-1}
\begin{displaymath}
  e^{N(z)}e^{\Gamma(s)} = e^{N(z)+\Gamma(s)+\bar A(z,s)}.
\end{displaymath}
Therefore,
\begin{equation}
  e^{-2i\delta(z,s)}e^{N(z)}e^{\Gamma(s)} =
  e^{-2i\delta(z,s) + N(z) + \Gamma(s)+ \bar A(z,s) + C(z,s)},
  \label{delta-C2}
\end{equation}
where $C(z,s)$ is a sum of terms
$X=\Ad(X_1)\circ\cdots\circ\Ad(X_{m-1})X_m$ with
some $X_j = -2i\delta(z,s)$ and the remaining terms $X_i$ either equal
to $-2i\delta(z,s)$ or to $ N(z) + \Gamma(s)+ \bar A(z,s)$.

Comparing \eqref{delta-B2} and \eqref{delta-C2} it follows that
\begin{equation}
  \begin{aligned}
    N(\bar z) + \bar\Gamma(s) & + A(z,s)-2i\delta + B(z,s) \\
    &= -2i\delta(z,s) + N(z) + \Gamma(s)+ \bar A(z,s) + C(z,s).
  \end{aligned}
  \label{delta-lim-1}
\end{equation}
Like with $A$ and $B$, we have $C^{-1,-1}(z,s)=0$.  Accordingly, taking the
$(-1,-1)$-component of equation \eqref{delta-lim-1} yields
\begin{displaymath}
  N(\bar z) + (\bar\Gamma)^{-1,-1}(s) -2i\delta^{-1,-1}
  = -2i\delta^{-1,-1}(z,s) + N(z) + \Gamma^{-1,-1}(s).
\end{displaymath}
Solving for $\delta^{-1,-1}(z,s)$ gives
\begin{equation}
  \delta^{-1,-1}(z,s) = N(\I(z)) + \I(\Gamma(s))^{-1,-1}
                     + \delta^{-1,-1}.        \label{delta-lim-2}
\end{equation}

Returning to equation \eqref{delta-C2} and noting that $A^{-1,-1}(z,s)=0$,
upon taking the $(-2,-2)$-component we obtain that
\begin{equation}
\begin{aligned}
  C^{-2,-2}(z,s)  &= \frac{1}{2}[-2i\delta^{-1,-1}(z,s),N(z)+\Gamma^{-1,-1}(s)] \\
  &= -i[N(\I(z)) + \I(\Gamma(s))^{-1,-1}
  + \delta^{-1,-1}, N(z)+\Gamma^{-1,-1}(s)] \\
  &= -i[N(\I(z)),\Gamma^{-1,-1}(s)]\\
  &\qquad -i[\I(\Gamma(s))^{-1,-1},
  N(z)+\Gamma^{-1,-1}(s)] \\
  &\qquad -i[\delta^{-1,-1},\Gamma^{-1,-1}(s)].
\end{aligned}
\label{C2-4}                     
\end{equation}
In particular, it follows from \eqref{C2-4} that
$C^{-2,-2}(z,s)$ belongs to class $\mathcal T$.

Taking $(-2,-2)$ components \eqref{delta-lim-1} implies 
\begin{multline*}
  (\bar\Gamma)^{-2,-2}(s) +A^{-2,-2}(z,s) + B^{-2,-2}(z,s) -2i\delta^{-2,-2} \\
  = -2i\delta^{-2,-2}(z,s)+\Gamma^{-2,-2}(s)+\bar
  A^{-2,-2}(z,s)+C^{-2,-2}(z,s)
\end{multline*}
and hence
\begin{displaymath}
  \delta^{-2,-2}(z,s) = \delta^{-2,-2} + D^{-2,-2}(z,s),
\end{displaymath}
where $D^{-2,-2}(z,s)$ belongs to the class $\mathcal T$. By
\eqref{hp-5b} we obtain that
\begin{displaymath}
  \lim_{\substack{\I(z)\to\infty\\s\to 0}} \delta^{-2,-2}(z,s) = \delta^{-2,-2} .
\end{displaymath}
Therefore,  we have completed the proof of Theorem
\ref{asymptotic-thm} in the case where $\ell(\mathcal V)=4$
(e.g. the dilogarithm variation in Example \ref{exam:6i}).

To verify the general statement,
we assume by induction that for $a=2,\cdots,k$ that
\begin{enumerate}
\item \label{item:9}  $C^{-a,-a}(z,s)$ belongs to class $\mathcal T$, and is given by
a Lie polynomial with terms
\begin{equation}
    \Ad(L_1)\circ\dots\circ \Ad(L_{r-1})L_r,  \label{hp-7}
\end{equation}
where each $L_j$ is either $\delta^{-b,-b}$, $N(z)$, $N(\bar z)$,
$\Gamma^{-b,-b}(s)$ or $\bar\Gamma^{-b,-b}(s)$.    
\item \label{item:10} $\delta^{-a,-a}(z,s) = \delta^{-a,-a} + D^{-a,-a}(z,s)$ where
$D^{-a,-a}(z,s)$ satisfies also condition \ref{item:9}. 
\end{enumerate}
The previous paragraphs establish the induction base $a=2$.

To establish the case $a=k+1$ we recall $C(z,s)$ is a sum of terms
$X=\Ad(X_1)\circ\cdots\circ\Ad(X_{m-1})X_m$ where
some $X_j = -2i\delta(z)$ and the remaining terms $X_i$ are either
$-2i\delta(z)$ or $N(z) + \Gamma(s)+ \bar A(z,s)$ (which occurs at least once).
In particular, upon expanding $\delta(z,s)$ into Hodge components, it follows
that $C^{-a-1,-a-1}(z,s)$ can be expanded into a sum of terms
\begin{displaymath}
  U = \Ad(U_1)\circ\dots \circ\Ad(U_{m-1})U_m
\end{displaymath}
of the required form \ref{item:9}.  It now follows from \eqref{delta-lim-1} and
the previous results about $A(z,s)$, $B(z,s)$ and $C^{-b,-b}(z,s)$ for
$b=1,\dots,k+1$ that \ref{item:10} holds.
\end{proof}

\subsection{Heights of Nilpotent Orbits}
\label{subsec:ht-nilp}

Let $\mathcal U\to\Delta^*$ be an admissible variation of mixed Hodge
structure over the punctured disk $\Delta^*$ with weight graded quotients
$\Gr^W_0=\Z(0)$, $\Gr^W_{-1}=\mathcal H$ and $\Gr^W_{-2}=\Z(1)$.
Assume that $\mathcal U$ has unipotent monodromy and select an embedding
of $\Delta^*$ into the coordinate disk
\begin{displaymath}
  \Delta = \{s\in\C \mid |s|<1\}
\end{displaymath}
as the complement of $s=0$.  In \S 3 of~\cite{BP:jumps},
the third author and P. Brosnan proved that there exists a rational number
$\mu$ such that
\begin{equation}
  h(s) = \Ht(\mathcal U_s) + \mu\log|s|    \label{limit-height-1}
\end{equation}
extends continuously to $\Delta$.  Moreover, $h(0)$ can be constructed
by pure linear algebra from the data of $(N,F_{\infty},W)$ of the nilpotent
orbit of $\mathcal U$.

Consider now an arbitrary oriented admissible variation $\mathcal V\to\Delta^*$
with unipotent monodromy.  As noted in \eqref{limit-mhs-rescale}, the data
$(N,F_{\infty},W)$ of the associated nilpotent orbit of $\mathcal V$
is only well defined up to replacing $F_{\infty}$ by $e^{\lambda N}\cdot F_{\infty}$.
In this section, we define a height $\Ht(N,F_{\infty},W)$ of an oriented
admissible nilpotent orbit $(e^{zN}\cdot F_{\infty},W)$ which generalizes the
construction of \cite{BP:jumps} and prove:

\begin{prop}\label{limit-ht-inv} If $\ell(\mathcal V)>2$ then, for any $\lambda \in \C$, 
\begin{displaymath}
  \Ht(N,e^{\lambda N}\cdot F_{\infty},W) = \Ht(N,F_{\infty},W).
\end{displaymath}
Thus $\Ht(N,F_{\infty},W)$ only depends on the variation $\mathcal{V}$
and not on a particular choice of limit Hodge filtration
$F_{\infty}$. 
If moreover $N$ acts trivially on $\Gr^W$ then $M(N,W) = W$ and
\begin{displaymath}
\Ht(N,F_{\infty},W) = \Ht(F_{\infty},M).
\end{displaymath}
On the right hand side $\Ht(F_{\infty},M)$ denotes the usual height of
the oriented extension $(F_{\infty},M)$.
\end{prop}

Accordingly, we can define the limit height of $\mathcal V$ to be
$\Ht(N,F_{\infty},W)$ of the associated nilpotent orbit. 

\begin{rmk} Unfortunately, we do not yet have the analog of
\eqref{limit-height-1} in general. In the next subsection we given an
example of an admissible nilpotent variation with weight graded quotients
$\Gr^W_0=\Z $, $\Gr^W_{-3}$ of rank 2 and $\Gr^W_{-6}\cong\Z (3)$
for which $\Ht(\mathcal V)$ grows like a multiple of $(\log|s|)^3$ as
$s\to 0$.  
\end{rmk}

To define the height of a nilpotent orbit, we will freely borrow from
section 6 of~\cite{BPR:nilpotent}.  The key concept is the notion of a Deligne
system, which originates from a letter of P. Deligne to E. Cattani
and A. Kaplan:

\begin{df} (6.6,~\cite{BPR:nilpotent})  Let $K$ be a field of characteristic
zero. A 1-variable Deligne system over K consists of the following data:
\begin{enumerate}
\item an increasing filtration $W$ of a finite dimensional $K$-vector
space $V$;
\item a nilpotent endomorphism $N$ of $V$ which preserves $W$ such that
the relative weight filtration $M = M(N,W)$ exists;  
\item a grading $Y$ of $M$ which preserves $W$ and satisfies $[Y,N]=-2N$.
\end{enumerate}
A morphism of Deligne systems $(W,N,Y)\to (\tilde W,\tilde N, \tilde Y)$ is
an endomorphism $T$ of the underlying $K$-vector spaces such that
\begin{displaymath}
T(W_i) \subseteq \tilde W_i,\qquad \tilde Y\circ T-T\circ Y =
0\qquad\text{and}\qquad \tilde N\circ T - T\circ N=0.   \label{deligne-funct}
\end{displaymath}
\end{df}
Given a Deligne system $(W,N,Y)$, each choice of grading $Y'$ of $W$
which commutes with $Y$ determines an $\lsl_{2}$-triple $(N_0,H,N_0^+)$ where
\begin{equation}
  N = \sum_{j\geq 0}\, N_{-j},\qquad [Y',N_{-j}] = -jN_{-j},
  \label{N-decomp}
\end{equation}
(so $N_0$ is the 0-eigencomponent of N relative to $\Ad Y'$) and
$H=Y-Y'$ (cf. equations 6.8 and 6.9 of \cite{BPR:nilpotent}).  The basic
construction of Deligne's letter is the following (see~\cite{BPR:nilpotent}
for additional history and references):

\begin{thm}[{\cite[6.10]{BPR:nilpotent}}]\label{deligne-system} Let $(W,N,Y)$
be a Deligne system. Then, there exists an unique functorial grading
$Y'= Y'(N,Y)$ of $W$ which commutes with $Y$ such that
\begin{equation}
    [N-N_0,N_0^+ ] = 0,     \label{sl2-triple}
\end{equation}
where $(N_0,H,N_0^+)$ is the associated $\lsl_2$-triple attached to
$Y'$ and $(W,N,Y)$.    
\end{thm}

In particular, given any admissible variation $\mathcal V$ of mixed
Hodge structure over the punctured disk $\Delta^*$ with unipotent monodromy,
we obtain a Deligne system $(W,N,Y)$ where $W$ is the weight filtration of
$\mathcal V$, $N$ is the local monodromy and $Y=Y_{(F_{\infty},M)}$
where $(F_{\infty},M) $
is the limit mixed Hodge structure of $\mathcal V$. If $\lambda
\in\C$, then 
$e^{\lambda N}$ is a morphism from $(W,N,Y)$ to $(W,N,Y+2\lambda N) =
(W,N,Y_{(e^{\lambda N}\cdot F_{\infty},M)})$. 
Therefore,
\begin{equation}
  Y'(N,Y_{(e^{\lambda N}\cdot F_{\infty},M)}) = e^{\lambda N}\cdot Y'(N,Y_{(F_{\infty},M)}).    \label{ds-1}
\end{equation}

We next proceed to the definition of the height of a nilpotent
orbit. So let $\mathcal M$ and $\check{\mathcal M}$ be the classifying
spaces of mixed Hodge structures of a filtered vector space $(V,W)$ and
its compact dual. Let $F\in \check{\mathcal M}$ and $N$ a nilpotent
endomorphism of $V$ such that $(e^{zN}\cdot F, W)$ is an admissible
nilpotent orbit. This means the following conditions
\begin{enumerate}
\item $N(F^{r})\subset F^{r-1}$ (horizontality,)
\item $e^{zN}\cdot F\in \mathcal M$ for $\im (z) \gg 0$,
\item the filtration $M=M(N,W)$ exists.
\end{enumerate}
Let $\max = \max(W)$, $\min=\min(W)$.  Assume $(e^{zN}\cdot F,W)$ is
oriented and $\ell = \ell(W)>2$.  We have a limit mixed Hodge
structure $(F,M)$. Let
$Y' = Y'(N,Y_{(F,M)})$ and $\delta = \delta_{(F,M)}$.  Write
\begin{equation}
  \delta = \sum_{j\geq 0}\delta_{-j},\qquad [Y',\delta_{-j}] = -j\delta_{-j}.
  \label{eigen-decomp}
\end{equation}
Note that this decomposition is with respect to a grading of $W$ and not
with respect to a grading of $M$.
We define the height of the admissible nilpotent orbit as 
\begin{equation}
    \Ht(N,F,W)e^{\vee} = \delta_{-\ell}\,e,   \label{limit-ht-def}
\end{equation}
where $e$ is a lift of the generator of $\Gr^W_{\max}$ and $e^{\vee}$ projects
to the generator of $\Gr^W_{\min}$. We stress here the fact that the
generators $e$ and $e^{\vee}$ as well as the grading $Y'$ correspond
to the filtration $W$, while the operator $\delta $ is defined by the
mixed Hodge structure $(F,M)$. We proceed in this way because there is
no reason for $(F,M)$ to be oriented.

\begin{proof}[Proof of Proposition \ref{limit-ht-inv}] Let
  $e^{zN}\cdot F$ be an admissible nilpotent orbit as before and
  $\lambda \in \C$.
  Let
  $\delta=\delta_{(F,M)}$ and
$\tilde\delta = \delta_{(e^{\lambda N}\cdot F,M)}$.  By Lemma \ref{rescale-by-N}
\begin{displaymath}
    \tilde\delta = \delta + \im(\lambda )N.
  \end{displaymath}
Moreover, since $N$ is a $(-1,-1)$-morphism of both $(F,M)$ and
$(e^{\lambda N}\cdot F,M)$ it follows that both $\delta$ and $\tilde\delta$
are fixed by the adjoint action of $e^{\lambda N}$.

Let $\delta=\sum_j\,\delta_j$ and $N = \sum_j\,N_j$ denote the
decomposition of $\delta$ and $N$ into eigencomponents with respect to
the adjoint action of $Y'=Y'(N,Y_{(F,M)})$ as in \eqref{eigen-decomp}.  Then,
\begin{equation}
     \tilde\delta
      = e^{\lambda N}\cdot \tilde\delta
      = e^{\lambda N}\cdot \sum_{j\geq 0}\, \delta_{-j} + \im(\lambda )N_{-j} 
      = \sum_{j\geq 0}\, e^{\lambda N}\cdot (\delta_{-j} + \im(\lambda )N_{-j}).
\label{ds-2}    
\end{equation}
Let $\tilde Y' = Y'(N,Y_{(e^{\lambda N}\cdot F,M)})$ and
\begin{displaymath}
  \tilde\delta = \sum_{j\geq 0}\,\tilde\delta_{-j},\qquad
  [\tilde Y',\tilde\delta_{-j}] = -j\tilde\delta_{-j},
\end{displaymath}
be the decomposition of $\tilde\delta$ into eigencomponents for
$\Ad\tilde Y'$.  By equation \eqref{ds-1}, $\tilde Y' = e^{\lambda N}\cdot Y'$.   
Moreover,
\begin{align*}
    {[e^{\lambda N}\cdot Y',e^{\lambda N}\cdot (\delta_{-j} + \im(\lambda )N_{-j})]}
    &=e^{\lambda N}\cdot [Y',\delta_{-j} + \im(\lambda )N_{-j}] \\  
    &=-j e^{\lambda N}\cdot (\delta_{-j} + \im(\lambda )N_{-j}).
\end{align*}
Comparing the previous equation with \eqref{ds-2} it follows that
\begin{equation}
  \tilde\delta_{-j} = e^{\lambda N}\cdot (\delta_{-j} + \im(\lambda )N_{-j}).  \label{ds-3}
\end{equation}

In the notation of \eqref{limit-ht-def}, we are interested in comparing
$\tilde\delta_{-\ell}$ and $\delta_{-\ell}$.  As a first step, we note that
$N$ acts trivially on $\Gr^W_{\max}$ and $\Gr^W_{\min}$ as each factor has
dimension 1 and $N$ is nilpotent.  As $N$ preserves $W$, it then follows
that $e^{\lambda N}$ fixes $\delta_{-\ell}$ and $N_{-\ell}$ under the adjoint action.
Thus,
\begin{displaymath}
  \tilde\delta_{-\ell} = \delta_{-\ell} + \I(\lambda)N_{-\ell}.
\end{displaymath}
       
The limit mixed Hodge structure $(F,M)$ induces on $\Gr^W$ the limit
mixed Hodge structures of the variations of pure Hodge structure on $\Gr^W$.
Let $2a=\max$ and $2b=\min$.  Then, $\Gr^W_{2a}$ is the constant variation
of type $(a,a)$ whereas $\Gr^W_{2b}$ is the constant variation of type
$(b,b)$.  Consequently, $F^a$ surjects on $\Gr^W_{2a}$ whereas $F^{a+1}$
maps to zero in $\Gr^W_{2a}$.  Moreover, $\Gr^W_{2b}=W_{2b}$ and $W_{2b}\subset F^{b}$
whereas $F^{b+1}\cap W_{2b}=0$.

By the previous paragraph, it follows that in equation
\eqref{limit-ht-def} we can arrange that $e\in F^a$.  By equation $(3.20)$
in~\cite{pearlstein:sl2}, $Y'$ preserves $F$.  Accordingly, since $N$ is
horizontal with respect to $F$, so is each eigencomponent $N_{-j}$.

Therefore,
$N_{-\ell}(e)\in F^{a-1}$.  But, $2a-2b>2$ implies $a-1>b$ and hence 
$N_{-\ell}(e)\in F^{b+1}\cap W_{2b}$.  Thus,
$N_{-\ell}(e)=0$. This proves the first
statement of Proposition \ref{limit-ht-inv}. 

Finally, if $N$ acts trivially on $\Gr^W$ then $N_0 = 0$ and hence
$H= Y-Y'=0$.  Therefore, $Y=Y'$ which implies $M=W$ and the decomposition
of $\delta$ with respect to $Y'$ is just the decomposition of $\delta$
with respect to $Y=Y_{(F_{\infty},M)}$.
\end{proof}

\subsection{Three Examples}
\label{subsec:examples}

In this subsection we show that the Bloch--Wigner dilogarithm $D_2$ is
the height of the dilogarithm variation over $\P ^1-\{0,1,\infty\}$.
We then show that up to a multiple of $4\zeta(2)$, we can express $D_2$
as the height of an elementary family of triangles of the type
considered in \ref{sec:an-example-dimension-2}. Finally, we show that the
height can become unbounded in the case where the underlying variation of
mixed Hodge structure is not unipotent in the sense of Hain and Zucker
\cite{HZ:uv}.

\begin{ex}\label{exam:6i}
Let $\mathcal V$ be the dilog variation over
$\P ^1-\{0,1,\infty\}$ (4.13, \cite{HZ:uv}).
Then, $\Ht(\mathcal V) = -D_2(s)$.

By (4.13, \cite{SZ:vmHsI}) we may select a basis $\{e_0,e_1,e_2\}$ of
$V_{\C} = \mathcal V_s$ such that $\mathcal V$ has bigrading
$I^{a,a} = \C e_{-a}$ and integral structure $V_{\Z }$ generated by 
\begin{align*}
        v_0(s) &= e_0-\log(1-s)e_1 + L_2(s)e_2,\\
        v_1(s) &= (2\pi i)(e_1 + \log (s)e_2), \\
        v_2(z) &= (2\pi i)^2 e_2,
\end{align*}
where $L_2(s) = \sum_{j=1}^{\infty}\,\frac{s^j}{j^2}$ is the dilogarithm.
By Lemma \ref{main-ht-formula} we need to compute
$\frac{1}{2}\text{Im}((e_0-\overline{e_0})_{-4})$.

Abbreviating $v_j(s)$ to $v_j$, it follows from the previous equations
that 
\begin{align*}
      e_2 &= (2\pi i)^{-2}v_2,  \\
      e_1 &= (2\pi i)^{-1} v_1-(2\pi i)^{-2}\log(s)v_2, \\
      e_0 &= v_0 +(2\pi i)^{-1}\log(1-s)v_1
             -(2\pi i)^{-2}[\log(1-s)\log(s)v_2 + L_2(s)]v_2.
\end{align*}
Therefore,
\begin{align*}
    e_0 &- \overline{e_0} \\
    &=  2(2\pi i)^{-1}\RE(\log(1-s))v_1
        -2i(2\pi i)^{-2}\I(\log(1-s)\log(s) + L_2(s))v_2 \\
    &=  2\RE(\log(1-s))(e_1+\log(s)e_2)
        -2i\I(\log(1-s)\log(s) + L_2(s))e_2.
\end{align*}
Accordingly,
\begin{align*}
    \frac{1}{2}\I((e_0 & -\overline{e_0})_{-4}) \\
       &= \RE(\log(1-s))\I(\log(s))
    -\I(\log(1-s)\log(s)) - \I(L_2(s)).
\end{align*}
To simplify the previous equation, let $\log(1-s) = A + iB$ and
$\log(s) = C + iD$.  Then,        
\begin{align*}
      \RE(\log(1-s))&\I(\log(s))
        -\I(\log(1-s)\log(s)) \\
       &= AD-(AD+BC) = -BC = -\arg(1-s)\log|s|.
\end{align*}
Thus,
\begin{displaymath}
     \Ht(\mathcal{V}_s)
     = -\I(L_2(s))-\arg(1-s)\log|s|
     =-D_2(s).
   \end{displaymath}
\end{ex}

\begin{ex}\label{exam:6ii} Returning to the setting of \ref{sec:an-example-dimension-2}, let
$W_{\beta}$ denote the standard triangle and consider the sections
\begin{align*}
s_{t,0}&=x_0+tx_1+x_2,\\
s_{t,1}&=x_0+x_1+tx_2,\\
s_{t,2}&=tx_0+x_1+x_2,
\end{align*}
of $\mathcal{O}_{\P^2}(1)$, where $t\in S =\P^1- \{-2,-1,0,1,\infty\}$.
Let $\ell_{t,i}=\Div(s_{t,i})$ for $i=0,1,2$ and consider the family of higher
cycles $\{Z_{\alpha}(t)\}_{t\in S}$, with individual $Z_{\alpha}(t)$ as defined in
subsection \ref{sec:an-example-dimension-2}. By the choice of $t$, all the
cycles $Z_{\alpha}(t)$ are non-degenerate and intersect $W_{\beta}$
properly and transversely. Moreover, the pair of cycles
$Z_{\alpha}(t)$, $W_{\beta}$ satisfies Assumption \ref{def:7}. 
Then,
\begin{displaymath}
    \Ht(B_{Z_{\alpha}(t),W_{\beta}})
     = \frac{3}{(2\pi i)^2}(D_2(t)+D_2(t)+D_2(t^{-2})).
   \end{displaymath}
To continue, recall that $D_2(z) = D_2(1-1/z)$ and hence
$D_2(t^{-2}) = D(1-t^2)$.  By the 5 term relation
\begin{displaymath}
  D_2(x)+D_2(y) + D_2\left(\frac{1-x}{1-xy}\right)
  +D_2(1-xy) + D_2\left(\frac{1-y}{1-xy}\right) =0.
\end{displaymath}
Setting $x=y=t$ it follows that
\begin{align*}
  D_2(t) + D_2(t) + D_2(1/t^2)
  &= D_2(t) + D_2(t) + D_2(1-t^2) \\
  &= -D_2\left(\frac{1-t}{1-t^2}\right)-D_2\left(\frac{1-t}{1-t^2}\right) \\
  &= -2D_2((1+t)^{-1}).
\end{align*}
Finally, $D_2(z) = -D_2(1/z)$ and $D_2(z) = -D_2(1-z)$.  Therefore,
\begin{displaymath}
  \Ht(B_{Z_{\alpha}(t),W_{\beta}})
  = \frac{6}{(2\pi i)^2}D_2(1+t)
  = \frac{-D_2(1+t)}{4\zeta(2)}
  = \frac{D_2(-t)}{4\zeta(2)}.
\end{displaymath}
In particular, upon setting $\theta=\pi/2$ in the formula
$
   D_2(e^{i\theta}) = \sum_{n=1}^{\infty}\, \frac{\sin(n\theta)}{n^2}
$
it follows that $D_2(\sqrt{-1})$ is equal to the Catalan constant $C$.
Thus,
\begin{displaymath}
  \Ht(B_{Z_{\alpha}(-\sqrt{-1}),W_{\beta}}) = \frac{C}{4\zeta(2)}.
\end{displaymath}
Also note that
\begin{displaymath}
\lim_{t\to p}\, \Ht(B_{Z_{\alpha}(t),W_{\beta}})=0
\end{displaymath}
for $p\in \{-2,-1,0,1,\infty\}$.
\end{ex}

To close this subsection, we give an example of an admissible nilpotent
orbit $(e^{zN}\cdot F,W)$ with weight graded quotients $\Gr^W_0\cong\Z(0)$,
$\Gr^W_{-3}$ of rank 2 and $\Gr^W_{-6}\cong \Z(3)$ such that the height grows
like $(\log|s|)^3$ for $s=e^{2\pi i z}$.

\begin{ex}\label{exam:6iii} Let $V_{\Z }$ be the lattice generated by $e_0$,
$e$, $f$ and $e_{-6}$.  Let
\begin{displaymath}
      W_{-6} = \Z  e_{-6},\qquad
      W_{-3} = W_{-6}\oplus\Z  f\oplus\Z  e,\qquad
      W_0 = V_{\Z },
    \end{displaymath}
with graded-polarizations
\begin{displaymath}
  \mathcal S_0([e_0],[e_0]) = \mathcal S_{-3}([e],[f])
  =\mathcal S_{-6}([e_{-6}],[e_{-6}]) = 1.
\end{displaymath}
Let $N$ be the nilpotent endomorphism obtained by setting
\begin{displaymath}
  N(e_0) = e,\qquad N(e)=f,\qquad N(f)=e_{-6},\qquad N(e_{-6})=0.
\end{displaymath}
Let $(F,M)$ be the mixed Hodge structure defined by setting
\begin{displaymath}
   I^{0,0} = \C e_0,\qquad
   I^{-1,-1} = \C e,\qquad
   I^{-2,-2} = \C f,\qquad
   I^{-3,-3} = \C e_{-6}.
 \end{displaymath}
Then, $N(I^{a,a})\subset I^{a-1,a-1}$ and hence $N$ is horizontal with respect
to $F$.  We also have $N(M_a)\subseteq M_{a-2}$.  To verify that $M$ is the
relative weight filtration $N$ and $W$ it remains to check $M$ induces the
monodromy weight filtration of $\Gr(N)$ shifted by $-k$ on $\Gr^W_k$.  This
is clear for $\Gr^W_0$ and $\Gr^W_{-6}$.  Let $\tilde N$ be the map induced
by $N$ on $\Gr^W_{-3}$ then
\begin{displaymath}
  W(\tilde N)_{-1} = \Z [f],\qquad
  W(\tilde N)_{1}  = \Gr^W_{-3},
\end{displaymath}
and hence $W(\tilde N)[3]_{-4} = W(\tilde N)_{-1}=\Z [f]$ while
$W(\tilde N)[3]_{-2} = W(\tilde N)_1 = \Gr^W_{-3}$.  Since $I^{-1,-1}=\C e$
and $I^{-2,-2}=\C f$ it follows that $M$ induces the correct filtration
on $\Gr^W_{-3}$.

Define
\begin{align*}
     \nu_0(z) &= e^{zN}(e_0)
               = e_0 + ze + \frac{1}{2}z^2 f + \frac{1}{6}z^3 e_{-6}, \\
   \nu_{-1}(z) &= e^{zN}(e)
               = e + zf + \frac{1}{2}z^2 e_{-6}, \\
   \nu_{-2}(z) &= e^{zN}(f)
               = f + ze_{-6}, \\
   \nu_{-3}(z) & = e^{zN}(e_{-6}) = e_{-6}.
\end{align*}
Then, $e^{zN}\cdot F^a = \oplus_{b\geq a}\,\C\nu_b(z)$.  Accordingly,
$e^{zN}\cdot F$ induces a pure Hodge structure of weight $k$ on $\Gr^W_k$:
For $\Gr^W_0$ and $\Gr^W_{-6}$ we just take the constant variations
of type $(0,0)$ and $(-3,-3)$.  The image $e^{zN}\cdot F^{-1}$ in $\Gr^W_{-3}$
is $\C[e+zf]$  which gives a variation of pure Hodge structure
of weight $-3$.

Recall (2.12, \cite{CKS:dhs}) that
\begin{displaymath}
  I^{p,q} = F^p\cap W_{p+q}
  \cap(\bar F^q\cap W_{p+q} + \overline{U^{q-1}_{p+q-1}}),
\end{displaymath}
where $U^a_b = \sum_{j\geq 0}\, F^{a-j}\cap W_{b-j}$.  In particular,
\begin{displaymath}
  I^{0,0}_{(e^{zN}\cdot F,W)} = \C\nu_0(z),\qquad
  I^{-1,-2}_{(e^{zN}\cdot F,W)} = \C\nu_{-1}(z),
\end{displaymath}
as both $e^{zN}\cdot F^0$ and $(e^{zN}\cdot F^{-1})\cap W_{-3}$ have rank 1.        

To determine $I^{-2,-1}$ note that
\begin{displaymath}
U^{-2}_{-5} = (\C\nu_{-2}(z))\cap W_{-5}\oplus
(\C\nu_{-3}(z))\cap W_{-6} = \C e_{-6}.
\end{displaymath}
Therefore,
\begin{displaymath}
  \overline{(e^{zN}\cdot F^{-1})}\cap W_{-3} + \overline{U^{-2}_{-5}}
  =\C\bar\nu_{-1}(z)\oplus\C e_{-6}
  =\C(e+\bar zf)\oplus\C e_{-6},
\end{displaymath}
and hence
\begin{align*}
      (e^{zN}\cdot F^{-2})\cap &W_{-3}
      \cap(\overline{(e^{zN}\cdot F^{-1})}\cap W_{-3} + \overline{U^{-2}_{-5}}) \\
      &= (\C\nu_{-1}(z)\oplus\C\nu_{-2}(z))
         \cap(\C(e+\bar zf)\oplus\C e_{-6}) \\
      &=\C(e+\bar zf + z(\bar z-\frac{1}{2}z)e_{-6}),
\end{align*}
because $e+\bar zf + z(\bar z-\frac{1}{2}z)e_{-6}
= \nu_{-1}(z) + (\bar z-z)\nu_{-2}(z)$.  As such,
$I^{-1,-2}_{(e^{zN}\cdot F,W)}\oplus I^{-2,-1}_{(e^{zN}\cdot F,W)}$ is spanned
by $\nu_{-1}(z)$ and $\nu_{-2}(z)$. 
Moreover $I^{-3,-3}_{(e^{zN}\cdot F,W)}= I^{-3,-3}$ is generated by
$e_{-6}$. 

To finish, observe that
\begin{displaymath}
  \nu_0(z) -\bar\nu_0(z)
  = (z-\bar z)(e+\frac{1}{2}(z+\bar z)f
  +\frac{1}{6}(z^2 + z\bar z + \bar z^2)e_{-6}).
\end{displaymath}
Next,
\begin{align*}
    e+\frac{1}{2}(z+\bar z)f
    &+\frac{1}{6}(z^2 + z\bar z + \bar z^2)e_{-6} \\
    &= \nu_{-1}(z)+\frac{1}{2}\nu_{-2}(z)+\frac{1}{6}(z-\bar z)^2 e_{-6}.
\end{align*}
Thus,
\begin{displaymath}
     (\nu_0(z)-\bar\nu_0(z))_{-6}
     = \frac{1}{6}(z-\bar z)^3 e_{-6}.
   \end{displaymath}
where $(\cdots)_{-6}$ is projection onto $I^{-3,-3}_{(e^{zN}\cdot F,W)}$ with
respect to the Deligne bigrading of $(e^{zN}\cdot F,W)$.
Write now $s=e^{2\pi iz}$, then the nilpotent orbit  $(e^{zN}\cdot F,W)$
defines a variation of mixed Hodge structures $\mathcal{V}$ over the
punctured unit disk with coordinate $s$. Then, by
\eqref{main-ht-formula},
\begin{displaymath}
  \Ht(\mathcal{V}_{s})= \frac{1}{12\pi ^{3}}(\log|s|)^{3}.
\end{displaymath}
 We note also that, since the
mixed Hodge structure $(F,M)$ is split over $\R$, then $\delta
_{(F,M)}=0$. Therefore, in this case
\begin{displaymath}
  \Ht(N,F,W)=0.
\end{displaymath}
\end{ex}      

\newcommand{\noopsort}[1]{} \newcommand{\printfirst}[2]{#1}
  \newcommand{\singleletter}[1]{#1} \newcommand{\switchargs}[2]{#2#1}
  \def\cprime{$'$}
\providecommand{\bysame}{\leavevmode\hbox to3em{\hrulefill}\thinspace}
\providecommand{\MR}{\relax\ifhmode\unskip\space\fi MR }

\providecommand{\MRhref}[2]{%
  \href{http://www.ams.org/mathscinet-getitem?mr=#1}{#2}
}
\providecommand{\href}[2]{#2}

\end{document}